\documentclass[11pt,reqno]{amsart}
\usepackage{fullpage}
\usepackage{times}
\usepackage[T1]{fontenc}
\usepackage[utf8]{inputenc}
\usepackage{amsmath,amsfonts,amssymb,amsxtra,setspace,xspace,graphicx,lmodern,fourier,mathrsfs}
\usepackage[colorlinks=true]{hyperref}
\hypersetup{urlcolor=blue, citecolor=red, linkcolor=blue}
\usepackage[usenames,dvipsnames]{color}
\usepackage[misc]{ifsym}
\usepackage{cleveref}

\newcommand{\R}{{\mathbb R}}

\newcommand{\T}{{\mathbb T}}

\newcommand{\be}[1]{\begin{equation*}\label{#1}}

\newcommand{\ee}{\end{equation*}}

\newtheorem{thm}{Theorem}[section]
\newtheorem{cor}[thm]{Corollary}
\newtheorem{lem}[thm]{Lemma}
\newtheorem{rmk}[thm]{Remark}

{ \theoremstyle{remark} }

\begin{document}

\title[Cutoff Boltzmann equation with polynomial perturbation near Maxwellian]
{Cutoff Boltzmann equation with polynomial perturbation  near Maxwellian
}

\author[C. Cao]{Chuqi Cao}
\thanks{C. Cao: Yau Mathematical Science Center and  Beijing Institute of Mathematical Sciences and Applications, Tsinghua University, Beijing, 100084, P. R. China. \\
Email address: chuqicao@gmail.com }

\begin{abstract}
In this paper, we consider the cutoff  Boltzmann equation near Maxwellian, we proved the global existence and uniqueness for the cutoff Boltzmann equation in polynomial weighted space for all $\gamma \in (-3, 1]$. We also proved initially polynomial decay for the large velocity in $L^2$ space will induce polynomial decay rate, while initially exponential decay will induce exponential rate for the convergence. Our proof is based on newly  established  inequalities for the cutoff Boltzmann equation and semigroup techniques. Moreover, by generalizing  the $L_x^\infty L^1_v \cap L^\infty_{x, v}$ approach, we prove the global existence and uniqueness  of a mild solution to the Boltzmann equation with bounded polynomial weighted $L^\infty_{x, v}$ norm under some small condition on the initial $L^1_x L^\infty_v$ norm and entropy so that this initial data allows large amplitude oscillations. \\

\end{abstract}

\maketitle

\setcounter{tocdepth}{1}
\tableofcontents

\section{Introduction}\label{sec1}

The Boltzmann equation reads
\begin{equation}\label{Boltzmann equation}
\partial_t F +v \cdot \nabla_x F =Q(F, F), \quad F(0, x, v) =F_0(x, v),
\end{equation}
where $F(t, x, v) \ge 0$ is a distributional functions of colliding particles which, at time $t>0$ and position $x \in \T^3$, move with velocity $v \in \R^3$. We remark that the Boltzmann equation is one of the fundamental equations of mathematical physics and is a cornerstone of statistical physics. The Boltzmann collision operator $Q$ is a bilinear operator which acts only on the velocity variable $v$, that is
\[
 Q(G,F)(v)=\int_{\R^3}\int_{\mathbb{S}^2}B(v-v_*,\sigma )(G'_*F'-G_*F)d\sigma dv_*.
\]
Let us give some explanations on the collision operator.
\begin{enumerate}
\item  We use the standard shorthand $F=F(v),G_*=G(v_*),F'=F(v'),G'_*=G(v'_*)$, where $v',v'_*$ are given by
\[
v'=\frac{v+v_*}{2}+\frac{|v-v_*|}{2}\sigma,\quad v_*'=\frac{v+v_*}{2}-\frac{|v-v_*|}{2}\sigma,\quad \sigma \in\mathbb{S}^2.
\]
 This representation follows the parametrization of set of solutions of the physical law of elastic collision:
\[
  v+v_*=v'+v'_*,  \quad |v|^2+|v_*|^2=|v'|^2+|v'_*|^2.
\]

\item  The nonnegative function $B(v-v_*,\sigma)$ in the collision operator is called the Boltzmann collision kernel. It is always assumed to depend only on $|v-v_*|$ and the deviation angle $\theta$ through $\cos\theta  := \frac{v-v_*}{|v-v_*|}\cdot\sigma$.
\item  In the present work,  our {\bf basic assumptions on the kernel $B$}  can be concluded as follows:
\begin{itemize}
  \item[$\mathbf{(A1).}$] The Boltzmann kernel $B$ takes the form: $B(v-v_*,\sigma)=|v-v_*|^\gamma b(\frac{v-v_*}{|v-v_*|}\cdot\sigma)$, where   $b$ is a nonnegative function.

\item[$\mathbf{(A2).}$] The angular function $b(\cos \theta)$  satisfies the Grad's cutoff assumption
\[
  K  \leq  b(\cos\theta)\leq   K^{-1}, \quad K>0.
\]

  \item[$\mathbf{(A3).}$]
  The parameter $\gamma$ satisfies the condition $-3<\gamma  \le 1$.

  \item[$\mathbf{(A4).}$]  Without lose of generality, we may assume that $B(v-v_*,\sigma)$ is supported in the set $0\le \theta \le \pi/2$, i.e.$\frac{v-v_*}{|v-v_*|}\cdot\sigma \ge 0$, otherwise $B$ can be replaced by its symmetrized form:
\[
\overline{B}(v-v_*,\sigma )=|v-v_*|^\gamma\big(b(\frac{v-v_*}{|v-v_*|}\cdot\sigma )+b(\frac{v-v_*}{|v-v_*|}\cdot(-\sigma))\big) \mathrm{1}_{\frac{v-v_*}{|v-v_*|}\cdot\sigma \ge0},
\]
where $\mathrm{1}_A$ is the characteristic function of the set $A$.
\end{itemize}
  \end{enumerate}
\begin{rmk}  Generally, the case $\gamma > 0$,  $\gamma = 0$, and  $\gamma < 0$ correspond to so-called hard, Maxwellian, and soft potentials respectively.
\end{rmk}

\subsection{Basic properties and the perturbation equation} We recall some basic facts on the Boltzmann equation.
 \smallskip

\noindent$\bullet$ {\bf Conservation Law.}  Formally if $F$ is the solution to the Boltzmann equation \eqref{Boltzmann equation} with initial data $F_0$, then it enjoys the conservation of mass, momentum and the energy, that is,
\begin{equation}
\label{conservation law}
\frac {d}{dt}\int_{\T^3\times{\R}^3} F(t, x, v)\varphi(v)dvdx= 0,\quad   \varphi(v)=1, v, |v|^2.
\end{equation}
For simplicity, we introduce   the normalization identities on the initial data $F_0$ which satisfies
\[
\int_{\T^3\times\R^3} F_0(x, v)\, dvdx= 1, \quad \int_{\T^3\times\R^3} F_0(x, v)\, v\, dvdx =0,
\quad \int_{\T^3\times\R^3} F_0(x, v)\, |v|^2 \, dvdx =3.
\]
This means that the  equilibrium associated to \eqref{Boltzmann equation} will be the standard Gaussian function, i.e.
\[
\mu(v) := (2\pi)^{-3/2} e^{-|v|^2/2}, 
\] 
which enjoys the same mass, momentum and energy as $F_0$. By the Boltzmann H-Theorem, the solution to the Boltzmann equation \eqref{Boltzmann equation} satisfies
\begin{equation}
\label{entropy inequality}
\int_{\T^3} \int_{\R^3} F(t, x, v) \ln F(t, x, v) dv dx \le \int_{\T^3} \int_{\R^3} F_0 (x, v) \ln F_0 (x, v) dv dx,\quad \forall t \ge 0.
\end{equation}
\smallskip

\noindent$\bullet$ {\bf Perturbation Equation.} In the perturbation framework, let $f$ be the perturbation such that
\[
F=\mu+f.
\]
The Boltzmann equation (\ref{Boltzmann equation}) becomes
\begin{equation}
\label{perturbation equation}
\partial_t f +v \cdot \nabla_x f= Q(\mu, f)+Q(f,\mu)+Q(f, f):=Lf +Q(f, f),
\end{equation}
with the linearized operator $L=Q(\mu,\cdot)+Q(\cdot,\mu)$.

\subsection{Brief review of previous results} In what follows we recall some known results on the Landau and Boltzmann equations with a focus on the topics under consideration in this paper, particularly on global existence and large-time behavior of solutions to the spatially inhomogeneous equations in the perturbation framework. For global solutions to the renormalized equation with large initial data, we mention the classical works \cite{DL, DL2, L, V2, V3, DV, AV}. For the stability of vacuum, see \cite{L2, G5, C2}  for the Landau, cutoff  and non-cutoff Boltzmann equation with moderate soft potential respectively.

We focus on the results in the perturbation framework. In the near Maxwellian framework, a key point is to characterize the dissipation property in the $L^2$ norm in $v$ for the linearized operator and further control the trilinear term in an appropriate way. More precisely, for the cutoff Boltzmann equation,  the corresponding linearized operator $L_\mu$ is self-adjoint and has the null space
\[
L_{\mu} f  = \frac {1} {\sqrt{\mu}}  (Q(\mu, \sqrt{\mu } f) + Q(\sqrt{\mu} f, \mu)  ),\quad \mbox{Null} (L_{\mu}  ) =\mbox{Span} \{\sqrt{\mu} ,\sqrt{\mu} v , \sqrt{\mu} |v|^2 \},
\]
and having the coercivity property
\[
 \quad (f, L_{\mu} f) \le -C  \Vert f\Vert^2_{L^2_{\gamma/2}}, \quad \forall f \in \hbox{Null} (L_{\mu}  )^\perp,
\]
for some constant $C>0$, such coercivity property is essential in the near Maxwellian framework.

In the near Maxwellian framework, global existence and large-time behavior of solutions to the spatially inhomogeneous equations is proved   in \cite{G2, G3, SG, SG2} for the cutoff Boltzmann equation and  in \cite{G} for the Landau equation.  For the non-cutoff Boltzmann equation it is proved in \cite{AMUXY, AMUXY2, AMUXY3, AMUXY4, GS, GS2}, see also \cite{DLSS} for a recent work. We also refer to \cite{G6, G7, G8, DL3, DLYZ, DLYZ2} for  the Vlasov-Poisson/Maxwell-Boltzmann/Landau equation near Maxwellian.   We remark here all these works above are based on the following decomposition 
\[
\partial_t f  + v \cdot \nabla_x f =L_{\mu} f + \Gamma(f, f), \quad L_{\mu} f  = \frac {1} {\sqrt{\mu}}  (Q(\mu, \sqrt{\mu } f) + Q(\sqrt{\mu} f, \mu)  ) ,\quad \Gamma(f, f) = \frac 1 { \sqrt{\mu}} Q(\sqrt{\mu} f, \sqrt{\mu}  f), 
\]
which means the result are in $\mu^{-1/2}$ weighted space.

In the near Maxwellian framework,  there are also several results for $L^\infty_{x, v}$ well-posedness results near Maxwellian. For the cutoff Boltzmann equation near Maxwellian, a $L^2-L^\infty$ approach has been introduced in  \cite{UY, G4} and apply to various contexts, see  \cite{K, GKTT} and the reference therein for example.  Global $L^\infty_{x, v}$ well-posedness result near Maxwellian is proved for the Landau equation in \cite{KGH} and for the non-cutoff Boltzmann equation in \cite{AMSY2, SS}. 
The solution with large amplitude initial data  is first proved in \cite{DHWY}  under the assumption of small  entropy, we also refer to \cite{DW} for the Boltzmann equation with large amplitude initial data in bounded domains and \cite{W} for the relativistic Boltzmann equation with large amplitude initial data.

For inhomogeneous equations  with  polynomial weighted perturbation near Maxwellian, in Gualdani-Mischler-Mouhot \cite{GMM} the authors first  prove the global existence and large-time behavior of solutions with polynomial velocity weight for the cutoff Boltzmann equation with hard potential in $L^1_vL^\infty_x(1+|v|^k), k>2$, this method is generalized to the Landau equation in \cite{CTW, CM}. The non-cutoff Boltzmann equation with hard potential is proved in \cite{HTT, AMSY}, the  soft potential case is proved in \cite{CHJ}. The cutoff Boltzmann equation with soft potential is proved in this paper.

\subsection{ Main results and notations}  Let us first introduce the function spaces and notations.

  \noindent $\bullet$  For any $p \in [1, +\infty)$, $q \in \R$ the $L^p_{q}$ norm is defined by
\[
\| f \|_{L^p_{q}}^p :  = \int_{\R^3} |f(v)|^p   \langle v \rangle^{pq} dv,
\]
where the Japanese bracket $\langle v \rangle$ is defined as $\langle v \rangle := (1+|v|^2)^{1/2}$. 

\noindent $\bullet$  For any $p \in [1, +\infty)$, $ k \in \R, a >0, b \in (0, 2)$ the  $L^p_{k, a, b}$ norm is defined by
\[
\| f  \|_{L^p_{k, a, b}}^p = \int_{\R^3} |f(v)|^p   \langle v \rangle^{pq} e^{p a \langle v \rangle^b} dv.
\]

\noindent $\bullet$  For real numbers $m, l$, we define the weighted Sobolev space $H_l^m$ by
\[
H^m_l:= \{ f(v) |  |f|_{ H^m_l}=|\langle \cdot\rangle^l \langle D\rangle^mf|_{L^2}< +\infty \},
\]
where  $a(D)$ is a pseudo-differential operator with the symbol $a(\xi)$ and it is defined as
\[
(a(D)f)(v):= \frac{1}{(2\pi)^3}  \int_{\R^3}\int_{\R^3}  e^{i(v-u)\xi}  a(\xi)  f(u)  du d\xi,
\]
and we denote $H^m := H^m_0$. The weighted Sobolev space $H^m_{k, a,  b}$ can be defined in a similar way.

\noindent $\bullet$ For function $f(x, v),x\in\T^3,v\in\R^3$, the norm $\|\cdot\|_{H^\alpha_xH^m_l}$ is defined as
\[
\|f\|_{H^\alpha_xH^m_l}:=\left(\int_{\T^3}\|\langle D_x\rangle^\alpha f(x,\cdot)\|^2_{H^m_l}dx\right)^{1/2}.
\]
If $\alpha=0$, $H^0_xH^m_l=L^2_xH^m_l$. The weighted Sobolev space $H^{\alpha}_x  H^m_{k, a,  b}$ can be defined in a similar way.

\noindent $\bullet$ We write $a \lesssim b$ indicate that there is a uniform constant $C$, which may be different on different lines, such that $a\le C b $. We use the notation $a\sim b$ whenever $a\lesssim b$ and $b\lesssim a$.  We denote $C_{a_1,a_2, \cdots, a_n}$ by a constant depending on parameters $a_1,a_2,\cdots, a_n$. Moreover, we use parameter $\epsilon$ to represent different positive numbers much less than 1 and determined in different cases.

\noindent $\bullet$ We use $(f, g)$ to denote the inner product of $f, g$ in the $v$ variable $(f, g)_{L^2_v}$ for short, we use $(f, g)_{L^2_k}$ to denote $(f, g  \langle v \rangle^{2k})_{L^2_v}$.

\noindent $\bullet$  For any function $f$ we define
\[
\|f(\theta)\|_{L^1_\theta} : =\int_{\mathbb{S}^2}f(\theta)d\sigma = 2\pi\int_0^{\pi}f(\theta)\sin\theta d\theta.
\]

\noindent $\bullet$ Gamma function and Beta function are defined by
\[
 \Gamma(x):=\int_0^\infty t^{x-1}e^{-t}dt, \quad  x>0, \quad B(p, q):=\int_0^1t^{p-1}(1-t)^{q-1}dt = \int_0^\infty \frac {t^{p-1}} {(1+t)^{p + q} }, \quad  p, q>0.
\]
We recall that Beta and Gamma functions fulfill the following properties:
\begin{equation}
\label{beta function}
B(p, q)=\frac{\Gamma (p)\Gamma (q)}{\Gamma (p+q)}, \quad B(p, q) \sim q^{-p}  \quad \hbox{if} \quad 0 < p \le 2.  
\end{equation}

\noindent $\bullet$  For the cross section $B(\cos \theta, |v-v_*|)$ with an angular cutoff, we will use the notation $Q^{\pm}$ defined as
\begin{equation}
\label{Q+-}
Q^+(f, g) = \int_{\R^3} \int_{\mathbb{S}^2} B(\cos \theta, |v_*-v|) f_*' g' dv_* d\sigma, \quad Q^-(f, g) = \int_{\R^3} \int_{\mathbb{S}^2} B(\cos \theta, |v-v_*|) f_* g dv_* d\sigma. 
\end{equation}
Note that $Q(f, g) = Q^+(f, g) -Q^-(f, g)$.

\noindent $\bullet$ For the linearized operator $L$ we have
\[
\ker(L) = \text{span} \{ \mu, v_1\mu, v_2 \mu, v_3 \mu, |v|^2 \mu \}.
\]
We define the projection onto $\ker(L)$ by
\begin{equation}
\label{projection}
P f : = \left(   \int_{\T^3} \int_{\R^3} f dv dx \right) \mu + \sum_{i=1}^3 \left(   \int_{\T^3}   \int_{\R^3} v_i f dv dx \right) v_i \mu + \left(    \int_{\T^3}  \int_{\R^3} \frac {|v|^2-3}{\sqrt{6} } f dv  dx \right) \frac  {|v|^2 -3} {\sqrt{6}} \mu.
\end{equation}

\noindent $\bullet$  For any $ k \in  \R, \gamma \in (-3, 1]$, we define
\[
\Vert f \Vert_{L^2_{k+\gamma/2, *}}^2 :=  \int_{\R^3} \int_{\R^3} \mu (v_*) |v-v_*|^\gamma |f(v)|^2 \langle v \rangle^{2k} dv dv_*.
\]
It is easily seen that $\Vert f \Vert_{L^2_{k+\gamma/2, *}} \sim \Vert f \Vert_{L^2_{k+\gamma/2}}$.

\noindent $\bullet$  During the whole paper, we will denote $N$ by 
\begin{equation}
\label{N}
N=\left\{
\begin{aligned}
&2,& \quad \gamma \in (-\frac 3 2 ,1],
\\ 
&3,& \quad \gamma \in (-\frac 5 2 ,-\frac 3 2] ,
\\ 
&4,& \quad \gamma \in (-3, -\frac 5 2].
\end{aligned}
\right.
\end{equation}

\noindent $\bullet$ We let the multi-indices $\alpha$ and $\beta$ be $\alpha=[\alpha_1,\alpha_2,\alpha_3]$, $\beta=[\beta_1,\beta_2,\beta_3]$ and define
\[
\partial^\alpha_{\beta}:=\partial^{\alpha_1}_{x_1}\partial^{\alpha_2}_{x_2}\partial^{\alpha_3}_{x_3}\partial^{\beta_1}_{v_1}\partial^{\beta_2}_{v_2}\partial^{\beta_3}_{v_3}.
\]

\noindent $\bullet$ 
If each component of $\theta$ is not greater than that of the $\overline\theta$'s, we denote by $\theta\le\overline\theta$;
$\theta<\overline\theta$ means $\theta\le\overline\theta$, and $|\theta|<|\overline\theta|$.

\noindent $\bullet$  For any $k \in \R, a>0, b \in (0, 2)$,  if $\gamma \in (-\frac 3 2, 1]$ we will denote
\begin{equation}
\label{X k 1}
\Vert f \Vert_{X_k}^2 := \sum_{|\alpha|\le 2} \Vert \langle v \rangle^{k} \partial^\alpha f\Vert_{L^2_{x, v}}^2, \quad 
\Vert f \Vert_{X_{k, a, b}}^2 := \sum_{|\alpha| \le 2} \Vert \partial^{\alpha} f \langle v \rangle^{k} e^{a \langle v \rangle^b}  \Vert_{L^2_{x, v} }^2.
\end{equation}
If $-3 <\gamma \le -\frac 3 2$ we will denote the weight function $w(\alpha, \beta) $ as
\begin{equation}
\label{weight function}
w(|\alpha|, |\beta|)  = \langle v \rangle^{k-  a|\alpha|  - b|\beta| + c}  , \quad b = 7\max \{-\gamma, 0 \}, \quad  a = b+ \min \{\gamma, 0\} = 6\max \{-\gamma, 0 \} , \quad c =  4 b.
\end{equation}
Note that $w(\alpha, \beta) =w(|\alpha|, |\beta|)$, the two notations will have the same meaning in the whole paper. We then define
\begin{equation}
\label{X k 2}
\Vert f \Vert_{X_k}^2 := \sum_{|\alpha| + |\beta| \le N} C^2_{|\alpha|, |\beta| }\Vert w(\alpha, \beta)\partial^\alpha_{\beta} f\Vert_{L^2_{x, v}}^2, \quad  \Vert f \Vert_{X_{k, a, b}}^2 := \sum_{|\alpha| + |\beta| \le  N} C_{|\alpha|, |\beta|}^2  \Vert \partial^{\alpha}_{\beta} f w (\alpha, \beta) e^{a \langle v \rangle^b} \Vert_{L^2_{x, v}}^2,
\end{equation}
where  the constant $C_{|\alpha|, |\beta|}$ satisfies
\begin{equation}
\label{constant}
C_{|\alpha|, |\beta| } \ll C_{|\alpha|, |\beta_1|}, \quad \forall |\alpha| \ge 0, \quad 0  \le |\beta_1| < |\beta|, \quad C_{|\alpha| +1,| |\beta|-1 } \gg C_{|\alpha|, |\beta|}, \quad \forall |\alpha| \ge 0, \quad |\beta| \ge 1,
\end{equation}
and we will denote $Y_{k} :=X_{k+\gamma/2}, Y_{k, a, b} :=X_{k+\gamma/2, a, b}$. We also define 
\begin{equation}
\label{Y k *}
\Vert f \Vert_{Y_{k, *}}^2 := \sum_{|\alpha|\le 2} \Vert \partial^\alpha f\Vert_{L^2_{x} L^2_{k+\gamma/2, *}}^2, \quad \gamma \in (-\frac 3 2, 1],\quad \Vert f \Vert_{Y_{k, *}}^2 := \sum_{|\alpha| + |\beta| \le N} C^2_{|\alpha|, |\beta| }\Vert w(\alpha, \beta)\partial^\alpha_{\beta} f\Vert_{L^2_{x} L^2_{\gamma/2, *}}^2,  \quad \gamma \in (-3, -\frac 3 2].
\end{equation}

\noindent $\bullet$  Define $\bar {X_0} : = H^2_xL^2_v$ if $\gamma \in (-\frac 3 2, 1]$, $\bar {X_0} : = H^N_{x, v}$ if $\gamma \in (-3, -\frac 3 2 ]$.

\noindent$\bullet$ Define the relative entropy by 
\begin{equation}
\label{entropy}
H (F(t) ) := \int_{\T^3} \int_{\R^3}  F(t, x ,v ) \ln F(t, x ,v) -\mu \ln \mu d  v d x,
\end{equation}
it is easily seen that $H (F(t)) \ge 0, \forall t \ge 0, H(\mu) =0$. 

%

\medskip

\subsection{Main results} We may now state our main results.

\begin{thm}\label{T12} Consider the  Cauchy problem
\[
\partial_t f +v\cdot\nabla_x f= L f +Q(f, f), \quad \mu+f\ge 0,\quad f(0) =f_0, \quad P  f_0 = 0.
\]
\begin{itemize}
  \item  Polynomial case: For any $k \ge 6$, there exists a small constant $\epsilon_0 >0$ such that for any initial data $f_0 \in X_k$ satisfies
\[
\mu +f_0 \ge 0, \quad  Pf_0 =0, \quad  \Vert f_0 \Vert_{X_6} < \epsilon_0, \quad \Vert f_0 \Vert_{X_k}< +\infty,
\]
there exists a unique global solution $f \in L^\infty( [0, +\infty), X_k  )$ satisfies $F = \mu +f \ge 0$. Moreover if $\gamma \in [0,1]$,  we have
\[
\Vert f(t) \Vert_{X_k} \lesssim e^{-\lambda t} \Vert f_0 \Vert_{X_{k}}, \quad \forall  t \in (0, +\infty),
\]
for some constant $\lambda>0$.  If $\gamma <0$, then for any $ 6 \le k_1 < k$  we have
\[
\Vert f(t) \Vert_{X_{k_1}} \lesssim t^{-\frac {|k- k_*|} {|\gamma|} } \Vert f_0 \Vert_{X_{k}} ,\quad \forall k_* \in (k_1, k), \quad \forall  t \in (0, +\infty).
\]
\item Exponential case: For any $k \in \R, a > 0, b \in (0, 2)$,  there exists a small constant $\epsilon_0 >0$ such that for any $f_0 \in X_{k, a, b}$ satisfies
\[
\mu +f_0 \ge 0, \quad  Pf_0 =0, \quad  \Vert f_0 \Vert_{X_{k, a, b}} < \epsilon_0,
\]
there exists a unique global solution  $f \in L^\infty( [0, +\infty), X_{k, a, b}  )$ satisfies $F = \mu +f \ge 0$. Moreover if $\gamma \in [0,1]$,  we have
\[
\Vert f(t) \Vert_{X_{k, a, b}} \lesssim e^{-\lambda t}\Vert f_0 \Vert_{X_{k, a, b}}, \quad \forall  t \in (0, +\infty),
\]
for some constant $\lambda>0$.  If $\gamma <0$, then for any $ 0 < a_0 < a$  we have
\[
\Vert f(t) \Vert_{X_{k, a_0, b}} \lesssim e^{-\lambda t^{\frac {b} {b - \gamma}} }\Vert f_0 \Vert_{X_{k, a, b}}, \quad \forall  t \in (0, +\infty),
\]
for some constant $\lambda>0$. 
\end{itemize}
\end{thm}

Several comments on the results are in order:

\noindent$\bullet$ The choice of $N$ is optimal in the sense that $N -2 $ is the minimal integer such that 
\[
\Vert Q(f, f) \Vert_{H^{N-2}_v} \le C\Vert f \Vert_{H^{N-2}_4}^2.
\]

\noindent$\bullet$ We emphasize that for global solutions in the polynomial weight space, we only need the small initial value of a given norm, more precisely we only require  the smallness of $X_6$  instead of $X_k$.

%

\medskip

By assuming the entropy is small we can prove the well-posedness for the Boltzmann equation with large amplitude initial data.

\begin{thm}\label{T13} Consider the  Cauchy problem
\[
\partial_t f +v\cdot\nabla_x f= L f +Q(f, f), \quad \mu+f_0 \ge 0,\quad f(0) =f_0, \quad P  f_0 = 0.
\] 
For any $\gamma \in (-3, 1]$, there exists a constant  $k_0 \ge 8$ such that for all $k \ge k_0$, for any $\beta \ge \max \{ 3, 3+\gamma \}$, for any fixed constant $M >1$, there exists a constant $\epsilon_0>0$ depends on $M, k , \beta$ such that if 
\begin{equation}
\label{smallness condition entropy}
\Vert \langle v \rangle^{k +\beta} f_0 \Vert_{L^\infty_{x, v}} \le M,\quad H (F_0) + \Vert \langle v \rangle^{k} f_0 \Vert_{L^1_xL_v^\infty} \le \epsilon_0,
\end{equation}
then the Boltzmann solution has a unique global mild solution $f \in L^\infty( [0, +\infty), L^\infty_{x, v} (\langle v \rangle^{k+\beta})  )$  satisfies that
\[
\Vert \langle v \rangle^{k +\beta}  f  (t)  \Vert_{L^\infty_{x, v}}\le CM^2, \quad \forall t \ge 0,
\]
where $C >0$ depends on $\gamma, k, \beta$. Moreover, for the case $\gamma \ge 0$ we have
\[
\Vert \langle v \rangle^{k + \beta} f(t) \Vert_{L^\infty_{x, v}}  \le Ce^{-\lambda t}, \quad \forall t \ge 0,
\]
for some constants $C, \lambda >0$. For the case $-3< \gamma<0 $, if we further assume $\beta \ge 6$, for any $r \in (0, 1)$ we have
\[
\Vert \langle v \rangle^k f (t) \Vert_{L^\infty_{x, v}}  \le C(1+t)^{-r}, \quad \forall t \ge 0,
\]
for some constants $C>0$.
\end{thm}

\noindent$\bullet$ \underline{\it Comment on the solutions.}  It should be pointed out that initial data satisfying the smallness condition \eqref{smallness condition entropy} are allowed to have large amplitude oscillations in the spatial variable. For instance, one  may take 
\[
F_0(x, v) = \rho_0(x) \mu (v) , \quad x \in \T^3 \times  \R^3,
\]
with $\rho_0(x) \ge 0, \rho_0 \in L^\infty_x, \rho_0 -1 \in L^1_x, \rho_0 \ln \rho_0 -\rho_0 + 1 \in L^1_x$. It is easy to verify that  \eqref{smallness condition entropy} holds if $ \|  \rho_0 \ln \rho_0 -\rho_0 + 1  \|_{L^1_x}  + \|  \rho_0 -1  \|_{L^1_x}$ is small. Even though $ \|  \rho_0 \ln \rho_0 -\rho_0 + 1  \|_{L^1_x}  + \|  \rho_0 -1  \|_{L^1_x}$ is required to be small, initial data are allowed to have large amplitude oscillations. 

%
%

\subsection{Strategies and ideas of the proof} In this subsection, we will explain main strategies and ideas of the proof for our results.

We briefly talk on the semigroup method, this method is first  initiated in \cite{M2} and extended into an abstract setting in a famous work   by Gualdani-Mischler-Mouhot \cite{GMM},  see also its application in kinetic Fokker-Planck equation in \cite{MM}. The main idea for the $L^2$ case can be   expressed briefly as follows: Taking the case $\gamma =0$ in the homogeneous Boltzmann equation for example, first by existing results we have 
\[
(Lf, f)_{L^2 (\mu^{-1/2})} \le -\lambda \Vert f \Vert_{L^2 (\mu^{-1/2})},   \quad \Vert S_L(t) f_0 \Vert_{L^2 (\mu^{-1/2})} \le e^{-\lambda t} \Vert f_0 \Vert_{L^2 (\mu^{-1/2})}^2,
\]
for some constant $\lambda>0$. If we can prove
\begin{equation}\label{L2k}
(Lf, f)_{L^2_k} \le -C \Vert f \Vert_{L^2_k}^2 +  C_k \Vert f \Vert_{L^2}^2,
\end{equation}
then define $A =M\chi_R, B=L-A$, where $\chi \in D(\R)$ a truncation function which satisfies $\mathbb{1}_{[-1,1]} \le \chi \le \mathbb{1}_{[-2,2]}$ and we denote $\chi_a(\cdot) := \chi(\cdot/a)$ for some constant $a > 0$. Taking $M, R$ large we have
\[
(Bf, f)_{L^2_k} \le -C \Vert f \Vert_{L^2_k}^2, \quad \Vert Af \Vert_{L^2(\mu^{-1/2})} \le C \Vert f \Vert_{L^2_k},
\]
which implies 
\[
\Vert S_B(t) f \Vert_{L^2_k} \le e^{-\lambda t} \Vert f_0 \Vert_{L^2_k },
\]
by Duhamel's formula 
\[
\Vert S_L(t) \Vert_{L^2_{k} \to L^2_k }   \le \Vert S_B(t) \Vert_{L^2_k \to L^2_k} + \int_0^t \Vert S_L(s)  \Vert_{L^2 (\mu^{-1/2}) \to L^2 (\mu^{-1/2})} \Vert A \Vert_{L^2_k \to L^2(\mu^{-1/2} ) }\Vert S_B(t-s)\Vert_{L^2_k \to L^2_k} \le C e^{-\lambda t}.
\]
The rate of convergence for the linear operator $L$ is established. Define a scalar product by 
\[
((f, g)) = (f, g)_{L^2_k} + \eta\int_0^{+\infty} (S_L(\tau )f, S_L(\tau) g ) d\tau,
\]
since 
\[
\int_0^{+\infty} (S_L(\tau )f, S_L(\tau) f ) d\tau \le C \Vert f \Vert_{L^2_k}^2 \int_0^{+\infty} e^{-2 \lambda \tau } d\tau \le C  \Vert f \Vert_{L^2_k}^2,
\]
we deduce $(( \cdot , \cdot))$ is an equivalent norm to $L^2_k$. By
\[
\int_0^{+\infty} (S_L(\tau ) L f, S_L(\tau) f ) d\tau = \int_0^{+\infty} \frac d {d\tau} \Vert S_L(\tau) f \Vert_{L^2}^2 d\tau = \lim_{\tau \to \infty} \Vert S_L(\tau) f \Vert_{L^2}^2 - \Vert f \Vert_{L^2}^2 = - \Vert f \Vert_{L^2}^2, 
\]
which implies
\[
((Lf, f)) = (Lf, f)_{L^2_k} + \eta\int_0^\infty (S_L(\tau ) L f, S_L(\tau) f) d\tau = -C_1 \Vert f \Vert_{L^2_k} +(C_2 -\eta) \Vert f \Vert_{L^2} \le -C_1 \Vert f \Vert_{L^2_k}, 
\]
by  choosing  a suitable $\eta$. The estimate for the linearized operator $L$ in this equivalent norm allows us to combine with the nonlinear estimates to conclude the full convergence. In other words, one of the main works of this paper is to prove \eqref{L2k} plus some appropriate upper bounds for the nonlinear operator which consists with  the lower bound in the linearized estimate.

We briefly describe one proof for \eqref{L2k}. We observed and proved,  for any $\gamma \in (-3, 1], k >\max  \{\gamma+3, 3 \}$
\begin{equation}
\label{linearized inequality}
\int_{\R^3}\int_{\mathbb{S}^2} |v-v_*|^\gamma  \frac {\langle v \rangle^k} {\langle v' \rangle^{k}  } e^{-\frac 1 2 | v_*' |^{2} }dv_* d\sigma \le  \frac c {k^{\frac {\gamma+3} 4}} \langle   v \rangle^{\gamma}  +C_k \langle v \rangle^{\gamma-2}, \quad \forall v \in \R^3,
\end{equation}
where the constant $c$ is independent of $k$. The proof of \eqref{linearized inequality} can be seen in Lemma \ref{L62}. By \eqref{linearized inequality}, under the cutoff assumption (A2), for the $Q^+$ term we have
\begin{align}
\label{estimate example Q+}
\nonumber
|(Q^+ (\mu, f), f \langle v \rangle^{2k}) |  \le&  \int_{\R^3} \int_{\R^3}\int_{\mathbb{S}^2}b(\cos \theta)  |v-v_*|^\gamma  |f(v') | e^{-\frac 1 2\langle v_*' \rangle^{2} } |f(v)| \langle v \rangle^{2k} dv_* dv d\sigma
\\ \nonumber
\le & C \left(\int_{\R^3} \int_{\R^3}\int_{\mathbb{S}^2} |v-v_*|^\gamma  |f(v')|^2 \langle v' \rangle^{2k} e^{-\frac 1 2\langle v_*' \rangle^{2} }dv_* dv d\sigma \right)^{1/2}
\\ \nonumber
&\left( \int_{\R^3} \int_{\R^3}\int_{\mathbb{S}^2} |v-v_*|^\gamma \frac {\langle v \rangle^{2k}} {\langle v' \rangle^{2k}  } e^{-\frac 1 2\langle v_*' \rangle^{2} }  |f(v)|^2 \langle v \rangle^{2k} dv_* dv d\sigma \right)^{1/2}
\\ 
\le& \frac {c_2} {k^{\frac {\gamma+3} 4}}  \Vert f \Vert_{L^2_{k+\gamma/2}}^2 + C_k  \Vert f\Vert_{L^2_{k+\gamma/2-1}}^2, 
\end{align}
for some constant $c_2>0$ independent of $k$, the $(Q^+( f, \mu), f \langle v \rangle^{2k} )$ term can be estimated by the same way. For the $Q^-$ term we easily compute
\[
-Q^-(f, \mu) - Q^{-}(\mu, f) =  \int_{\R^3} \int_{\mathbb{S}^2} |v-v_*|^\gamma  b(\cos \theta)( \mu (v_*) f(v) + \mu(v) f(v_* )  ) dv_* d\sigma  \le -  c_1 \langle v \rangle^\gamma f(v) + C\mu(v) \Vert f \Vert_{L^1_4}, 
\]
so we have
\[
-(Q^-(f, \mu) +  Q^{-}(\mu, f) , f \langle v \rangle^{2k} ) \le - c_1 \Vert f \Vert_{L^2_{k+\gamma/2}}^2 + C_k  \Vert f\Vert_{L^2_{6}}^2 . 
\]
Gathering the two terms we have
\[
(Lf, f \langle v \rangle^{2k}) \le -(c_1 -\frac {2c_2} {k^{\frac {\gamma+3} 4}} )\Vert f \Vert_{L^2_{k+\gamma/2}}^2 + C_k  \Vert f\Vert_{L^2_{k+\gamma/2-1}}^2.
\]
If  $k$ is larger than some constant $k_0$ such that $2c_2  < c_1{k_0^{\frac {\gamma+3} 4}}$, \eqref{L2k}  is thus proved. To get an optimal $k$ for the polynomial weight case,  we will use the pre-post collisional change of variables to give more precise computation in this paper.

We remark here that \eqref{linearized inequality} also serves an important point in the proof for Boltzmann equation with large amplitude initial data. In \cite{DHWY}, the following inequality
\[
\int_{\R^3}\int_{\mathbb{S}^2} |v-v_*|^\gamma  \frac {e^{-\frac 1 4 | v |^{2} }} {e^{-\frac 1 4 | v' |^{2} }  } e^{-\frac 1 2 | v_*' |^{2} }dv_* d\sigma \le  C_k \langle v \rangle^{\gamma-2},
\] 
plays a key role in the proof for Boltzmann equation with large amplitude initial data for the $\mu^{-1/2}$ case.  The inequality \eqref{linearized inequality} can be seen as its polynomial version.

For the upper bound, we observed and proved such fact that
\begin{equation}
\label{nonlinear upper bound}
\int_{\R^3}\int_{\mathbb{S}^2} |v-v_*|^\gamma  \frac {\langle v \rangle^k} {\langle v' \rangle^{k}  \langle v_*' \rangle^{k} } dv_* d\sigma \le  C_k \langle   v \rangle^{\gamma}, \quad \forall v \in \R^3,
\end{equation}
which plays an essential role in the proof of the upper bound. Compared to the $\mu^{-1/2}$ case, if we replace $\langle v \rangle^k $   by $\mu^{-\frac  1 2} = e^{\frac 1 4 |v|^2}$ in  \eqref{nonlinear upper bound}, we easily seen that
\[
\int_{\R^3}\int_{\mathbb{S}^2} |v-v_*|^\gamma  \frac {  e^{\frac 1 4 |v|^2}  } {  e^{\frac 1 4 |v'|^2}    e^{\frac 1 4 |v_*'|^2} } dv_* d\sigma = \int_{\R^3} \int_{\mathbb{S}^2} |v-v_*|^\gamma e^{-\frac 1 4 |v_*|^2}  dv_* d\sigma \le C \langle v \rangle^\gamma. 
\]

Now we explain the strategy of the proof for  the Boltzmann equation with  large amplitude initial data. For the polynomial weight case, let $f = \langle v \rangle^{-k} (F- \mu) $ in \eqref{Boltzmann equation}, then $f$ satisfies 
\[
\partial_t f  + v \cdot \nabla_x f   =L_{k} f + \Gamma_k (f, f), 
\]
with
\[
L_k f = \langle v \rangle^k Q( f \langle v \rangle^{ - k}   , \mu  ) + \langle v \rangle^{k}  Q(\mu, f \langle v \rangle^{ - k }   ),  \quad \Gamma_k (f, f) = \langle v \rangle^{k}  Q( f \langle v \rangle^{-k} ,  f \langle v \rangle^{- k }  ). 
\]
For any $k  \ge \max \{3, 3+\gamma \}$, like \cite{DHWY}, instead of estimate the nonlinear term in the following way
\[
|\langle v \rangle^\beta \Gamma_k(f, f)  |\le C \langle v \rangle^\gamma \Vert \langle v \rangle^\beta f \Vert_{L^\infty}^2,
\]
we prove a new estimate 
\[
|\langle v \rangle^\beta \Gamma_k(f, f)  |\le C \langle v \rangle^\gamma \Vert \langle v \rangle^\beta f \Vert_{L^\infty}^{2 -a} \left (\int_{\R^3} |f(t, x, v)|  dv\right)^a,
\]
for some constant $0<a<1$. Second,  we observe that under the condition
\[
H (F_0) + \Vert \langle v \rangle^{k} f_0 \Vert_{L^1_xL_v^\infty} \le \epsilon_0,
\]
we can prove that $\int_{\R^3} |f(t, x, v)|  dv$ will be small after some positive time even if it could be initially large. This observation is the key point to control the nonlinear term  $\Gamma_k(f, f)$, we can finally obtain the uniform  $L^\infty_{x, v}$ estimate under the smallness of $\Vert \langle v \rangle^{k}  f_0 \Vert_{L^1_xL^\infty_v}$ and $H (F_0)$ so that initial data is allowed to have large amplitude oscillations.

\subsection{Remark on the Cutoff assumption (A2)}
Compared to the  cutoff assumption (A2), in many text people define the cutoff assumption in the following way:
\begin{equation}
\label{cutoff assumption 2}
K \le \int_{\mathbb{S}^2} b(\cos\theta) d \theta \le K^{-1}
\end{equation}
which is weaker than assumption (A2). But assumption (A2) plays an important in our proof, for example, if we want to use \eqref{linearized inequality} to prove \eqref{estimate example Q+}, the assumption \eqref{cutoff assumption 2} is not enough, we need to assume that (A2) holds.  

In the proof of Theorem \ref{T12}, we use the prepost-collisional change of variable instead of using inequalities like \eqref{linearized inequality}, thus under assumption \eqref{cutoff assumption 2}, Theorem \ref{T12} is still true,  actually Theorem \ref{T12} can be proved for the Boltzmann equation without cutoff, see \cite{CHJ}.  

But for Theorem \ref{T13} we need to assume that (A2) holds since technically we need to use inequality like \eqref{linearized inequality} to prove it.

 \subsection{Organization of the paper}
Technical tools and lemmas are listed in Section \ref{section2}. Section \ref{section3} is devoted to the upper bounds  and coercivity estimate on collision operator $Q$.  In Section \ref{section4} we will prove estimates for the inhomogeneous equation. We  obtain  global well-posedness and rate of convergence  in $L^2$ in  Section \ref{section5}.  Section \ref{section6}  and Section \ref{section7} are devoted to the proof for the Boltzmann equation with large amplitude initial data. In Section \ref{section6} we prove  global well-posedness for the Boltzmann equation with large amplitude initial data and in Section \ref{section7} we establish rate of convergence to the equilibrium.

\section{Preliminaries}\label{section2}

In later analysis, we often use two types of change of variables below. 
\begin{lem} (\cite{ADVW}) For any smooth function $f$ we have \\
(1) (Regular change of variables)
\[
\int_{\R^3} \int_{\mathbb{S}^2} b(\cos \theta) |v-v_*|^\gamma f(v') d \sigma dv= \int_{\R^3} \int_{\mathbb{S}^2} b(\cos \theta)\frac 1 {\cos^{3+\gamma} (\theta/2)} |v-v_*|^\gamma f(v) d\sigma dv.
\]
(2) (Singular change of variables)
\[
\int_{\R^3} \int_{\mathbb{S}^2} b(\cos \theta) |v-v_*|^\gamma f(v') d \sigma dv_* = \int_{\R^3} \int_{\mathbb{S}^2} b(\cos \theta)\frac 1 {\sin^{3+\gamma} (\theta/2)} |v-v_*|^\gamma f(v_*) d\sigma dv_* .
\]
\end{lem}

\begin{lem}\label{L22}For any smooth function $f, g, h, b$, for any constant $\gamma \in\R$, we have
\begin{equation*}
\begin{aligned}
&\left( \int_{\R^3}\int_{\R^3} \int_{\mathbb{S}^2}b( \cos \theta) |v-v_*|^\gamma f_* g h' dv dv_* d\sigma  \right)^2
\\
\le & \left( \int_{\mathbb{S}^2} b(\cos \theta) \sin^{-\frac {3+\gamma} 2} \frac \theta 2   d\sigma\right)^{2} \int_{\R^3}\int_{\R^3}  |v-v_*|^\gamma |f_*|^2 |g| dv dv_*   \int_{\R^3}\int_{\R^3}  |v-v'|^\gamma |g| |h'|^2dv dv'.
\end{aligned}
\end{equation*}
Similarly
\begin{equation*}
\begin{aligned}
& \left(  \int_{\R^3}\int_{\R^3} \int_{\mathbb{S}^2}b(  \cos \theta) |v-v_*|^\gamma f_* g h' dv dv_* d\sigma \right)^2
\\
\le & \left( \int_{\mathbb{S}^2} b(\cos \theta) \cos^{-\frac {3+\gamma} 2} \frac \theta 2   d\sigma \right)^{2}  \int_{\R^3}\int_{\R^3}  |v-v_*|^\gamma |f_*| |g|^2 dv dv_* \int_{\R^3}\int_{\R^3}  |v_*-v'|^\gamma |f_*| |h'|^2dv_* dv'.
\end{aligned}
\end{equation*}
\end{lem}
\begin{proof}
For the first inequality, by Cauchy-Schwarz inequality and singular change of variables
\begin{equation*}
\begin{aligned} 
&\left(\int_{\R^3}\int_{\R^3} \int_{\mathbb{S}^2}b(\cos \theta) |v-v_*|^\gamma f_* g h' dv dv_* d\sigma  \right)^2
\\
\le  &\int_{\R^3}\int_{\R^3}  \int_{\mathbb{S}^2} b(\cos \theta) \sin^{-\frac {3+\gamma} 2} \frac \theta 2  |v-v_*|^\gamma |f_*|^2 |g| dv dv_*  d\sigma
\int_{\R^3}\int_{\R^3}  \int_{\mathbb{S}^2} b(\cos \theta) \sin^{\frac {3+\gamma} 2} \frac \theta 2 |v-v_*|^\gamma | |g| |h'|^2 dv dv_*  d\sigma 
\\
\le& \left(\int_{\mathbb{S}^2} b(\cos \theta) \sin^{-\frac {3+\gamma} 2} \frac \theta 2 d\sigma  \right)^2\int_{\R^3}\int_{\R^3}  |v-v_*|^\gamma |f_*|^2 |g| dv dv_* \int_{\R^3}\int_{\R^3}  |v-v'|^\gamma |g| |h'|^2dv dv',
\end{aligned}
\end{equation*}
which finishes the proof. The second inequality can be proved similarly by  regular change of variables.
\end{proof}

\begin{lem}\label{L23}(\cite{V}, Section 1.4)(Pre-post collisional change of variables)
For smooth function $F$ we have
\[
\int_{\R^3}\int_{\R^3} \int_{\mathbb{S}^2 }F(v, v_*, v', v_*') B(|v-v_*|, \cos \theta) dv dv_* d \sigma =\int_{\R^3}\int_{\R^3} \int_{\mathbb{S}^2 }F(v', v_*', v, v_*) B(|v-v_*|, \cos \theta) dv dv_* d \sigma .
\]
\end{lem}

\begin{lem} (Hardy-Littlewood-Sobolev inequality) (\cite{LL}, Chapter 4)
Let $1< p, r < +\infty$ and $0<\lambda<d$ with $1/p+\lambda/d+1/r = 2$.  Then there exists a constant $C(n,\lambda, p)$, such that for all smooth function $f$, $h$ we have
\[
\int_{\R^d} \int_{\R^d} f(x) |x-y|^{-\lambda}  h(y) dx dy \le C(n,\lambda, p) \Vert f \Vert_{L^p} \Vert h \Vert_{L^r} .
\]
\end{lem}

The following can be seen as a weak version of Hardy-Littlewood-Sobolev inequality when $q = +\infty$  which is useful in the following proof. 

\begin{lem}\label{L25} Suppose $\gamma \in (-3, 0)$, then for any smooth function $f$, the following estimate holds: \\ 
If $\gamma \in (-\frac 3 2, 0)$, we have
\[
\sup_{v \in \R^3}\int_{\R^3} |v-v_*|^{\gamma} |f| (v_*) dv_* \lesssim \Vert f \Vert_{L^1}^{1 + \frac { 2 \gamma} 3} \Vert f \Vert_{L^2}^{-\frac  {  2 \gamma} 3} .
\]
If $\gamma \in (-2, 0)$, we have
\[
\sup_{v \in \R^3}\int_{\R^3} |v-v_*|^{\gamma } |f| (v_*) dv_* \lesssim \Vert f \Vert_{L^1}^{ 1 + \frac { \gamma } 2} \Vert f \Vert_{L^3}^{- \frac  { \gamma } 2} . 
\]
If $\gamma \in ( -3, 0)$, we have
\[
\sup_{v \in \R^3}\int_{\R^3} |v-v_*|^{\gamma} |f| (v_*) dv_* \lesssim \Vert f \Vert_{L^1}^{1 + \frac \gamma 3} \Vert f \Vert_{L^\infty}^{ - \frac \gamma 3} . 
\]
\end{lem}
\begin{proof} 
Assume $f$ does not equal to $0$ otherwise the estimate is trivial. Let $\lambda > 0$ be a constant to be determined. We divide the integral into two regions $|v-v_*| \le \lambda $ and $|v-v_*| > \lambda $ we have
\[
\int_{\R^3} |v-v_*|^{\gamma } |f| (v_*) dv_*  \le \int_{|v-v_*| \le \lambda} |v-v_*|^{\gamma } |f| (v_*) dv_* +\int_{|v-v_*| > \lambda } |v-v_*|^{\gamma } |f| (v_*) dv_* .
\]
The second part is bounded  by
\[
\int_{|v-v_*| > \lambda } |v-v_*|^{\gamma  } |f| (v_*) dv_*  \lesssim \lambda^{\gamma } \Vert f  \Vert_{L^1}.
\]
For the first part we divided it into three cases,  for $\gamma \in ( -\frac 3 2, 0)$, by Cauchy-Schwarz inequality we have
\[
\int_{|v-v_*| \le \lambda} |v-v_*|^{ \gamma  } |f| (v_*) dv_* \le \left( \int_{|v-v_*| \le \lambda} |v-v_*|^{2 \gamma  } dv_* \right)^{\frac 1 2}\Vert f \Vert_{L^2} \lesssim \lambda^{\frac 3 2 + \gamma } \Vert f \Vert_{L^2}.
\]
For $\gamma \in (-2,  0)$, by H\"older's inequality we have
\[
\int_{|v-v_*| \le \lambda} |v-v_*|^{ \gamma } |f| (v_*) dv_* \le \left( \int_{|v-v_*| \le \lambda} |v-v_*|^{\frac 3 2 \gamma} dv_* \right)^{\frac 2 3}\Vert f \Vert_{L^3} \lesssim \lambda^{2 + \gamma } \Vert f \Vert_{L^3}. 
\]
For $\gamma \in (-3 , 0)$, we have
\[
\int_{|v-v_*| \le \lambda} |v-v_*|^{\gamma } |f| (v_*) dv_* \le \int_{|v-v_*| \le \lambda} |v-v_*|^{ \gamma } dv_* \Vert f \Vert_{L^\infty} \lesssim \lambda^{3 + \gamma } \Vert f \Vert_{L^\infty},
\]
so the proof is ended by taking $\lambda = \Vert f \Vert_{L^1}^{  \frac  2 3} \Vert f \Vert_{L^2}^{ - \frac 2 3},  \Vert f \Vert_{L^1}^{\frac  1 2} \Vert f \Vert_{L^3}^{ - \frac 1 2}, \Vert f \Vert_{L^1}^{\frac  1 3} \Vert f \Vert_{L^\infty}^{-\frac 1 3}$ respectively.
\end{proof}

\begin{lem}\label{L26}
For any smooth function $f, g$, for any $\gamma \in (-2, 0)$ we have
\[
\mathcal{R} : = \int_{\R^3} \int_{\R^3} |v-v_*|^\gamma |f(v_*)|^2  |g(v) |^2 dv dv_* \lesssim \min_{m+ n =1}  \{ \| f \|_{H^m_{\gamma/2} }^2 \|  g \|_{H^n_2 }^2,    \|  f \|_{H^m_2 }^2 \| g \|_{H^n_{\gamma/2} }^2 \}.
\]
Similarly For any $\gamma \in (-3, -2]$ we have
\[
\mathcal{R} : =  \int_{\R^3} \int_{\R^3} |v-v_*|^\gamma |f(v_*)|^2  |g(v) |^2 dv dv_* \lesssim \min_{m + n =2}  \{ \| f \Vert_{H^m_{\gamma/2} }^2 \| g \|_{H^n_2 }^2, \| f \Vert_{H^m_{2} }^2 \| g \|_{H^n_{\gamma/2 } }^2 \}.
\]
\end{lem}

\begin{proof}
Without loss of generality we only prove that
\[
\mathcal{R} \lesssim \min_{m + n =1}  \{ \| f \Vert_{H^m_{\gamma/2} }^2 \| g \|_{H^n_2 }^2\}, \quad \gamma \in (-2, 0) , \quad \mathcal{R} \lesssim \min_{m + n =2}  \{ \| f \Vert_{H^m_{\gamma/2} }^2 \| g \|_{H^n_2 }^2\}, \quad \gamma \in (-3, -2].
\]
For both $\gamma \in (-2, 0)$ and $\gamma \in (-3, -2]$,  by
\[
C\langle v_* \rangle^{ - \gamma }  \le   \langle v- v_* \rangle^{  - \gamma } \langle v \rangle^{-\gamma   }, 
\]
we compute
\begin{equation*}
\begin{aligned}
\mathcal{R}   \lesssim & \int_{\R^3} \int_{\R^3} |v-v_*|^\gamma \langle v- v_* \rangle^{  - \gamma }   |f(v_*)|^2 \langle v_* \rangle^{\gamma}|g(v) |^2 \langle v \rangle^{-\gamma   }  dv dv_* 
\\
\lesssim & \int_{\R^3} \int_{\R^3}  (1+ |v-v_*|^\gamma) |f(v_*)|^2 \langle v_* \rangle^{\gamma}|g(v) |^2 \langle v \rangle^{-\gamma   }  dv dv_*  : = I_1 +I_2.
\end{aligned}
\end{equation*}
For the $I_1$ term, for both $\gamma \in (-2, 0)$ and $\gamma \in (-3, -2]$ we easily compute
\[
I_1 \lesssim \Vert f \Vert_{L^2_{\gamma/2}}^2 \Vert g \Vert_{L^2_{- \gamma/2 }}^2 \lesssim \Vert f \Vert_{L^2_{\gamma/2}}^2 \Vert g \Vert_{L^2_{2}}^2. 
\]
For the $I_2$ term, when $\gamma \in (-2, 0)$ by Lemma \ref{L25} we have
\begin{equation*}
\begin{aligned}
I_2 \lesssim \left(\sup_{v \in \R^3}\int_{\R^3} |v-v_*|^{\gamma} |f (v_*)|^2 \langle v_* \rangle^{\gamma}  dv_*  \right) \|  g  \| _{L^2_{-\gamma/2}}^2 &\lesssim \| f^2 \|_{L^1_{\gamma}}^{1 + \frac {  \gamma} 2} \| f^2 \|_{L^3_{\gamma}}^{-\frac  { \gamma} 2}   \| g \|_{L^2_2}^2  
\\
&\lesssim \| f \|_{L^2_{\gamma/2 }}^{2 +  \gamma }  \| f \|_{L^6_{\gamma/2}}^{  - \gamma }   \| g \|_{L^2_2}^2   \lesssim \| f \|_{H^1_{\gamma/2}}^2 \| g \|_{L^2_2}^2,
\end{aligned}
\end{equation*}
similarly 
\begin{equation*}
\begin{aligned}
I_2 \lesssim \left(\sup_{v_* \in \R^3}\int_{\R^3} |v-v_*|^{\gamma} |g (v)|^2 \langle v \rangle^{-\gamma}   dv  \right) \|  f  \| _{L^2_{\gamma/2}}^2   \lesssim \| f \|_{L^2_{\gamma/2}}^2  \| g \|_{H^1_{2}}^2 ,
\end{aligned}
\end{equation*}
the case $\gamma \in (-2, 0)$ is thus proved. When $\gamma \in (-3, -2]$,  by Lemma \ref{L25} we have
\begin{equation*}
\begin{aligned}
I_2   \lesssim \left(\sup_{v \in \R^3}\int_{\R^3} |v-v_*|^{\gamma} |f (v_*)|^2 \langle v_* \rangle^{\gamma}    dv_*  \right) \|  g  \| _{L^2_{-\gamma/2}}^2 \lesssim& \| f^2 \|_{L^1_{\gamma}}^{1 + \frac {  \gamma} 3} \| f^2 \|_{L^\infty_{\gamma}}^{ - \frac  { \gamma} 3}   \| g \|_{L^2_2}^2  
\\
\lesssim& \| f \|_{L^2_{\gamma/2 }}^{2 +  \frac {2\gamma} 3 }  \| f \|_{L^\infty_{\gamma/2}}^{ - \frac  {2\gamma}  3 }   \| g \|_{L^2_2}^2   \lesssim \| f \|_{H^2_{\gamma/2}}^2 \| g \|_{L^2_2}^2 ,
\end{aligned}
\end{equation*}
similarly 
\[
I_2  \lesssim \left(\sup_{v_* \in \R^3}\int_{\R^3} |v-v_*|^{\gamma} |g (v)|^2 \langle v \rangle^{-\gamma}   dv  \right) \|  f  \| _{L^2_{\gamma/2}}^2 \lesssim \| f \|_{L^2_{\gamma/2}}^2   \| g \|_{H^2_{2}}^2  ,
\]
and by Hardy-Littlewood-Sobolev inequality
\[
I_2  \lesssim \Vert f^2 \Vert_{L^p_{\gamma}} \Vert g^2 \Vert_{L^3_{-\gamma}} \lesssim \Vert f\Vert_{L^{2p}_{\gamma/2}}^2 \Vert g \Vert_{L^6_{-\gamma/2}}^2 \lesssim \Vert f\Vert_{H^1_{\gamma/2}}^2 \Vert g \Vert_{H^1_2}^2,
\]
where $p = \frac {3} {5+\gamma}$ implies $2p \in [2, 3)$. So the case $\gamma \in (-3, -2]$ is proved  by combining the three cases. 
\end{proof}

We will use the following representation of $v'$ which can be proved directly. We have
\begin{equation}
\label{representation v'}
\langle v' \rangle^2 = \langle v \rangle^2 \cos^2 \frac \theta 2 + \langle v_*\rangle^2 \sin^2 \frac \theta 2+ 2 \cos\frac \theta 2 \sin \frac \theta 2 |v-v_*| v \cdot \omega,   \quad \omega \perp (v-v_*), \quad v \cdot w = v_* \cdot w,
\end{equation}
where $\omega = \frac {\sigma - (\sigma \cdot k)k } {|\sigma - (\sigma \cdot k)k |}$ with $k = \frac {v-v_*} {|v-v_*|}$. We have the following estimate for the term $\langle v '\rangle^k$.

\begin{lem}\label{L27}
For any constant $k \ge 4$ we have
\begin{equation}
\label{v' estimate 1}
\langle v' \rangle^k =\sin^k \frac \theta 2 \langle v_* \rangle^k + R_1 +R_2, \quad |R_1| \le C_k \sin^2 \frac \theta 2 \langle v_* \rangle^{k-1}  \langle v \rangle, \quad |R_2 |  \le C_k   \langle v \rangle^k,
\end{equation}
for some constant $C_k>0$. We also have
\begin{equation}
\label{v' estimate 2}
\langle v' \rangle^{k} -\langle v \rangle^{k} \cos^{k} \frac \theta 2 =  k  \langle v \rangle^{k-2} \cos^{k-1} \frac \theta 2 \sin \frac \theta 2 |v-v_*| (v \cdot  \omega) + L_1 +L_2,
\end{equation}
with 
\[
|L_1| \le C_k\sin^{k-2} \frac \theta 2 \langle  v_* \rangle^{k} \langle v \rangle^2 ,\quad |L_2| \le C_k \langle v \rangle^{k-2}  \langle v_* \rangle^4 \sin^2 \frac \theta 2,
\]
for some constant $C_k>0$, in particular
\[
\langle v' \rangle^k = \cos^k \frac \theta 2 \langle v \rangle^k + Q_1, \quad |Q_1| \le  C_k  \langle v_* \rangle^k  \langle v \rangle^{k-1}.
\]
\end{lem}
\begin{proof} 

For simplicity, we prove the result for $\langle v' \rangle^{2k}$. By \eqref{representation v'} and mean value theorem we have
\begin{equation*}
\begin{aligned}
\langle v' \rangle^{2k} -\langle v_* \rangle^{2k} \sin^{2k} \frac \theta 2 =& k \int_0^1 \left(\langle v_* \rangle^2 \sin^2 \frac \theta 2 + t (\langle v\rangle^2 \cos^2 \frac \theta 2+ 2 \cos\frac \theta 2 \sin \frac \theta 2 |v-v_*| v \cdot \omega) \right)^{k-1} dt
\\
& \times \left (\langle v \rangle^2 \cos ^2 \frac \theta 2+ 2 \cos\frac \theta 2 \sin \frac \theta 2 |v-v_*| v \cdot \omega  \right) ,
\end{aligned}
\end{equation*}
since $v \cdot w = v_* \cdot w$ which implies $|v-v_*| |v \cdot \omega| \le |v| |v_*\cdot \omega| + |v_*| |v \cdot \omega| \le 2|v| |v_*|$, so we have
\[
\langle v\rangle^2 \cos^2 \frac \theta 2+ 2 \cos\frac \theta 2 \sin \frac \theta 2 |v-v_*| v \cdot \omega \le  \langle v \rangle^2 + 4 \sin  \frac \theta 2 |v| |v_*| \le 3 \langle v \rangle^2 + 2 \sin^2  \frac \theta 2 \langle v_* \rangle^2,
\]
so if $2k - 1 \ge 2$ we have
\begin{equation*}
\begin{aligned}
\langle v' \rangle^{2k} -\langle v_* \rangle^{2k} \sin^{2k} \frac \theta 2 \le & C_k (\sin^2 \frac \theta 2 \langle v_* \rangle^2  +  \langle v \rangle^2 )^{k-1} (  \langle v \rangle^2 + \sin  \frac \theta 2 |v_*| |v|   ) 
\\
\le &C_k  (\sin^{2k-2} \frac \theta 2 \langle v_* \rangle^{2k-2}  +  \langle v \rangle^{2k-2}) ( \sin \frac \theta 2 \langle v_* \rangle + \langle v \rangle )\langle v \rangle 
\\
\le &C_k \sin^2 \frac \theta 2 \langle v_* \rangle^{2k-1}  \langle v \rangle + C_k \langle v \rangle^{2k},
\end{aligned}
\end{equation*}
\eqref{v' estimate 1} is thus proved. For \eqref{v' estimate 2},  using \eqref{representation v'} and the mean value theorem twice we have
\begin{equation*}
\begin{aligned}
&\langle v' \rangle^{2k} -\langle v \rangle^{2k} \cos^{2k} \frac \theta 2 - 2k  (\langle v \rangle^{2} \cos^2 \frac \theta 2 )^{k-1} \cos \frac \theta 2 \sin \frac \theta 2 |v-v_*| (v \cdot \omega)
\\
=& k (\langle v \rangle^{2} \cos^2 \frac \theta 2 )^{k-1} \langle  v_* \rangle^2 \sin^2  \frac \theta 2\\
& + k ( k - 1 )\int_0^1 (1-t )\left(\langle v \rangle^2 \cos^2 \frac \theta 2 + t (\langle v_*\rangle^2 \sin^2 \frac \theta 2+ 2 \cos\frac \theta 2 \sin \frac \theta 2 |v-v_*| v \cdot \omega) \right)^{k-2} dt
\\
& \times \left (\langle v_*\rangle^2 \sin^2 \frac \theta 2+ 2 \cos\frac \theta 2 \sin \frac \theta 2 |v-v_*| v \cdot \omega  \right)^2  : =I_1 +I_2.
\end{aligned}
\end{equation*}
For $I_1$ we easily deduce 
\[
I_1 \le C_k \langle v \rangle^{2k-2} \langle v_* \rangle^{2} \sin^2 \frac \theta 2,
\]
since $|v-v_*| |v \cdot \omega|  \le 2|v| |v_*|$, we have
\[
\langle v_*\rangle^2 \sin^2 \frac \theta 2+ 2 \cos\frac \theta 2 \sin \frac \theta 2 |v-v_*| v \cdot \omega \le  \langle v_* \rangle^2\sin^2 \frac \theta 2 +  4|v||v_*| \sin \frac \theta 2 \le 3 \langle v_* \rangle^2\sin^2 \frac \theta 2 +  2\langle v \rangle^2,
\]
so we have
\begin{equation*}
\begin{aligned}
I_2 \le& C_k  (\langle v_* \rangle^2\sin^2 \frac \theta 2 +  4|v||v_*| \sin \frac \theta 2 )^2  (3 \langle v_* \rangle^2\sin^2 \frac \theta 2 +  3\langle v \rangle^2)^{k-2}
\\
\le & C_k\sin^2 \frac \theta 2 \langle  v_* \rangle^{2}  (\langle v_* \rangle^2 +\langle v \rangle^2) (\langle v_* \rangle^{2k-4}   \sin^{2k-4} \frac \theta 2 +  \langle v \rangle^{2k-4} )
\\
\le & C_k\sin^2 \frac \theta 2 \langle  v_* \rangle^{4} \langle v \rangle^2 (\langle v_* \rangle^{2k-4}   \sin^{2k-4} \frac \theta 2 +  \langle v \rangle^{2k-4} )
\\
\le & C_k\sin^{2k-2} \frac \theta 2 \langle  v_* \rangle^{2k} \langle v \rangle^2 +  C_k \langle v \rangle^{2k-2}  \langle v_* \rangle^4 \sin^2 \frac \theta 2, 
\end{aligned}
\end{equation*}
if $2k \ge 4$, so the theorem is thus proved. 
\end{proof}

\begin{lem}\label{L28} For any smooth function $g$ and $f$, if $\gamma > -3$ we have
\[
\int_{\R^3} \int_{\R^3} |v-v_*|^\gamma g_* f dv_* dv \le C\Vert g \Vert_{L^2_{ l}} \Vert f \Vert_{L^2_{ l}}^2 ,
\]
with $l = \max \{\gamma+2, \frac 3 2 \}$.
\end{lem}
\begin{proof}
We split it into two cases $\gamma \ge 0$ and $\gamma<0$, for the case $\gamma\ge 0$, we easily compute
\[
\int_{\R^3} \int_{\R^3} |v-v_*|^\gamma g_* f dv_* dv\lesssim \Vert f \Vert_{L^{1}_{\gamma}}\Vert g \Vert_{L^{1}_{\gamma}}\lesssim \Vert f \Vert_{L^{2}_{\gamma+2}}\Vert g \Vert_{L^{2}_{\gamma+2}}.
\]
For the case $\gamma<0$, by Hardy-Littlewood-Sobolev inequality we have 
\[
\int_{\R^3} \int_{\R^3} |v-v_*|^\gamma g_* f dv_* dv  \lesssim \Vert f \Vert_{L^{p}} \Vert g \Vert_{L^{p}}\lesssim \Vert f \Vert_{L^{2}_{3/2}}\Vert g \Vert_{L^{2}_{3/2}},
\]
where $\frac 1 p =\frac 1 2(2 + \frac \gamma 3) \in (\frac 1 2, 1)$ implies $p \in (1, 2)$. The proof is thus finished by gathering the two cases.
\end{proof}

\begin{lem}\label{L29}
For any $\gamma >-3$, $k >\max \{ 3, 3+\gamma \}$ we have
\[
\int_{\R^3} |v-v_*|^\gamma \langle v_* \rangle^{-k} dv_*\le \frac c {k} \langle v \rangle^\gamma + C_k \langle v \rangle^{\gamma-2}.
\]
\end{lem}

\begin{proof} For the  case $\gamma > 0$, we have
\[
\int_{\R^3} |v-v_*|^\gamma \langle v_* \rangle^{-k} dv_* \le C\int_{\R^3}( |v|^\gamma+ |v_*|^\gamma )\langle v_* \rangle^{-k} dv_*,
\]
It is easily seen that
\[
\int_{\R^3} |v_*|^\gamma \langle v_* \rangle^{-k} dv_* \le C_k \le \frac c {k} \langle v \rangle^\gamma + C_k \langle v \rangle^{\gamma-2},
\]
recall the theory of beta function  \eqref{beta function} we have
\[
\langle v \rangle^\gamma \int_{\R^3}  \langle v_* \rangle^{-k} dv_*  \le C \langle v \rangle^\gamma  \int_{0}^{+\infty} r^2 \frac 1 {(1+r^2)^{k/2}} dr \le \frac c {k^{3/2}}   \langle v \rangle^\gamma,
\]
so the case $\gamma \in  [0, 1]$ is thus proved. For the case $\gamma \in (-3, 0)$, If $|v| \le \frac 1 2$, we have $|v_*| +\frac 1 2 \le 1 + |v-v_*|$, so we have
\[
\int_{\R^3} |v-v_*|^\gamma \langle v_* \rangle^{-k} dv_* =\int_{\R^3} |v_*|^\gamma \langle v- v_* \rangle^{-k} dv_* \le C_k \int_{\R^3} |v_*|^\gamma \langle  v_* \rangle^{-k} dv_* \le C_k \le C_k \langle v \rangle^{\gamma-2}.
\]
Consider now $|v| > \frac 1 2$ and split the integral into two regions $|v - v_*| > \langle v \rangle /4$ and $|v - v_*| \le \langle v \rangle /4$. For the first region we obtain
\[
\int_{|v-v_*| > \frac {\langle v \rangle} 4 } |v-v_*|^\gamma \langle v_* \rangle^{-k} dv_* \le C \langle v \rangle^\gamma \int_{\R^3}  \langle v_* \rangle^{-k} dv_*  \le C \langle v \rangle^\gamma  \int_{0}^\infty r^2 \frac 1 {(1+r^2)^{k/2}} dr \le \frac c {k^{3/2}}  \langle v \rangle^\gamma.
\]
For the second region, $|v| > 1/2$ and $|v - v_*| \le \langle v \rangle /4$ imply $| v_*| \ge \langle v \rangle /8$, hence
\[
\int_{|v-v_*| \le \frac {\langle v \rangle} 4 } |v-v_*|^\gamma \langle v_* \rangle^{-k} dv_* \le C_k \langle v \rangle^{-k} \int_{|v-v_*| \le \frac {\langle v \rangle} 4 } |v-v_*|^\gamma dv_*   \le C_k \langle v \rangle^{-k+\gamma+3} \le \frac c {k} \langle v \rangle^\gamma + C_k \langle v \rangle^{\gamma-2},
\]
so the theorem is thus proved. 
\end{proof}

We introduce a change of variable which will be used frequently. 

\begin{lem}\label{L210} (\cite{S}, Lemma A.1.) 
For any non negative function $F$ in terms of $v_*, v, r =|v-v_*|, \theta, v', v_*'$, we have
\[
\int_{\R^3} \int_{\mathbb{S}^2} F d \sigma dv_* = 4 \int_{\R^3} \frac 1 {|v' -v|} \int_{ \{ w: w \cdot   ( v'-v)  =0 \}  } F \frac 1 {r} d w  d v',
\]
where $r, \theta, v_*', v_*$ in the left hand side is changed to 
\[
r=\sqrt{|v'-v|^2 +|w|^2 }, \quad \cos (\theta/2) = |w|/r, \quad v_*' = v+w, \quad v_* = v' +w,
\]
respectively.
\end{lem}

\begin{lem}\label{L211}
For any $-3 < \gamma \le 2$, for any constant $k>\max \{3, 3+\gamma \}$  we have
\[
I := \int_{\R^3} \int_{\mathbb{S}^2} |v-v_*|^\gamma \frac {\langle v \rangle^k} { \langle v' \rangle^k \langle v_*' \rangle^k}  dv_* d\sigma \le  C_k \langle v \rangle^\gamma,
\]
for all $v \in \R^d$. In fact for $\gamma \in [0, 1]$ we can prove a stronger estimate
\[
I := \int_{\R^3} \int_{\mathbb{S}^2} |v-v_*|^\gamma \frac {\langle v \rangle^k} { \langle v' \rangle^k \langle v_*' \rangle^k}  dv_* d\sigma \le\frac c k \langle v \rangle^\gamma  +   C_k \langle v \rangle^{\gamma-2}.
\]
\end{lem}

\begin{proof}
If $|v| \le 1$, since $\langle v_* \rangle \le \langle v_*' \rangle \langle v' \rangle$, we have
\[
I \le  2^k\int_{\R^3} \int_{\mathbb{S}^2} |v-v_*|^\gamma \frac {1} { \langle v_* \rangle^k}  dv_* d\sigma \le  C_k \langle v \rangle^{\gamma-2},
\]
we focus on the case $|v| \ge 1$ later. We first prove the case $\gamma \in [0, 1]$. By Lemma \ref{L210} and since $0 \le  \gamma \le 1$, we have
\begin{equation*}
\begin{aligned}
I =& \int_{\R^3}  \frac 1 {|v' -v|}  \frac {\langle v \rangle^k} { \langle v' \rangle^k }   \int_{ \{ w: w \cdot   ( v'-v)  =0 \}  }  \frac 1 {\sqrt{|v'-v|^2 +|w|^2 }} (\sqrt{|v'-v|^2 + |w|^2})^{\gamma}  \langle v+w \rangle^{-k} d w  dv'
\\
\le& \int_{\R^3}  \frac 1 {|v' -v|^{2-\gamma}}  \frac {\langle v \rangle^k} { \langle v' \rangle^k }   \int_{ \{ w: w \cdot   ( v'-v)  =0 \}  } \langle v+w \rangle^{-k} d w  dv'.
\end{aligned}
\end{equation*}
We first split $v$ into two parts
\begin{equation}
\label{decomposition v}
v = v_\perp +v_\parallel , \quad v_\perp =  \frac {v \cdot ( v'-v)} {|v-v'|^2}(v'-v), \quad |v|^2 =|v_\perp|^2 +|v_\parallel |^2, \quad  v_\perp \perp v_\parallel  , \quad v_\parallel \parallel  w, \quad |v +w|^2  = |w+v_\parallel|^2 +|v_\perp|^2. 
\end{equation}
Then the integral becomes
\[
 \int_{ \{ w: w \cdot   ( v'-v)  =0 \}  } \langle v+w \rangle^{-k} d w = \int_{\R^2} \frac 1 { (1+ |w+v_\parallel|^2 +|v_\perp|^2 )^{k/2} } d w =  \int_{\R^2} \frac 1 { (1+ |w|^2 +|v_\perp |^2 )^{k/2} } d w,
\]
making a change of variable $w =  \sqrt{1 + |v_\perp|^2} x $ we have
\[
 \int_{\R^2} \frac 1 { (1+ |w|^2 +|v_\perp|^2 )^{k/2} } d w = (1 + |v_\perp|^2)^{1-\frac k 2} \int_{\R^2} \frac 1 { (1+ |x|^2)^{k/2} } d x = 2\pi \frac {1} {k-2}  (1 + |v_\perp|^2)^{1-\frac k 2},
\]
so $I$ turns to
\[
I \le 2\pi \frac {1} {k-2}  \int_{\R^3}  \frac 1 {|v' -v|^{2-\gamma}}  \frac {\langle v \rangle^k} { \langle v' \rangle^k } \frac 1 {\langle v_\perp \rangle^{k-2} } dv', \quad |v_\perp| = \frac {|v \cdot ( v'-v)|} {|v-v'|}.
\]
Since  $|v_\perp| \le |v|$, so we split it into three cases $|v_\perp| \ge (1- \frac {1} k )|v|$,  $|v_\perp| \le \frac {|v|} k$ and $\frac 1 k |v| \le  |v_\perp| \le (1- \frac {1} k )|v|$.  For the case  $|v_\perp| \ge (1- \frac {1} k )|v|$ we have
\[
\langle v \rangle^{k-2} \le ( 1 + \frac {1} {k-1} )^{k-2}  \langle v_\perp \rangle^{k-2} \le e \langle v_\perp \rangle^{k-2},
\]
together with Lemma \ref{L29} we have
\[
I \le \frac c {k} \langle v \rangle^2 \int_{\R^3} \frac 1 {|v' -v|^{2-\gamma}}  \frac {1} { \langle v' \rangle^k } dv' \le   \frac c {k^2} \langle v \rangle^2 \langle v \rangle^{\gamma-2}  +C_k  \langle v \rangle^2 \langle v \rangle^{\gamma-4} \le \frac c k \langle v \rangle^{\gamma} +C_k \langle v \rangle^{\gamma-2}.  
\]
For the case $|v_\perp| \le \frac {1} k |v|$ we have
\[
|v'|^2 = |v-v'|^2 +|v|^2 +2 v\cdot(v'-v ) \ge  |v-v'|^2 +|v|^2 - 2 |v-v'| |v_\perp|   \ge  (1-\frac 1 k) (|v-v'|^2 +|v|^2 ),
\]
which implies
\begin{equation*}
\begin{aligned}
I \le&\frac c k \int_{\R^3}  \frac 1 {|v' -v|^{2-\gamma}}  \frac {\langle v \rangle^k} { (1+   |v-v'|^2 +|v|^2  )^{k/2} } \frac 1 {\langle v_\perp \rangle^{k-2} } dv'
\\
\le&\frac c k  \int_{\R^3}  \frac 1 {|v' -v|^{2-\gamma}}  \frac {1} { (1+  \frac   {|v-v'|^2} {1+|v|^2}   )^{k/2} } \frac 1 {\langle v_\perp \rangle^{k-2} }dv'.
\end{aligned}
\end{equation*}
We make the change of variables 
\begin{equation}
\label{change of variable r theta}
r = |v-v'|, \quad v \cdot (v'-v) =r |v| \cos \theta, \quad dv' = r^2 \sin \theta d r d\theta d \phi, \quad |v_\perp|^2 = |v|^2\cos^2  \theta, \quad |v_\parallel|^2 =|v|^2 \sin^2 \theta.
\end{equation}
So $I$ turns to 
\begin{equation*}
\begin{aligned}
I \le& \frac c k  2\pi \int_{0}^\infty r^\gamma \frac {1} { (1+  \frac   {r^2} {1+|v|^2}   )^{k/2} }\int_0^\pi \frac 1 {\langle |v|^2\cos^2\theta\rangle^{k-2} } \sin \theta dr d\theta .
\end{aligned}
\end{equation*}
Making another change of variables
\begin{equation}
\label{change of variable x y}
r =\sqrt{1+|v|^2} x, \quad y=|v| \cos\theta, \quad dy=|v| \sin\theta d\theta,
\end{equation}
then by \eqref{beta function} we deduce
\begin{equation*}
\begin{aligned}
I \le& \frac c k  2\pi  \frac 1 {|v|} \langle v \rangle^{\gamma+1} \int_{0}^\infty x^\gamma \frac {1} { (1+ |x|^2 )^{k/2} } \int_{-|v|}^{|v|} \frac 1 {\langle y \rangle^{k-2}} dy dx 
\\
\le &\frac c k  2\pi  \langle v \rangle^\gamma \int_{0}^\infty x^\gamma \frac {1} { (1+ |x|^2 )^{k/2} } dx \int_{-\infty}^{+\infty} \frac 1 {\langle y \rangle^{k-2}} dy \le\frac c k  2\pi \langle v \rangle^\gamma.
\end{aligned}
\end{equation*}
For the case $\frac 1 k |v| \le |v_\perp| \le (1-\frac 1 k)|v| $ we have
\[
\langle v \rangle^{k-2} \le  k^{k-2}  \langle v_\perp \rangle^{k-2} \le C_k \langle v_\perp \rangle^{k-2}, \quad  |v'|^2\ge  |v-v'|^2 +|v|^2 - 2 |v-v'| |v_\perp|   \ge  \frac 1 k (|v-v'|^2 +|v|^2 ),
\]
still make the same change of variables as \eqref{change of variable r theta} and \eqref{change of variable x y} we have
\begin{equation*}
\begin{aligned}
I \le& C_k \langle v \rangle^2 \int_{\R^3} \frac 1 {|v' -v|^{2-\gamma}}  \frac {1} { (1+   |v-v'|^2 +|v|^2  )^{k/2} } dv' 
\\
\le&C_k \langle v \rangle^{2-k}\int_{\R^3}  \frac 1 {|v' -v|^{2-\gamma}}  \frac {1} { (1+  \frac   {|v-v'|^2} {1+|v|^2}   )^{k/2} } dv'
\\
 \le& C_k   2\pi\langle v \rangle^{2-k} \int_{0}^\infty r^\gamma \frac {1} { (1+  \frac   {r^2} {1+|v|^2}   )^{k/2} } dr d\theta 
 \\
 \le& C_k   2\pi\langle v \rangle^{2 + \gamma - k} \int_{0}^\infty x^\gamma \frac {1} { (1+  x^2  )^{k/2} } dx
 \le  \frac c k \langle v \rangle^{\gamma} +C_k \langle v \rangle^{\gamma-2},
\end{aligned}
\end{equation*}
the case $\gamma \in [0, 1]$ is thus proved. For  the case $\gamma \in (-3, 0)$, since  $\langle v \rangle^k \le C_k \langle v' \rangle^k +  C_k \langle v_*' \rangle^k$ , we have
\[
I \le  C_k\int_{\R^3} \int_{\mathbb{S}^2} |v-v_*|^\gamma \frac {1} { \langle v' \rangle^k}  dv_* d\sigma +  C_k\int_{\R^3} \int_{\mathbb{S}^2} |v-v_*|^\gamma \frac {1} { \langle v_*' \rangle^k}  dv_* d\sigma :=I_1 +I_2 . 
\]
We only prove the term $I_1$ since the $I_2$ term is easily achieved by interchange $v'$ and $v_*'$. For the $I_1$ term  by Lemma \ref{L210} we have
\begin{equation*}
\begin{aligned}
I_1  \le &C_k \int_{\R^3}  \frac 1 {|v' -v|}\frac {1} { \langle v' \rangle^k}  \int_{ \{ w: w \cdot   ( v'-v)  =0 \}  }  \frac 1 {\sqrt{|v'-v|^2 +|w|^2 }} (\sqrt{|v'-v|^2 + |w|^2})^{\gamma}   d w  dv'
\\
\le &C_k  \int_{\R^3}  \frac 1 {|v' -v|} \frac {1} { \langle v' \rangle^k}   \int_{ \R^2 }  \frac 1 {(|v'-v|^2 +|w|^2 )^{\frac {1-\gamma} 2} }   d w  dv',
\end{aligned}
\end{equation*}
by a change of variable $w = |v-v'| x$ we deduce
\[
I_1 \le C_k  \int_{\R^3}  |v' -v|^\gamma \frac {1} { \langle v' \rangle^k}   \int_{ \R^2 }  \frac 1 {(1+|x|^2 )^{\frac {1-\gamma} 2} }   dx  dv'
\le C_k \int_{\R^3}  |v' -v|^\gamma \frac {1} { \langle v' \rangle^k}    dv',
\le C_k \langle v \rangle^\gamma,
\]
since  $\gamma<0$ implies $1-\gamma>1$ so $(1+|x|^2)^{- \frac {1-\gamma} 2} $ is integrable, the case $\gamma \in (-3, 0)$ is thus proved. For the case $\gamma \in [1, 2]$, since $1 \le \gamma$, we have
\begin{equation*}
\begin{aligned}
I =& \int_{\R^3}  \frac 1 {|v' -v|}  \frac {\langle v \rangle^k} { \langle v' \rangle^k }   \int_{ \{ w: w \cdot   ( v'-v)  =0 \}  }  \frac 1 {\sqrt{|v'-v|^2 +|w|^2 }} (\sqrt{|v'-v|^2 + |w|^2})^{\gamma}  \langle v+w \rangle^{-k} d w  dv'
\\
\le&C \int_{\R^3}  \frac 1 {|v' -v|^{2-\gamma}}  \frac {\langle v \rangle^k} { \langle v' \rangle^k }   \int_{ \{ w: w \cdot   ( v'-v)  =0 \}  } \langle v+w \rangle^{-k} d w  dv'
\\
&+ C\int_{\R^3}  \frac 1 {|v' -v|} \frac {\langle v \rangle^k} { \langle v' \rangle^k }   \int_{ \{ w: w \cdot   ( v'-v)  =0 \}  } |w|^{\gamma-1} \langle v+w \rangle^{-k} d w  dv' :=I_1 +I_2,
\end{aligned}
\end{equation*}
since $\gamma \le 2$, the $I_1$ term is the same as the term $I$ in the case $\gamma \in [0, 1]$ so can be estimated by the same way. We now focus on the $I_2$ term, we have
\begin{equation*}
\begin{aligned}
 \int_{ \{ w: w \cdot   ( v'-v)  =0 \}  } |w|^{\gamma-1} \langle v+w \rangle^{-k} d w =& \int_{\R^2} \frac { |w|^{\gamma-1} } { (1+ |w+v_\parallel|^2 +|v_\perp|^2 )^{k/2} } d w 
\\ 
\le & C \int_{\R^2} \frac { |w|^{\gamma-1} +|v_\parallel|^{\gamma-1} }{ (1+ |w|^2 +|v_\perp |^2 )^{k/2} } d w.
\end{aligned}
\end{equation*}
Making the change of variable $|w|= \sqrt{1 + |v_\perp|^2} x $ we have
\begin{equation*}
\begin{aligned}
\int_{\R^2} \frac { |w|^{\gamma-1} +|v_\parallel|^{\gamma-1} }{ (1+ |w|^2 +|v_\perp |^2 )^{k/2} } d w  \le & |v_\parallel|^{\gamma-1} (1 + |v_\perp|^2)^{1-\frac k 2} \int_{\R^2} \frac 1 { (1+ |x|^2)^{k/2} } dx
\\
&+ \langle v_\perp  \rangle^{\gamma-1} (1 + |v_\perp|^2)^{1-\frac k 2} \int_{\R^2} \frac {|x|^\gamma} { (1+ |x|^2)^{k/2} } d x
\\
&\le C_k  \langle v \rangle^{\gamma-1} (1 + |v_\perp|^2)^{1-\frac k 2} ,
\end{aligned}
\end{equation*}
we decude
\[
I_2\le C_k  \langle v \rangle^{\gamma-1} \int_{\R^3} \frac 1 {|v' -v|}  \frac {\langle v \rangle^k} { \langle v' \rangle^k } \frac 1 {\langle v_\perp \rangle^{k-2} } dv', 
\]
it is easily seen that $I_2$ also be estimated  by the same way as the term $I$ in the case $\gamma \in [0, 1]$, so the proof is thus finished. 
\end{proof}

For the exponential weight case, we have a better estimate for the linearized operator.

\begin{lem}\label{L212}
For any $-3 < \gamma \le 1$, for any constant $a > 0$ and $b \in (0, 2)$ we have
\[
I := \int_{\R^3} \int_{\mathbb{S}^2} |v-v_*|^\gamma \frac {e^{a   \langle v \rangle^b}} {e^{ a  \langle v' \rangle^b}} e^{-\frac 1 2 |v_*'|^2 } dv_* d\sigma \le  C_{a , b} \langle v \rangle^{\gamma - \frac {b (\gamma+3)} 4},
\]
for all $v \in \R^d$. 

\end{lem}

\begin{proof}
Since $\gamma - 1 \le 0$, by Lemma \ref{L210} we have
\begin{equation*}
\begin{aligned}
I & = 4\int_{\R^3} \frac 1 {|v' -v|}      \frac {e^{ a\langle v \rangle^b }} {e^{ a   \langle v' \rangle^b}}   \int_{ \{ \omega: \omega \cdot   ( v'-v)  =0 \}  }  \frac 1 {\sqrt{|v'-v|^2 +|w|^2 }} (\sqrt{|v'-v|^2 + |w|^2})^{\gamma}  e^{- \frac {|v+w|^2} 2}  d w  dv'
\\
& \le  4\int_{\R^3} \frac 1 {|v' -v|^{\frac {3-\gamma} 2 }} \frac {e^{a  \langle v \rangle^b }} {e^{a   \langle v' \rangle^b }}  \int_{ \{ \omega: \omega \cdot   ( v'-v)  =0 \}  } |w|^{\frac {\gamma-1} 2 } e^{- \frac {|v+w|^2} 2}  d w dv'.
\end{aligned}
\end{equation*}
Recall the decomposition \eqref{decomposition v} we have
\[
I \le  4\int_{\R^3} \frac 1 {|v' -v|^{\frac {3-\gamma} 2 }}    \frac {e^{a \langle v \rangle^b }} {e^{a \langle v' \rangle^b }}  e^{- \frac {|v_\perp|^2}  2 }  \int_{ \{ \omega: \omega \cdot   ( v'-v)  =0 \}  } |w|^{\frac {\gamma-1} 2 } e^{- \frac {|v_\parallel +w|^2}  2 } d w  dv'.
\]
Since $\frac {\gamma-1} 2 > -2$, we have
\[
\int_{ \{ \omega: \omega \cdot   ( v'-v)  =0 \}  } |w|^{\frac {\gamma-1} 2 } e^{- \frac {|v_\parallel +w|^2}  2 }  d w = \int_{\R^2 } |w - v_\parallel |^{\frac {\gamma-1} 2 } e^{- \frac {|w|^2}  2 }  d  w \le C \langle v_\parallel \rangle^{\frac {\gamma-1} 2 },
\]
which implies
\[
I \le  C \int_{\R^3}  |v' -v|^{\frac {\gamma-3} 2 } \frac {e^{a    \langle v \rangle^b }} {e^{a \langle v' \rangle^b }}    e^{- \frac {|v_\perp|^2}  2 } \langle v_\parallel \rangle^{\frac {\gamma-1} 2 } dv' .
\]
If $|v| \le 1$, since $\frac {\gamma-3} 2  > -3 $ it is easily seen that
\[
I \le  C_{a, b } \int_{\R^3}  |v' -v|^{\frac {\gamma-3} 2 }  \frac {1} {e^{a \langle v' \rangle^b }}   d v' \le  C_{a  , b } \le  C_{a , b } \langle v \rangle^{\gamma - \frac {b (\gamma+3)} 4},
\]
so we focus on the $|v|>1$ case. We split it into two case $|v_\perp| \ge \frac {|v|} 2$ and $|v_\perp| <\frac {|v|} 2$. For the case $|v_\perp| \ge \frac {|v|} 2$ we easily compute
\[
I \le  C \int_{\R^3}  |v' -v|^{\frac {\gamma-3} 2 } \frac {e^{a  \langle v \rangle^b }} {e^{a \langle v' \rangle^b  }}    e^{- \frac {|v|^2}  8 }  dv'  \le  C_{a  , b}  e^{- \frac {|v|^2}  {16} }  \int_{\R^3}  |v' -v|^{\frac {\gamma-3} 2 } \frac {1} {e^{a  \langle v' \rangle^b }}    dv'   \le  C_{a  , b } \langle v \rangle^{\gamma - b}. 
\]
For the case $|v_\perp| <\frac {|v|} 2$ we have $|v_\parallel| \ge \frac {|v|} 2$ which implies $\langle v_\parallel \rangle^{\frac {\gamma-1} 2} \le C \langle v\rangle^{\frac {\gamma-1} 2}$, so we have
\[
I \le   C \langle v \rangle^{\frac {\gamma-1} 2 }\int_{ \R^3 }  |v' -v|^{\frac {\gamma-3} 2 } \frac {e^{a \langle v \rangle^b }} {e^{a  \langle v' \rangle^b }}    e^{- \frac {|v_\perp|^2}  2 }  dv' .
\]
Still take the change of variables \eqref{change of variable r theta}, since $|v'|^2 = |v-v'|^2 +|v|^2 +2 v\cdot(v'-v ) $,  $I$ turns to 
\[
I \le C  \langle v  \rangle^{\frac {\gamma-1} 2 }   \int_{0}^{+\infty} \int_0^{\pi}  r^{\frac {\gamma+1} 2} e^{a (1+|v|^2)^{\frac b 2} - a (1+ |v|^2 +r^2 + 2r|v| \cos \theta  )^{\frac b 2}  }  e^{- \frac {|v|^2 \cos^2 \theta}  2 }     \sin \theta dr d \theta. 
\]
Since $0< b <2$, there exists a constant $C_{a, b}$ such that
\[
\frac {|x|^2} 4 +C_{a, b} > 4a \langle x \rangle ^b, \quad \forall x \in \R, \quad 
x^{\frac b 2} + y^\frac b 2 \ge (x+y)^\frac b 2, \quad \forall x, y \ge 0, 
\]
which implies that 
\[
\frac {|v|^2 \cos^2 \theta}  4 +C_{a, b } \ge  a (4+4|v|^2 \cos^2 \theta  )^{\frac b 2}, 
\]
and
\[
 a (1+ |v|^2 +r^2 +2r|v| \cos \theta  )^{\frac b 2} + a (4+4|v|^2 \cos^2 \theta  )^{\frac b 2} \ge a  (1+ |v|^2 +r^2 + 2r|v| \cos \theta +4|v|^2 \cos^2 \theta  )^{\frac b 2} \ge  a ( 1+|v|^2 +\frac 1 2 r^2)^{\frac b 2}.
\]
So we have
\[
I \le C_{a, b } \langle v \rangle^{\frac {\gamma-1} 2 } \int_{0}^{+\infty} \int_0^{\pi}   r^{\frac {\gamma+1} 2} e^{a (1+|v|^2)^{\frac b 2} - a  (1+ |v|^2 +\frac {r^2} 2  )^{\frac b 2}  }  e^{- \frac {|v|^2 \cos^2 \theta}  4 }  \sin \theta  dr d \theta.
\]
As before, taking another change of variables \eqref{change of variable x y} we have
\begin{equation*}
\begin{aligned}
I \le& C_{a, b } \frac 1 {|v|} \langle v \rangle^{\gamma+1  } \int_{0}^{+\infty}   x^{\frac {\gamma+1} 2} e^{a (1+|v|^2)^{\frac b 2} - a (1+ |v|^2 +\frac {x^2 (1+|v|^2) } 2  )^{\frac b 2}  }  dx \int_{-|v|}^{|v|} e^{- \frac {y^2}  4 } d y  
\\
\le & C_{a, b } \langle v \rangle^{\gamma  } \int_{0}^{+\infty}   x^{\frac {\gamma+1} 2} e^{a (1+|v|^2)^{\frac b 2}(1 -  (1 +\frac {x^2  } 2  )^{\frac b 2} ) }  dx \int_{-\infty}^{+\infty} e^{- \frac {y^2}  4 } d y 
\\
\le & C_{a, b } \langle v \rangle^{\gamma  } \int_{0}^{+\infty}   x^{\frac {\gamma+1} 2} e^{a \langle v \rangle^b (1 -  (1 +\frac {x^2  } 2  )^{\frac b 2} ) }  dx.
\end{aligned}
\end{equation*}
We split it into to two case $x \ge 2$ and $x \in [0,  2 ]$.  For the case $x \ge 2$ we have $ (1 +\frac {x^2  } 2  )^{\frac b 2}  -1   \ge C_b \langle x \rangle^{b }$ for some $C_b>0$. Together with $\langle v \rangle^b + \langle x \rangle^b \le  2\langle x \rangle^b \langle v \rangle^b$ implies
\begin{equation*}
\begin{aligned}
\int_{2}^{+\infty}   x^{\frac {\gamma+1} 2} e^{a \langle v \rangle^b (1 -  (1 +\frac {x^2  } 2  )^{\frac b 2} ) }  dx \le & C_b \int_{2}^{+\infty}   x^{\frac {\gamma+1} 2} e^{ - a C_b \langle v \rangle^b  \langle x \rangle^b } dx  
\\
\le& C_b e^{- \frac a 2 C_b \langle v \rangle^b}\int_{0}^{+\infty}   x^{\frac {\gamma+1} 2} e^{ - \frac a  2 C_b  \langle x \rangle^b } dx \le C_{a, b} \langle v \rangle^{ -b}.
\end{aligned}
\end{equation*}
For the case $x \in [0, 2]$, we have $(1 +\frac {x^2  } 2  )^{\frac b 2}  -1  \ge C_b x^2$ for some constant $C_b>0$. Making the change of variable $z = \langle v \rangle^{\frac b 2} x$, since $\frac {\gamma+1} 2 >-1$ we have
\begin{equation*}
\begin{aligned}
\int_{0}^{ 2}   x^{\frac {\gamma+1} 2} e^{a \langle v \rangle^b (1 -  (1 +\frac {x^2  } 2  )^{\frac b 2} ) }  dx \le& C_b \int_{0}^{+\infty}   x^{\frac {\gamma+1} 2}   e^{ - a C_b \langle v \rangle^b   x ^2 } dx   
\\
\le & C_b \langle v \rangle^{- \frac {b(\gamma +  3)} 4 }   \int_{0}^{+\infty}   z^{\frac {\gamma+1} 2}  e^{ - a  C_b | z |^b } dz \le C_{a, b} \langle v \rangle^{- \frac {b (\gamma+3)} 4 },
\end{aligned}
\end{equation*}
so the theorem is thus proved by gathering the terms together. 
\end{proof}

We introduce the following lemma about relative entropy which will be used later.

\begin{lem} (\cite{G9}, \cite{DHWY}, Lemma 2.7) \label{L213} For any smooth function $F$ satisfies \eqref{conservation law} and \eqref{entropy inequality},  we have
\begin{equation*}
\begin{aligned}
&\int_{\T^3} \int_{\R^3} \frac {|F(t, x, v) -\mu(v)|^2 } {4 \mu(v)} 1_{|F(t, x, v -\mu(v)| \le \mu(v)} dv dx 
\\
+&\int_{\T^3} \int_{\R^3} \frac  1 4  |F(t, x, v) -\mu(v)| 1_{|F(t, x, v) -\mu(v)| \ge \mu(v)} dv dx  \le H ( F_0 ),
\end{aligned}
\end{equation*}
where the relative entropy $H$ is defined in \eqref{entropy}.

 The following estimate is established for the weight function $w(\alpha, \beta)$ defined in \eqref{weight function}.
\begin{lem}
Suppose $|\alpha|, |\beta|$ non-negative integers, $k \in \R$, the weight function $w(\alpha, \beta)$ satisfies the following properties
\[
w(|\alpha|, |\beta|) \le  w(|\alpha|, |\beta_1|), \quad w(|\alpha|, |\beta|)   \le  w(|\alpha_1|, |\beta|), \quad \forall 0 \le |\alpha_1| < |\alpha|, \quad 0 \le|\beta_1 | < |\beta|,
\]
and
\begin{equation}
\label{w v nabla x}
w(|\alpha|, |\beta|) \le \langle v \rangle^{\min \{ \gamma, 0 \} } w (|\alpha|+1, |\beta|-1 ), \quad \forall |\alpha| \ge 0, \quad \forall |\beta| \ge 1,
\end{equation}
which implies
\[
\langle v \rangle^{k} \le \min_{|\alpha| +|\beta| \le 4} \{ w(|\alpha|, |\beta|) \}.
\]
We also have 
\begin{equation}
\label{ab3}
\max_{|\alpha|+|\beta| =3}w^2 (\alpha, \beta)  \le  w(0, 2) w(1, 2), 
\end{equation}
and 
\begin{equation}
\label{ab4}
\max_{|\alpha|+|\beta| =4}w^2 (\alpha, \beta) \le \min \{ w(1, 2)^{4/5} w(2, 2)^{6/5}, w(1, 2) w(2, 2), w(0, 3) w(1, 3)  \}.
\end{equation}
\end{lem}

\begin{proof}
The first is just form the fact that $a,  b \ge 0$. \eqref{w v nabla x} just follows form the fact that
\[
k- a|\alpha|  - b|\beta| + c \le k - a(|\alpha| + 1)  - b ( |\beta| - 1)  +  c  + \min \{ \gamma, 0\}  \quad \Longleftrightarrow  \quad  b -a +\min \{ \gamma, 0\} \ge 0,
\]
and we conclude form the definition of $a$ and $b$. Then third statement is just by
\[
\langle v \rangle^{k}  =  w(0, 4) = \min_{|\alpha| +|\beta| \le 4} \{ w(|\alpha|, |\beta|) \}.
\]
For equation \eqref{ab3} we easily compute 
\[
\max_{|\alpha|+|\beta| =3}w^2 (\alpha, \beta)  = w^2(3, 0) = 2k-6a +2c \le 2k - a -4b +2c =  w(0, 2) w(1, 2). 
\]
For equation \eqref{ab3} and\eqref{ab4}, since $7a=6b$, we easily compute 
\[
\max_{|\alpha|+|\beta| =4}w^2 (\alpha, \beta) = w^2(4, 0)  = 2k - 8a +2c  \le 2k- \frac {16 } 5 a - 4b  +2c =  w(1, 2)^{4/5} w(2, 2)^{6/5}, 
\]
and
\[
\max_{|\alpha|+|\beta| =4}w^2 (\alpha, \beta) = w^2(4, 0) =  2k - 8a  +2c  = 2k - 6b - a  +2c=   w(0, 3) w(1, 3) \le w(1, 2) w(2, 2), 
\]
so the proof of the lemma is thus finished. 
\end{proof}

\end{lem}

Next we prove a lemma related to the exponential weight case.

\begin{lem} \label{L215}
For any $a >0, b \in (0, 2), k \ge 0$, define $f(x) : = a ( 1+ x )^{\frac b 2} -  \frac k 2  \ln (1 + x),  x \ge 0,  f(0)=a$, then the following two statements holds
\[
 f(c) \le f(d) +C_{k, a, b}, \quad \forall 0 \le c \le d, \quad  f(c + d ) \le f(c) +f(d) +C_{k, a, b} , \quad \forall c, d \ge 0,
\]
for some constant $C_{k, a, b}>0$ independent of $c, d$. 
\end{lem}
\begin{proof}
We easily compute that 
\[
f ' (x) =  a \frac b 2 ( 1 + x )^{\frac b 2-1} -  \frac k 2  \frac 1 {1 + x}, \quad x  \in \R, \quad  x \ge 0,
\]
we easily we have there exists a constant $e$ which may depend on $k, a, b$ such that
\[
f'(x) <0 ,  \quad \hbox{if} \quad  x \le e, \quad  f'(x) \ge 0,  \quad \hbox{if} \quad x \ge e,
\]
so we have
\[
\max_{ 0 \le c \le d} f(c) -f(d) =f(0) -f(e) \le  C_{k, a, b} , \quad \hbox{if} \quad e >0,  \quad \max_{ 0 \le c \le d } f(c) -f(d) = 0, \quad \hbox{if} \quad e \le 0,
\]
so the first statement is thus proved. For the second statement since 
\[
f '' (x) = a  \frac b 2(\frac b  2 -1) (1 + x )^{b/2-2} +  \frac k 2  \frac 1 {(1+x)^2}, \quad  x \ge 0,
\]
Since $0 < b/2 <1$, there exists a constant $f$ which may depend on $k, a, b$ such that
\[
f'' (x) \ge 0 , \quad \hbox{if} \quad   x \le f, \quad  f''(x) < 0,  \quad \hbox{if} \quad  x \ge f.
\]
Thus for any  $c, d \ge 0$, we easily compute that
\[
f(c + d) -f(c)  - f(d) + f(0) = \int_0^d \int_{0}^c f''( x + s) ds dx \le \int_0^{\max \{f, 0\}} \int_0^{\max \{f, 0\}}  |f''(x+s)| ds dx \le C_{k, a, b},
\]
so the second statement is thus proved since $f(0 )= a$. 
\end{proof}

\begin{lem}\label{L216} For any constant $k \in \R, a >0, b \in (0, 2)$ we have
\[
\frac {e^{a\langle v \rangle^b }} {e^{ a \langle v' \rangle^b} e^{a\langle v_*' \rangle^b} }   \le C_{k, a, b} \frac {\langle v \rangle^k} {\langle v_* \rangle^{k} \langle v_*' \rangle^{k}},
\]
for some constant $C_{k, a, b}>0$.
\end{lem}

\begin{proof}
The case $k \le 0$ is easy, since $b \in (0, 2) $ we easily have
\[
\frac {e^{a \langle v \rangle^b }} {e^{ a \langle v' \rangle^b } e^{a \langle v_*' \rangle^b } }   \le 1 \le  \frac {\langle v \rangle^k} {\langle v_* \rangle^{k} \langle v_*' \rangle^{k}}.
\]
We focus on the case $k \ge 0$  later. This is equivalent to prove that 
\[
a\langle v \rangle^{b} - k \ln \langle v \rangle \le a\langle v' \rangle^{b } - k \ln \langle v' \rangle + a\langle v_*' \rangle^{b} - k \ln \langle v_*' \rangle  +C_{k, a, b},
\]
for some constant $C_{k, a, b}$ independent of $v$. By Lemma \ref{L215} we have
\begin{equation*}
\begin{aligned}
a \langle v \rangle^{b } - k \ln \langle v \rangle =f(|v|^2) \le& f(|v_*'|^2 +|v'|^2) +C_{k, a, b } 
\\
\le& f(|v_*'|^2) +f(|v'|^2) +C_{k, a, b} = a \langle v' \rangle^{b} - k \ln \langle v' \rangle + a \langle v_*' \rangle^{b} - k \ln \langle v_*' \rangle  +C_{k, a, b},
\end{aligned}
\end{equation*}
where $f$ is defined in Lemma \ref{L215}, so the proof is thus finished. 
\end{proof}

We also recall some basic interpolation on $x$, the proof is elementary and thus omitted. 
\begin{lem}
For any non-negative integer $m, n$,  for any function $f$, for any constant $ k \in \R$ we have
\[
\Vert f g \Vert_{H^2_x} \lesssim \min_{m + n = 2} \{ \Vert f \Vert_{H^m_x} \Vert g \Vert_{H^n_x}  \},
\]
and 
\begin{equation}
\label{Linfty}
\Vert f \Vert_{L^\infty_x L^2_v} \lesssim \Vert f \Vert_{H^{8/5}_x L^2_v} \le \Vert f \langle v \rangle^{\frac 3 2 k}\Vert_{H^1_{x}L^2_v}^{2/5} \Vert f\langle v \rangle^{-k} \Vert_{H^2_xL^2_x }^{3/5} \lesssim \Vert f \langle v \rangle^{\frac 3 2 k} \Vert_{H^1_{x}L^2_v}+ \Vert f \langle v \rangle^{-k}  \Vert_{H^2_xL^2_v} ,
\end{equation}
also we have
\begin{equation}
\label{L3}
\Vert f  \Vert_{L^3_xL^2_v}  \lesssim \Vert f  \Vert_{H^{1/2}_x L^2_v } \lesssim\Vert f \langle v \rangle^{k} \Vert_{L^2_{x}}+ \Vert f \langle v \rangle^{-k} \Vert_{H^1_x}.
\end{equation}
\end{lem}

\section{Linearized and nonlinear estimate for the Boltzmann operator}\label{section3}

In this chapter we will prove linearized and nonlinear estimate for the collision operator $Q$. 
\begin{lem}\label{L31}
For any $-3 <\gamma \le 1$, for any $k \ge 4$, $h, g$ smooth, we have
\begin{equation*}
\begin{aligned} 
|(Q (h, \mu), g \langle v \rangle^{2k}) | \le &  \Vert b(\cos \theta) \sin^{k-\frac {3+\gamma} 2} \frac \theta 2 \Vert_{L^1_\theta}   \Vert  h   \Vert_{L^2_{ k+\gamma/2, * } }\Vert g \Vert_{ {L}^2_{ k+\gamma/2, * } } + C_k \Vert h \Vert_{L^2_{k+\gamma/2-1/2}}\Vert g \Vert_{L^2_{k+\gamma/2-1/2}}
\\
\le & \Vert b(\cos \theta) \sin^{k-2} \frac \theta 2 \Vert_{L^1_\theta}   \Vert  h   \Vert_{L^2_{ k+\gamma/2, * } }\Vert g \Vert_{ {L}^2_{ k+\gamma/2, * } } + C_k \Vert h \Vert_{L^2_{k+\gamma/2-1/2}}\Vert g \Vert_{L^2_{k+\gamma/2-1/2}},
\end{aligned}
\end{equation*}
for some constant $C_k>0$.
\end{lem}
\begin{proof}
We first compute $Q^+$, by pre-post collisional change of variables and Lemma \ref{L27} we have
\begin{equation*}
\begin{aligned} 
(Q^+ (h, \mu), g \langle v \rangle^{2k}) =&\int_{\R^3}\int_{\R^3}\int_{\mathbb{S}^2} |v-v_*|^\gamma b(\cos \theta) h(v_*) \mu(v) g(v') \langle v' \rangle^{2k}  dv dv_* d \sigma
\\
=& \int_{\R^3}\int_{\R^3}\int_{\mathbb{S}^2}  |v-v_*|^\gamma  b(\cos \theta) \sin^k \frac \theta 2 h(v_*) \langle v_* \rangle^k \mu(v) g(v') \langle v' \rangle^k dv dv_* d \sigma
\\
&+ \int_{\R^3}\int_{\R^3}\int_{\mathbb{S}^2} |v-v_*|^\gamma  b(\cos \theta)  h(v_*) \mu(v)   g(v') \langle v' \rangle^k   R_1dv dv_* d \sigma
\\
&+ \int_{\R^3}\int_{\R^3}\int_{\mathbb{S}^2}  |v-v_*|^\gamma  b(\cos \theta)   h(v_*)  \mu(v)   g(v') \langle v' \rangle^k R_2 dv dv_* d \sigma
:= \Gamma_1 +\Gamma_2+\Gamma_3,
\end{aligned}
\end{equation*}
where 
\[
|R_1| \le C_k \sin^2 \frac \theta 2 \langle v_* \rangle^{k-1}  \langle v \rangle, \quad |R_2 |  \le C_k \langle v \rangle^k.
\] 
For the $\Gamma_1$ term,  by the first inequality of Lemma \ref{L22}
\begin{equation*}
\begin{aligned}
|\Gamma_1| \le &\int_{\mathbb{S}^2} b(\cos \theta) \sin^{k-\frac {3+\gamma} 2} \frac \theta 2   d\sigma \left(\int_{\R^3}\int_{\R^3}  |v-v_*|^\gamma |h_*|^2 \langle v_* \rangle^{2k} \mu dv dv_* \right)^{1/2}  
\\
&\left(\int_{\R^3}\int_{\R^3}  |v-v'|^\gamma \mu |g'|^2\langle v' \rangle^{2k}dv dv' \right)^{1/2}  =  \int_{\mathbb{S}^2} b(\cos \theta) \sin^{k-\frac {3+\gamma} 2} \frac \theta 2   d\sigma \Vert h \Vert_{L^2_{k+\gamma/2, *}}\Vert g \Vert_{L^2_{k+\gamma/2, *}}.
\end{aligned}
\end{equation*}
For the $\Gamma_2$ term, since $\frac {3+\gamma} 2 \le 2$, by Lemma \ref{L22} we have
\begin{equation*}
\begin{aligned}
|\Gamma_2| \lesssim& \int_{\R^3} \int_{\R^3} \int_{\mathbb{S}^2} b(\cos \theta) \sin^{2}  \frac \theta 2 |v-v_*|^\gamma  |h_*| \langle v_* \rangle^{k-1}  \mu  \langle v \rangle   |g'|\langle v' \rangle^k dvdv_* d\sigma
\\
\lesssim& \int_{\R^3} \int_{\R^3} \int_{\mathbb{S}^2} b(\cos \theta) \sin^{2}  \frac \theta 2  |v-v_*|^\gamma   |h_*|   \langle v_* \rangle^{k-1/2} \mu \langle v \rangle^{2}   |g'|\langle v' \rangle^{k-1/2} dvdv_* d\sigma
\\
\lesssim & \int_{\mathbb{S}^2} b(\cos \theta) \sin^{2-\frac {3+\gamma} 2} \frac \theta 2 d\sigma  \left( \int_{\R^3}\int_{\R^3}  |v-v_*|^\gamma |h_*|^2 \langle v_* \rangle^{2k-1} \langle v \rangle^{2}\mu  dv dv_*  \right)^{1/2} 
\\
&\left(\int_{\R^3}\int_{\R^3}  |v-v'|^\gamma \langle v \rangle^{2} \mu  |g'|^2\langle v' \rangle^{2k-1}dv dv' \right)^{1/2}
\lesssim \Vert h \Vert_{L^2_{k+\gamma/2-1/2}}\Vert g \Vert_{L^2_{k+\gamma/2-1/2}}.
\end{aligned}
\end{equation*}
For the $\Gamma_3$ term, since $k-1 \ge 3$, by Lemma \ref{L27} we have
\[
\langle v' \rangle^{k-1} \le C_k \sin^2 \frac \theta 2  \langle v_* \rangle^{k-1} + C_k \langle v \rangle^{k-1},
\]
thus we split $\Gamma_3$ into 
\begin{equation*}
\begin{aligned} 
|\Gamma_3|  \lesssim& \int_{\R^3}\int_{\R^3}\int_{\mathbb{S}^2}  |v-v_*|^\gamma  b(\cos \theta)   |h_*|  \mu \langle v \rangle^k   |g'| \langle v' \rangle^k dv dv_* d \sigma
\\
\lesssim&\int_{\R^3} \int_{\R^3} \int_{\mathbb{S}^2} b(\cos \theta) \sin^{2}  \frac \theta 2 |v-v_*|^\gamma  |h_*| \langle v_* \rangle^{k-1}  \mu  \langle v \rangle^{k}   |g'|\langle v' \rangle dvdv_* d\sigma
\\
&+\int_{\R^3} \int_{\R^3} \int_{\mathbb{S}^2} b(\cos \theta)  |v-v_*|^\gamma  |h_*|   \mu  \langle v \rangle^{2k-1}   |g'|  \langle v' \rangle   dv dv_* d\sigma : =\Gamma_{31} + \Gamma_{32},
\end{aligned}
\end{equation*}
the $\Gamma_{31}$ term can be proved the same way as the $\Gamma_2$ term. For $\Gamma_{32}$ term,  by the regular change of variables and Lemma \ref{L28} if $ k \ge 4$ we have 
\begin{equation*}
\begin{aligned}
|\Gamma_{32}| &\lesssim \int_{\R^3} \int_{\R^3} \int_{\mathbb{S}^2} b(\cos \theta) |v-v_*|^\gamma |h_*| \mu \langle v \rangle^{2k -1 } |g'| \langle v' \rangle   dvdv_* d\sigma
\\
&\lesssim \int_{\R^3} \int_{\R^3}  \int_{\mathbb{S}^2}b(\cos \theta) |v-v_*|^\gamma |h_*| \langle v_* \rangle^{1/2}  \mu \langle v \rangle^{ 2 k  -1 /2}  |g'| \langle v' \rangle^{1/2} dvdv_*  d \sigma
\\
&\lesssim\int_{\mathbb{S}^2} b(\cos \theta) \cos^{-(3 +\gamma ) } \frac \theta 2\int_{\R^3}\int_{\R^3}  |v'-v_*|^\gamma |h_*| \langle v_* \rangle^{1/2 }  |g'| \langle v' \rangle^{1/2  }  dv' dv_*   
\\
&\lesssim \Vert h \Vert_{L^{2}_{k+\gamma/2-1/2}}\Vert g \Vert_{L^{2}_{k+\gamma/2-1/2}},
\end{aligned}
\end{equation*}
the $Q^+$ term  is thus proved.  For the $Q^-$ part by Lemma \ref{L28} we have
\begin{equation*}
\begin{aligned} 
|(Q^-(h, \mu ), g \langle \cdot \rangle^{2k})| = & \left| \int_{\R^3}\int_{\R^3}\int_{\mathbb{S}^2} |v-v_*|^\gamma  b(\cos \theta) h(v_*) \mu(v)   g(v) \langle v \rangle^{2k}  dv dv_* d \sigma \right|
\\
\lesssim &\int_{\R^3}\int_{\R^3}\int_{\mathbb{S}^2} |v-v_*|^\gamma  b(\cos \theta) |h(v_*)|  |g(v)| dv dv_* d \sigma\lesssim  \Vert h \Vert_{L^2_{k+\gamma/2-1/2}}\Vert g \Vert_{L^2_{k+\gamma/2-1/2}},
\end{aligned}
\end{equation*}
 the theorem is thus proved by combing the $Q^+$ and $Q^-$ term. 
\end{proof}

\begin{lem}\label{L32}
Suppose that $- 3< \gamma \le 1$, then for any smooth function $f$ we have
\begin{equation*}
\begin{aligned}
(Q(\mu ,  f), f \langle v \rangle^{2k} )\le&  -\Vert  b(\cos \theta) (1- \cos^{k- \frac {3+ \gamma} 2} \frac \theta 2  )\Vert_{L^1_\theta}\Vert f \Vert_{L^2_{k+\gamma/2 ,  * }}^2 + C_{k}  \Vert f \Vert_{L^2_{k+\gamma/2-1/2}}^2,
\end{aligned}
\end{equation*}
for some constant $C_k>0$. In particular if $ k  \ge 4$  we have
\begin{equation*}
\begin{aligned}
(Q(\mu ,  f), f \langle v \rangle^{2k} )\le&  -\Vert  b(\cos \theta)  \sin^2 \frac \theta 2\Vert_{L^1_\theta}\Vert f \Vert_{L^2_{k+\gamma/2 ,  * }}^2 + C_{k}  \Vert f \Vert_{L^2_{k+\gamma/2-1/2}}^2.
\end{aligned}
\end{equation*}
\end{lem}

\begin{proof}
We first compute $Q^+$, by pre-post collisional change of variables and Lemma \ref{L27}
we have
\begin{equation*}
\begin{aligned} 
(Q^+ (\mu, f), f \langle v \rangle^{2k}) =&\int_{\R^3}\int_{\R^3}\int_{\mathbb{S}^2} |v-v_*|^\gamma b(\cos \theta) \mu(v_*) f(v) f(v') \langle v' \rangle^{2k}  dv dv_* d \sigma
\\
=& \int_{\R^3}\int_{\R^3}\int_{\mathbb{S}^2} |v-v_*|^\gamma  b(\cos \theta) \cos^k \frac \theta 2 \mu(v_*) f (v) \langle v \rangle^k  f(v') \langle v' \rangle^k dv dv_* d \sigma
\\
&+ \int_{\R^3}\int_{\R^3}\int_{\mathbb{S}^2} |v-v_*|^\gamma  b(\cos \theta)  \mu(v_*)  f(v)  g(v') \langle v' \rangle^k  Q_1 dv dv_* d \sigma
: = \Gamma_1 +\Gamma_2,
\end{aligned}
\end{equation*}
where $|Q_1| \le  C_k  \langle v_* \rangle^k  \langle v \rangle^{k-1}$. 
For the $\Gamma_1$ term,  by the second inequality of Lemma \ref{L22} we have
\begin{equation*}
\begin{aligned}
|\Gamma_1| &\le \int_{\mathbb{S}^2} b(\cos \theta) \cos^{k-\frac {3+\gamma} 2} \frac \theta 2   d\sigma \int_{\R^3}\int_{\R^3}  |v-v_*|^\gamma \mu(v_*)  |f|^2 \langle v \rangle^{2k}dv dv_*.
\end{aligned}
\end{equation*}
For the $\Gamma_2$ term, we compute
\begin{equation*}
\begin{aligned}
|\Gamma_2| \lesssim& \int_{\R^3} \int_{\R^3} \int_{\mathbb{S}^2} b(\cos \theta) |v-v_*|^\gamma \mu_*   \langle v_* \rangle^k   |f| \langle v \rangle^{k-1}  |f'|\langle v' \rangle^k dvdv_* d\sigma
\\
\lesssim& \int_{\R^3} \int_{\R^3} \int_{\mathbb{S}^2} b(\cos \theta)  |v-v_*|^\gamma  \mu_* \langle v_* \rangle^{k + 1/2} |f| \langle v \rangle^{k- 1/2}    |f'|   \langle v' \rangle^{k-1/2} dvdv_* d\sigma
\\
\lesssim &  \int_{\mathbb{S}^2} b(\cos \theta) \cos^{-\frac {3+\gamma} 2} \frac \theta 2   d\sigma \left( \int_{\R^3}\int_{\R^3}  |v-v_*|^\gamma\langle v_* \rangle^{k+1/2}  \mu_*  |f|^2 \langle v \rangle^{2k-1}  dv dv_*  \right)^{1/2} 
\\
&\left(\int_{\R^3}\int_{\R^3}  |v-v'|^\gamma \langle v_* \rangle^{ k + 1/2 } \mu_* |f'|^2\langle v' \rangle^{2k-1}dv dv' \right)^{1/2}
\lesssim \Vert f \Vert_{L^2_{k+\gamma/2-1/2}}^2.
\end{aligned}
\end{equation*}
It is easily seen that
\[
( - Q^- (\mu, f), f \langle v \rangle^{2k})  = - \int_{\mathbb{S}^2} b(\cos \theta)    d\sigma \int_{\R^3}\int_{\R^3}  |v-v_*|^\gamma \mu(v_*)  |f|^2 \langle v \rangle^{2k}dv dv_*,
\]
so the proof is finished by gathering the three terms. 
\end{proof}

Then we come to estimate the nonlinear operator $Q^+$.

\begin{lem}\label{L33}
For any $\gamma \in(-3, 1]$,  for smooth function $f, g, h$, recall $N$ is defined in \eqref{N}, for any $k\ge 4$ we have
\[
|( Q^+(f, g), h \langle v \rangle^{2k}) | \le C_k  \min_{m+n =  N-2} \{\Vert f \Vert_{H^m_{k+\gamma/2}} \Vert g \Vert_{H^n_4}  \}  \Vert h\Vert_{L^2_{k+\gamma/2}}+ C_k \min_{m_1 +n_1  = N-2} \{ \Vert g \Vert_{H^{m_1}_{k+\gamma/2}} \Vert f \Vert_{H^{n_1}_4}  \} \Vert h\Vert_{L^2_{k+\gamma/2}}, 
\]
where $m, n, m_1, n_1$ are nonnegative integers.
\end{lem}

\begin{proof}
We first prove the case $-\frac 3 2 < \gamma \le 1$, we have
\begin{equation*}
\begin{aligned}
|( Q^+(f, g), h \langle v \rangle^{2k})  | \le &   \int_{\R^3} \int_{\R^3}\int_{\mathbb{S}^2}  |v - v_* |^\gamma b(\cos \theta) |f(v_*')| | g(v')| |h(v) | \langle v \rangle^{2k}   dv dv_* d \sigma
\\
\le& \left(\int_{\R^3} \int_{\R^3}\int_{\mathbb{S}^2} b(\cos \theta ) \langle v \rangle^{2k-6+\gamma}  \langle v' \rangle^6 \langle v_*' \rangle^6     |f(v_*')|^2 | g(v')|^2 dv dv_* d \sigma   \right)^{1/2}
\\
&\left(\int_{\R^3} \int_{\R^3}\int_{\mathbb{S}^2}  |v - v_* |^{2\gamma}   b(\cos \theta ) \frac {\langle v \rangle^6} { \langle v' \rangle^6 \langle v_*' \rangle^6}   \langle v \rangle^{2k-\gamma} |h(v)|^2  dv dv_* d \sigma   \right)^{1/2},
\end{aligned}
\end{equation*}
Since $-3< 2\gamma \le 2$, by Lemma \ref{L211} we have
\[
\int_{\R^3} \int_{\R^3}\int_{\mathbb{S}^2}  |v - v_* |^{2\gamma}  b(\cos \theta ) \frac {\langle v \rangle^6} { \langle v' \rangle^6 \langle v_*' \rangle^6}   \langle v \rangle^{2k-\gamma} |h(v)|^2  dv dv_* d \sigma \le C_k \Vert h \Vert_{L^2_{k+\gamma/2}}^2.
\]
Since $2k-6+\gamma \ge 0$, by  pre-post collisional change of variables we have
\begin{equation*}
\begin{aligned}
&\int_{\R^3} \int_{\R^3}\int_{\mathbb{S}^2}  b(\cos \theta ) \langle v \rangle^{2k - 6 + \gamma}  \langle v' \rangle^6 \langle v_*' \rangle^6     |f(v_*')|^2 | g(v')|^2 dv dv_* d \sigma
\\
= &\int_{\R^3} \int_{\R^3}\int_{\mathbb{S}^2}  b(\cos \theta ) \langle v' \rangle^{2k - 6 + \gamma}  \langle v \rangle^6 \langle v_* \rangle^6     |f(v_*)|^2 | g(v)|^2 dv dv_* d \sigma 
\\
\lesssim & \int_{\R^3} \int_{\R^3}\int_{\mathbb{S}^2} b(\cos \theta )  \langle v \rangle^{2k + \gamma} \langle v_* \rangle^6     |f(v_*)|^2 | g(v)|^2 dv dv_* d \sigma
\\
&+\int_{\R^3} \int_{\R^3}\int_{\mathbb{S}^2}  b(\cos \theta )  \langle v \rangle^{6} \langle v_* \rangle^{2k+\gamma}     |f(v_*)|^2 | g(v)|^2 dv dv_* d \sigma
\\
\lesssim&  \Vert f \Vert_{L^2_{3}}^2 \Vert g \Vert_{L^2_{k+\gamma/2}}^2  +  \Vert g \Vert_{L^2_{3}}^2 \Vert f \Vert_{L^2_{k+\gamma/2}}^2 ,
\end{aligned}
\end{equation*}
so the case $\gamma \in (-\frac 3 2, 1]$ is proved.  For the case $\gamma \in (-\frac 5 2, -\frac 3 2]$,  we have $-3< \gamma-\frac 1 2 <0, -2 < \gamma + \frac 1 2 <0$, hence
\begin{equation*}
\begin{aligned}
|( Q^+(f, g), h \langle v \rangle^{2k})  | \le &   \int_{\R^3} \int_{\R^3}\int_{\mathbb{S}^2} |v-v_*|^\gamma b(\cos \theta) |f(v_*')| | g(v')| |h(v) | \langle v \rangle^{2k}   dv dv_* d \sigma
\\
\le& \left(\int_{\R^3} \int_{\R^3}\int_{\mathbb{S}^2} |v - v_*|^{\gamma+\frac 1 2} b(\cos \theta ) \langle v \rangle^{2k-\frac 9  2 }  \langle v' \rangle^4 \langle v_*' \rangle^4     |f(v_*')|^2 | g(v')|^2 dv dv_* d \sigma   \right)^{1/2}
\\
&\left(\int_{\R^3} \int_{\R^3}\int_{\mathbb{S}^2} |v - v_*|^{\gamma-\frac 1 2} b(\cos \theta ) \frac {\langle v \rangle^4} { \langle v' \rangle^4 \langle v_*' \rangle^4}   \langle v \rangle^{2k  + \frac 1 2 } |h(v)|^2  dv dv_* d \sigma   \right)^{1/2},
\end{aligned}
\end{equation*}
still by Lemma \ref{L211} we have
\[
\int_{\R^3} \int_{\R^3}\int_{\mathbb{S}^2} |v - v_*|^{\gamma-\frac 1 2} b(\cos \theta ) \frac {\langle v \rangle^4} { \langle v' \rangle^4 \langle v_*' \rangle^4}   \langle v \rangle^{2k + \frac 1 2 } |h(v)|^2  dv dv_* d \sigma  \le C_k \Vert h \Vert_{L^2_{k+\gamma/2}}^2,
\]
similarly we have 
\begin{equation*}
\begin{aligned}
&\int_{\R^3} \int_{\R^3}\int_{\mathbb{S}^2} |v - v_* |^{\gamma + \frac 1 2} b(\cos \theta ) \langle v' \rangle^{2k- \frac 9 2}  \langle v \rangle^4 \langle v_* \rangle^4     |f(v_*)|^2 | g(v)|^2 dv dv_* d \sigma
\\
\lesssim & \int_{\R^3} \int_{\R^3}\int_{\mathbb{S}^2} |v - v_* |^{\gamma + \frac 1 2} b(\cos \theta )  \langle v \rangle^{2k  - \frac 1 2 } \langle v_* \rangle^4     |f(v_*)|^2 | g(v)|^2 dv dv_* d \sigma
\\
&+\int_{\R^3} \int_{\R^3}\int_{\mathbb{S}^2}  |v - v_* |^{\gamma + \frac 1 2} b(\cos \theta )  \langle v \rangle^{4} \langle v_* \rangle^{2k   - \frac 1 2}     |f(v_*)|^2 | g(v)|^2 dv dv_* d \sigma.
:=I_1 +I_2.
\end{aligned}
\end{equation*}
Since $\gamma + \frac 1 2 \in (-2, 0)$ by Lemma \ref{L26}  we have
\[
I_1 \lesssim \min_{ m + n =1} \{ \Vert f \Vert_{H^m_4}^2  \Vert g \Vert_{H^n_{k+\gamma/2} }^2 \}, \quad I_2 \lesssim \min_{ m + n  =1} \{  \Vert f \Vert_{H^m_{k+\gamma/2} }^2   \Vert g \Vert_{H^n_4}^2 \},
\]
so the case $\gamma \in (-\frac 5 2, -\frac 3 2 ]$ is proved. For the case $\gamma \in (-3, 0)$ we have
\begin{equation*}
\begin{aligned}
|( Q^+(f, g), h \langle v \rangle^{2k})  | \le &   \int_{\R^3} \int_{\R^3}\int_{\mathbb{S}^2} |v - v_* |^\gamma b(\cos \theta) |f(v_*')| | g(v')| |h(v) | \langle v \rangle^{2k}   dv dv_* d \sigma
\\
\le& \left(\int_{\R^3} \int_{\R^3}\int_{\mathbb{S}^2} |v - v_* |^\gamma b(\cos \theta ) \langle v \rangle^{2k-4}  \langle v' \rangle^4 \langle v_*' \rangle^4     |f(v_*')|^2 | g(v')|^2 dv dv_* d \sigma   \right)^{1/2}
\\
&\left(\int_{\R^3} \int_{\R^3}\int_{\mathbb{S}^2} |v - v_* |^\gamma b(\cos \theta ) \frac {\langle v \rangle^4} { \langle v' \rangle^4 \langle v_*' \rangle^4}   \langle v \rangle^{2k} |h(v)|^2  dv dv_* d \sigma   \right)^{1/2}, 
\end{aligned}
\end{equation*}
still by Lemma \ref{L211} we have
\[
\int_{\R^3} \int_{\R^3}\int_{\mathbb{S}^2} |v - v_*|^{\gamma} b(\cos \theta ) \frac {\langle v \rangle^4} { \langle v' \rangle^4 \langle v_*' \rangle^4}   \langle v \rangle^{2k  } |h(v)|^2  dv dv_* d \sigma  \le C_k \Vert h \Vert_{L^2_{k+\gamma/2}}^2,
\]
and  by pre-post collisional change of variables we have
\begin{equation*}
\begin{aligned}
&\int_{\R^3} \int_{\R^3}\int_{\mathbb{S}^2} |v - v_* |^\gamma b(\cos \theta ) \langle v' \rangle^{2k-4}  \langle v \rangle^4 \langle v_* \rangle^4     |f(v_*)|^2 | g(v)|^2 dv dv_* d \sigma
\\
\lesssim & \int_{\R^3} \int_{\R^3}\int_{\mathbb{S}^2} |v - v_* |^\gamma b(\cos \theta )  \langle v \rangle^{2k} \langle v_* \rangle^4     |f(v_*)|^2 | g(v)|^2 dv dv_* d \sigma
\\
&+\int_{\R^3} \int_{\R^3}\int_{\mathbb{S}^2}  |v - v_* |^\gamma b(\cos \theta )  \langle v \rangle^{4} \langle v_* \rangle^{2k}     |f(v_*)|^2 | g(v)|^2 dv dv_* d \sigma:=I_1 +I_2.
\end{aligned}
\end{equation*}
Since $\gamma \in (-3, -\frac 5 2]$ by Lemma \ref{L25} we have
\[
I_1 \lesssim \min_{m + n =2} \{ \Vert f \Vert_{H^m_4}^2  \Vert g \Vert_{H^n_{k+\gamma/2} }^2 \}, \quad I_2 \lesssim \min_{m + n =2}  \{ \Vert f \Vert_{H^m_{k+\gamma/2} }^2   \Vert g \Vert_{H^n_4}^2 \},
\]
thus the theorem is proved by gathering all the terms together.
\end{proof}

The estimate for the $Q^-$ term is similar.
\begin{lem} \label{L34}
For any $\gamma \in(-3, 1]$,  for smooth function $f, g, h$, recall $N$ is defined in \eqref{N}, for any $k\ge 4$ we have
\[
|( Q^-(f, g), h \langle v \rangle^{2k}) | \le C_k \min_{m+n = N-2}  \{ \Vert f \Vert_{H^m_3}  \Vert g \Vert_{H^n_{k+\gamma/2}}  \} \Vert h\Vert_{L^2_{k+\gamma/2}},
\]
where $m, n$ are nonnegative integers.
\end{lem}
\begin{proof}
The proof is similar to the $Q^+$ case,  first for the case $\gamma \in (-\frac 3 2, 1]$ we have
\begin{equation*}
\begin{aligned}
|( Q^-(f, g), h \langle v \rangle^{2k})  | \le &   \int_{\R^3} \int_{\R^3}\int_{\mathbb{S}^2} |v - v_* |^\gamma b(\cos \theta) |f(v_*)| | g(v)| |h(v) | \langle v \rangle^{2k}   dv dv_* d \sigma
\\
\le& \left(\int_{\R^3} \int_{\R^3}\int_{\mathbb{S}^2}  b(\cos \theta ) \langle v \rangle^{2k+\gamma}  \langle v_* \rangle^6     |f(v_*)|^2 | g(v)|^2 dv dv_* d \sigma   \right)^{1/2}
\\
&\left(\int_{\R^3} \int_{\R^3}\int_{\mathbb{S}^2} |v - v_* |^{2\gamma} b(\cos \theta ) \frac {1} { \langle v_* \rangle^6 }   \langle v \rangle^{2k-\gamma} |h(v)|^2  dv dv_* d \sigma   \right)^{1/2},
\end{aligned}
\end{equation*}
it is easily seen that
\[
\int_{\R^3} \int_{\R^3}\int_{\mathbb{S}^2}  b(\cos \theta ) \langle v \rangle^{2k+\gamma}  \langle v_* \rangle^6     |f(v_*)|^2 | g(v)|^2 dv dv_* d \sigma \le C_k \Vert f \Vert_{L^2_3}^2 \Vert g \Vert_{L^2_{k+\gamma/2}}^2,
\]
since $2\gamma \in (-3, 2]$ we have
\[
\int_{\R^3} \int_{\R^3}\int_{\mathbb{S}^2} |v - v_* |^{2\gamma} b(\cos \theta ) \frac {1} { \langle v_* \rangle^6 }   \langle v \rangle^{2k-\gamma} |h(v)|^2  dv dv_* d \sigma \le C_k \Vert h \Vert_{L^2_{k+\gamma/2}}^2,
\]
so  the case $\gamma \in (-\frac 3 2, 1]$ is thus proved. For the case $\gamma \in (-\frac 5 2,  -\frac 3 2 ] $,  we have $-3< \gamma-\frac 1 2 <0, -2 < \gamma + \frac 1 2 <0$, hence
\begin{equation*}
\begin{aligned}
|( Q^- (f, g), h \langle v \rangle^{2k})  | \le &    \left(\int_{\R^3} \int_{\R^3}\int_{\mathbb{S}^2} |v - v_*|^{\gamma+\frac 1 2} b(\cos \theta ) \langle v \rangle^{2k-\frac 1  2 }  \langle v_* \rangle^4     |f(v_*)|^2 | g(v)|^2 dv dv_* d \sigma   \right)^{1/2}
\\
&\left(\int_{\R^3} \int_{\R^3}\int_{\mathbb{S}^2} |v - v_*|^{\gamma-\frac 1 2} b(\cos \theta ) \frac {1} { \langle v_* \rangle^4}   \langle v \rangle^{2k  + \frac 1 2 } |h(v)|^2  dv dv_* d \sigma   \right)^{1/2},
\end{aligned}
\end{equation*}
since  $-3< \gamma-\frac 1 2 <0$, it is easily seen that
\[
\int_{\R^3} \int_{\R^3}\int_{\mathbb{S}^2} |v - v_*|^{\gamma-\frac 1 2} b(\cos \theta ) \frac {1} { \langle v_* \rangle^4}   \langle v \rangle^{2k  + \frac 1 2 } |h(v)|^2  dv dv_* d \sigma  \le C_k \Vert h \Vert_{L^2_{k+\gamma/2}}^2,
\]
since $-2< \gamma-\frac 1 2 <0$, by Lemma \ref{L26} we have
\[
\int_{\R^3} \int_{\R^3}\int_{\mathbb{S}^2} |v - v_*|^{\gamma+\frac 1 2} b(\cos \theta ) \langle v \rangle^{2k-\frac 1  2 }  \langle v_* \rangle^4     |f(v_*)|^2 | g(v)|^2 dv dv_* d \sigma  
\lesssim \min_{m + n = 1}  \{ \Vert f \Vert_{H^m_4}^2  \Vert g \Vert_{H^n_{k+\gamma/2} }^2 \},
\]
so the case $\gamma \in (-\frac 5 2, -\frac 3 2]$ is proved. For the case $\gamma \in (-3, -\frac 5 2] $ we have
\begin{equation*}
\begin{aligned}
|( Q^-(f, g), h \langle v \rangle^{2k})  | \le &   \left(\int_{\R^3} \int_{\R^3}\int_{\mathbb{S}^2} |v - v_* |^{\gamma} b(\cos \theta ) \langle v \rangle^{2k}  \langle v_* \rangle^4     |f(v_*)|^2 | g(v)|^2 dv dv_* d \sigma   \right)^{1/2}
\\
&\left(\int_{\R^3} \int_{\R^3}\int_{\mathbb{S}^2} |v - v_* |^{\gamma} b(\cos \theta ) \frac {1} { \langle v_* \rangle^4 }   \langle v \rangle^{2k} |h(v)|^2  dv dv_* d \sigma   \right)^{1/2},
\end{aligned}
\end{equation*}
it is easily seen that
\[
\int_{\R^3} \int_{\R^3}\int_{\mathbb{S}^2} |v - v_* |^{\gamma} b(\cos \theta ) \frac {1} { \langle v_* \rangle^4 }   \langle v \rangle^{2k} |h(v)|^2  dv dv_* d \sigma   \le C_k \Vert h \Vert_{L^2_{k+\gamma/2}}^2,
\]
and by Lemma \ref{L26}
\[
\int_{\R^3} \int_{\R^3}\int_{\mathbb{S}^2} |v - v_* |^{\gamma} b(\cos \theta ) \langle v \rangle^{2k}  \langle v_* \rangle^4     |f(v_*)|^2 | g(v)|^2 dv dv_* d \sigma  \lesssim \min_{ m + n =2} \{ \Vert f \Vert_{H^m_4}^2  \Vert g \Vert_{H^n_{k+\gamma/2} }^2 \},
\]
the lemma is proved by gathering all the terms together. 
\end{proof}

\begin{cor}\label{C35}
For any $\gamma \in(-3, 1]$,  for smooth function $f, g, h$ , for any $k\ge 4$ we have
\[
|( Q(f, g), h \langle v \rangle^{2k}) | \le C_k \min_
{ m + n  = N-2} \{ \Vert f \Vert_{H^m_{k+\gamma/2}} \Vert g \Vert_{H^n_4}  +  \Vert g \Vert_{H^m_{k+\gamma/2}} \Vert f \Vert_{H^n_4}   \} \Vert h\Vert_{L^2_{k+\gamma/2}},
\]
with $m, n$  nonnegative integers.
\end{cor}

\begin{proof}
By taking $m = m_1, n=n_1$ in Lemma \ref{L33} and Lemma \ref{L34} the proof is finished. 
\end{proof}

\begin{cor}\label{C36}
For any $-3 < \gamma \le 1$, for any $|\beta| \le 4$, for any $f, g$ smooth we have
\[
|(Q(f, \partial_\beta \mu  ), g \langle  v \rangle^{2k} ) | + |(Q(\partial_\beta \mu, f  ), g \langle  v \rangle^{2k} ) |\le C_k \Vert f \Vert_{L^2_{k+\gamma/2}}  \Vert g \Vert_{L^2_{k+\gamma/2}},
\]
for some constant $C_k >0$. 
\end{cor}
\begin{proof} By taking $ m =  n_1 =0$ in Lemma \ref{L33} and Lemma \ref{L34} we have

\begin{equation*}
\begin{aligned}
&| (Q(f, \partial_\beta \mu  ), g \langle  v \rangle^{2k} ) |  + |(Q(\partial_\beta \mu, f  ), g \langle  v \rangle^{2k} ) | 
\\
\le & C_k \Vert f \Vert_{L^2_{k+\gamma/2}} \Vert \partial_\beta \mu  \Vert_{H^n_5}  \Vert g\Vert_{L^2_{k+\gamma/2}}+ C_k \Vert \partial_\beta \mu  \Vert_{H^{m_1}_{k+\gamma/2}} \Vert f \Vert_{L^2_5}  \Vert g\Vert_{L^2_{k+\gamma/2}}   \le  C_k \Vert f \Vert_{L^2_{k+\gamma/2}}  \Vert g \Vert_{L^2_{k+\gamma/2}},
\end{aligned}
\end{equation*}
so the proof is ended.
\end{proof}

The estimate for the exponential weight case is similar. 

\begin{lem}\label{L37}
For any $\gamma \in(-3, 1]$,  for smooth function $f, g, h$ , for any $ k \in \R, a > 0 ,b \in (0, 2) $ we have
\begin{equation*}
\begin{aligned}
|( Q(f, g ), h \langle v \rangle^{2k} e^{2a\langle v \rangle^b}) | \le& C_{k, a, b}   \min_{ m + n  = N-2}  \{ \Vert f \Vert_{H^m_{k+\gamma/2, a, b}} \Vert g \Vert_{H^n_{k - 96 , a, b}}  + \Vert g  \Vert_{H^{m}_{k+\gamma/2, a, b}} \Vert f  \Vert_{H^{n}_{k - 96, a, b}}   \} \Vert h  \Vert_{L^2_{k+\gamma/2, a, b}},
\end{aligned}
\end{equation*}
with $m, n$ nonnegative integers. 
\end{lem}

\begin{proof} For the case $\gamma \in (-\frac 3 2, 1]$, by Lemma \ref{L216} we have
\[
\langle v \rangle^{k - 100 } e^{a \langle v \rangle^b}  \le C_{k, a, b}\langle v_*' \rangle^{k - 100   } e^{a \langle v_*' \rangle^b}  \langle v' \rangle^{k - 100 } e^{a \langle v' \rangle^b} ,
\] 
which implies
\begin{equation*}
\begin{aligned}
&|( Q^+(f, g), h e^{2 a \langle v \rangle^b} \langle v \rangle^{2k})  
\\
\le &  \int_{\R^3} \int_{\R^3}\int_{\mathbb{S}^2} |v - v_* |^\gamma b(\cos \theta) |f(v_*')| | g(v')|    |h(v) | \langle v \rangle^{2k} e^{2 a \langle v \rangle^b } dv dv_* d \sigma
\\
\le &  C_{k, a, b}\int_{\R^3} \int_{\R^3}\int_{\mathbb{S}^2} |v - v_* |^\gamma b(\cos \theta) |f(v_*')| \langle v_*' \rangle^{k - 100  } e^{a \langle v_*' \rangle^b} | g(v')|  \langle v' \rangle^{k - 100  } e^{a \langle v' \rangle^b}    |h(v) | \langle v \rangle^{k + 100 } e^{a \langle v \rangle^b } dv dv_* d \sigma
\\
 = & (Q^+ ( f \langle v \rangle^{k - 100 } e^{a \langle v \rangle^b}, g  \langle v \rangle^{k- 100 } e^{a \langle v \rangle^b}     ), ( h \langle v \rangle^{k - 100} e^{a \langle v \rangle^b}  ) \langle v \rangle^{200}),
\end{aligned}
\end{equation*}
 taking $ k = 100$  in Lemma \ref{L33} the $Q^+$ part is proved. Similarly for the $Q^-$ term we have
\[
( Q^-(f, g), h e^{2\alpha \langle v \rangle^\beta} \langle v \rangle^{2k}) = ( Q^-(f, g \langle v \rangle^{k-100 } e^{\alpha \langle v \rangle^\beta}  ), (h \langle v \rangle^{k - 100 } e^{\alpha \langle v \rangle^\beta} )\langle v \rangle^{200}),
\]
we conclude by taking $k = 100$ in Corollary \ref{C35}.
\end{proof}

\begin{cor}\label{C38}
For any $-3 < \gamma \le 1$, for any $|\beta| \le 4$, for any $f, g$ smooth we have
\[
|(Q(f, \partial_\beta \mu  ), g \langle  v \rangle^{2k}  e^{a \langle v \rangle^b}   ) | + |(Q(\partial_\beta \mu, f  ), g \langle  v \rangle^{2k}  e^{\alpha \langle v \rangle^\beta}  ) |\le  C_{k, a, b}  \Vert f \Vert_{L^2_{k+\gamma/2, a, b}}  \Vert g \Vert_{L^2_{k+\gamma/2, a, b}},
\]
for some constant $C_{k, a, b} >0$. 
\end{cor}

\begin{proof}
The proof is similar as Lemma \ref{C36} thus omitted. 
\end{proof}

Next we come to prove the linearized part for the exponential weight case.
\begin{lem}\label{L39}

For any $-3 <\gamma \le 1$, for any $k \in \R, a>0, b \in (0, 2) $ and  $f$ smooth, we have
\[
|(Q^+ (\mu, f), f \langle v \rangle^{2k} e^{2a \langle v \rangle^b}   ) | + |(Q^+ (f, \mu), f \langle v \rangle^{2k} e^{2a \langle v \rangle^b }   ) | \le  C_{k, a, b} \Vert f  \Vert_{L^2_{k+\gamma/2, a , b}}\Vert f \Vert_{L^2_{k+\gamma/2-b (\gamma+3)/ 4, a, b}},
\]
for some constant $C_k>0$.
\end{lem}

\begin{proof}
For the first term, by Lemma \ref{L216} we have
\[
\langle v \rangle^{k} e^{\frac 1 2 a \langle v \rangle^b} \lesssim \langle v_*' \rangle^{k} e^{\frac 1 2 a \langle v_*' \rangle^b}  \langle v' \rangle^{k} e^{\frac 1 2 a \langle v' \rangle^b} ,
\]
together with Lemma \ref{L212}   we have
\begin{equation*}
\begin{aligned}
&|( Q^+(f, \mu), f \langle v \rangle^{2k} e^{2a  \langle v \rangle^b}   )  | 
\\
\lesssim &   \int_{\R^3} \int_{\R^3}\int_{\mathbb{S}^2} |v - v_* |^\gamma b(\cos \theta) | \mu (v_*')| |f (v')| |f(v) | \langle v \rangle^{2k} e^{2 a \langle v \rangle^b}   dv dv_* d \sigma
\\
\lesssim &   \int_{\R^3} \int_{\R^3}\int_{\mathbb{S}^2} |v - v_* |^\gamma b(\cos \theta) e^{-\frac 1 2 |v_*' |^2} \langle v_*' \rangle^{k} e^{\frac 1 2 a  \langle v_*'  \rangle^b}   |f (v')| \langle v' \rangle^{k} e^{\frac 1 2 a \langle v '\rangle^b}   |f(v) |   \langle v \rangle^{k} e^{\frac 3 2 a \langle v \rangle^b }   dv dv_* d \sigma
\\
\lesssim & \left(\int_{\R^3} \int_{\R^3}\int_{\mathbb{S}^2} |v - v_* |^\gamma b(\cos \theta )  e^{-\frac 1 2 |v_*'|^2} e^{a \langle v_*' \rangle^b } \langle v_*' \rangle^{2k}  e^{2 a \langle v' \rangle^b } \langle v ' \rangle^{2k}    |f(v')|^2  dv dv_* d \sigma   \right)^{1/2}
\\
&\left(\int_{\R^3} \int_{\R^3}\int_{\mathbb{S}^2} |v - v_* |^\gamma b(\cos \theta ) \frac { e^{a  \langle v \rangle^b } } {  e^{a \langle v '\rangle^b }}   e^{-\frac 1 2 |v_*'|^2} e^{2a  \langle v \rangle^b } \langle v \rangle^{2k} | f (v)|^2  dv dv_* d \sigma   \right)^{1/2}
\\
\lesssim &  C_{k, a, b } \Vert f  \Vert_{L^2_{k+\gamma/2, a, b }}\Vert f \Vert_{L^2_{k+\gamma/2 - b (\gamma+3)/ 4, a, b}}, 
\end{aligned}
\end{equation*}
the second term can be proved  by changing  $v'$ and $v_* '$, so the proof is thus finished. 
\end{proof}

\begin{lem}\label{L310}
For any $-3 <\gamma \le 1$, for any $k\ge 0, a>0, b  \in (0, 2) $, $f$ smooth, we have
\[
 - (Q^- (\mu, f), f \langle v \rangle^{2k} e^{2 a \langle v \rangle^b}   )  -  (Q^- (f, \mu), f \langle v \rangle^{2k} e^{2 a \langle v \rangle^b}   )  \le  -C_1 \Vert f  \Vert_{L^2_{k + \gamma/2, a, b}}^2 + C_{k, a, b} \Vert f \Vert_{L^2_3}^2,
\]
for some constants $C_1, C_{k, a, b} > 0$. 
\end{lem}

\begin{proof}
It is esaily seen that
\[
( - Q^- (\mu, f), f \langle v \rangle^{2k}e^{2 a \langle v \rangle^b} )  = - \int_{\mathbb{S}^2} b(\cos \theta)    d\sigma \int_{\R^3}\int_{\R^3}  |v-v_*|^\gamma \mu(v_*)  |f|^2 \langle v \rangle^{2k}   e^{2 a \langle v \rangle^b }    dv dv_* ,
\]
and  by Lemma \ref{L28} we have
\begin{equation*}
\begin{aligned} 
|(Q^-(f, \mu ), f \langle \cdot \rangle^{2k} e^{2 a \langle v \rangle^b} )| = & \left| \int_{\R^3}\int_{\R^3}\int_{\mathbb{S}^2} |v-v_*|^\gamma  b(\cos \theta) f(v_*) \mu(v)   f(v) \langle v \rangle^{2k}     e^{2 a  \langle v \rangle^b }  dv dv_* d \sigma \right|
\\
\lesssim &\int_{\R^3}\int_{\R^3}\int_{\mathbb{S}^2} |v-v_*|^\gamma  b(\cos \theta) |f(v_*)|  |f(v)| dv dv_* d \sigma
\lesssim  \Vert f \Vert_{L^2_{3}}^2.
\end{aligned}
\end{equation*}
Gathering the two terms, the lemma is thus proved. 
\end{proof}

\begin{cor}\label{C311}
For any $-3 <\gamma \le 1$, for any $k \in   \R , a > 0, b \in (0, 2) $ and  $f$ smooth, we have
\[
(Q (\mu, f), f \langle v \rangle^{2k} e^{2 a \langle v \rangle^b }   )  + (Q (f, \mu), f \langle v \rangle^{2k} e^{2 a \langle v \rangle^b }   )  \le  -C_1 \Vert f \Vert_{L^2_{k + \gamma/2, a, b }}^2 + C_{k, a, b}\Vert f \Vert_{L^2}^2,
\]
for some constants $C_1, C_{k, a , b} > 0$. 
\end{cor}

\begin{lem}\label{L312}
For any $-3 <\gamma \le 1$,   $f$ smooth, define 
\[
\bar{L} f = - v \cdot \nabla_x f+  Q(\mu, f  ) +  Q(f, \mu),
\]
then for any $k > 4$ we have
\begin{equation}
\label{poly weight estimate}
(\bar{L} f, f\langle v \rangle^{2k})_{L^2_{x, v}} \le -  C \Vert f \Vert_{L^2_x L^2_{k+\gamma/2}}^2 + C_k  \Vert f \Vert_{L^2_x L^2_v}^2, 
\end{equation}
for some constants $C, C_k>0$. For any $k \in \R, a >0, b  \in (0, 2) $ we have
\begin{equation}
\label{exp weight estimate}
(\bar{L} f , f \langle v \rangle^{2k} e^{2a \langle v \rangle^b }   )_{L^2_{x, v}}  \le  -C_1 \Vert f  \Vert_{L^2_x L^2_{k + \gamma/2, a, b }}^2 + C_{k, a, b}\Vert f \Vert_{L^2_x L^2_v}^2,
\end{equation}
for some constants $C_1, C_{k, a, b }>0$. As a consequence for the solution to the inhomogeneous Boltzmann  equation 
\[
\partial_t f =\bar{L} f   = - v \cdot \nabla_x f + Q(f, \mu) + Q(\mu, f) , \quad f|_{t=0} = f_0.
\]
If $\gamma \in [0, 1]$, for any $k >4$ we have
\[
\Vert f - P f\Vert_{L^2_x L^2_k} \lesssim e^{-\lambda t}   \Vert f_0 - P f_0 \Vert_{L^2_x L^2_k},
\]
for some constant $\lambda>0$.  If $\gamma \in (-3, 0)$, for any $4 < k_0 <  k$  we have
\[
\Vert f - P f\Vert_{L^2_x L^2_{k_0}} \lesssim \langle  t \rangle^{-\frac {k -k_* } {|\gamma| }}\Vert f_0 - P f_0 \Vert_{L^2_x L^2_k} ,\quad \forall k_* \in (k_0, k).
\]
For the exponential weight, if $\gamma \in [0, 1]$, for any $k \in \R, 0< a, b \in (0, 2)$ we have
\[
\Vert f - P f\Vert_{L^2_x L^2_{k, a , b}} \lesssim  e^{- \lambda t }\Vert f_0 - P f_0 \Vert_{L^2_x L^2_{k, a, b}},
\]
for some constant $\lambda>0$.  If $\gamma \in (-3, 0)$, for any $k \in \R, 0 < a_0 < a, b \in (0, 2)$ we have
\[
\Vert f - P f\Vert_{L^2_x L^2_{k, a_0 , b}} \lesssim  e^{- \lambda t^{\frac b {b - \gamma}  } }\Vert f_0 - P f_0 \Vert_{L^2_x L^2_{k, a, b}},
\]
for some constant $\lambda > 0$.
\end{lem}

\begin{proof}  The exponential case \eqref{exp weight estimate} just follows from Corollary \ref{C311} above. We only prove the  polynomial case \eqref{poly weight estimate}.  
By Lemma \ref{L31} and \ref{L32} we have
\[
(Lf, f \langle v \rangle^{2k}) = -\Vert  b(\cos \theta) (\sin^2 \frac \theta 2  -\sin^{k- 2} \frac \theta 2 )\Vert_{L^1_\theta} \Vert f \Vert_{L^2_{k+\gamma/2, *} } + C_k \Vert f \Vert_{L^2_{k+\gamma/2-1/2}}^2,
\]
if $k >4$, then 
\[
\sin^2  \frac \theta 2  - \sin^{k -2   } \frac \theta 2 \ge \sin^2  \frac \theta 2 (1 - \sin^{k -4 }  \frac \theta 2) >0,
\]
so the polynomial case follows by interpolation. Then we come to prove the convergence rate. By combing the results in \cite{SG2, SG, K, DYZ2} we have
\[
\Vert (I-P) S_L(t) f \Vert_{L^2_x L^2(\mu^{-1/2})} \lesssim e^{-\lambda t}\Vert (I - P ) S_L(t) f \Vert_{L^2_x L^2(\mu^{-1/2})} , \quad \hbox{if} \quad \gamma \ge 0,
\]
and 
\[
\Vert (I-P) S_L(t) f \Vert_{L^2_xL^2(\mu^{-1/2})} \lesssim e^{-\lambda t^{\frac {2} {2-\gamma} }}\Vert (I - P) S_L(t) f \Vert_{L^2_x L^2(\mu^{-3/4})} , \quad \hbox{if} \quad -3 <  \gamma < 0.
\]
Define two operators $L=A+B$ with $A=M\chi_{R_1}$ and $B=L-M\chi_{R_1}$ for some $M, R_1>0$ large, where $\chi_{R_1}$ is the truncation function in ball with center zero and radius $R_1>0$. Denote $S_L$ and $S_B$ semigroups generated by $L$ and $B$ respectively. Then if $M, R_1$ is large we have
\[
(Bf, f)_{L^2_x L^2_k} \lesssim - \Vert f \Vert_{L^2_xL^2_{k+\gamma/2}} , \quad \Vert A f  \Vert_{L^2_xL^2(\mu^{-3/4})} \lesssim  \Vert f \Vert_{L^2_x L^2_v}. 
\]
For the case $\gamma \ge 0$, for the polynomial case we have
\[
\Vert S_B(t) f \Vert_{L^2_x L^2_k} \lesssim e^{-\lambda t} \Vert  f \Vert_{L^2_x L^2_k} .
\]
By Duhamel's principle we have
\begin{equation*}
\begin{aligned} 
\Vert (I - P) S_L(t) \Vert_{L^2_x L^2_k \to L^2_x L^2_k } \lesssim  & \int_0^t \Vert (I - P) S_L(t-s) \Vert_{L^2_x L^2(\mu^{-1/2}) \to L^2_x L^2(\mu^{-1/2}) }  \Vert A \Vert_{L^2_x L^2_k \to L^2_x L^2(\mu^{-1/2}) }  \Vert S_B(s) \Vert_{L^2_xL^2_k \to L^2_xL^2_k }  ds
\\
&+ \Vert (I - P) S_B(t) \Vert_{L^2_xL^2_k \to L^2_xL^2_k } \lesssim e^{-\lambda t },
\end{aligned}
\end{equation*}
and for the exponential weight case the proof for $\gamma \ge 0$ is the same. For the case $\gamma <0$, we first prove the polynomial case, on one hand, we have
\[
\Vert S_B(t)f \Vert_{L^2_xL^2_{k}} \lesssim \Vert f\Vert_{L^2_xL^2_k}.
\]
On the other hand, for any $k_0 < k$, for any $R>0$ we have
\[
\frac{d}{dt} \|S_B(t)f\|^2_{L^2_x L^{2}_{k_0}}\le -c\|S_B(t)f\|^2_{L^2_x L^2_{k_0+\gamma/2}} \leq  -c\langle R \rangle^{\gamma}\|S_B(t)f\|^2_{L^2_x L^2_{k_0}}+C\langle R\rangle^{ - 2(k-k_0)+\gamma}\|S_B(t)f\|^2_{L^2_x L^2_{k}},
\]
where we use the following interpolation
\[
\langle R\rangle^{\gamma}\|f\|^2_{L^2_x L^2_{k_0}}\leq \| f\|^2_{L^2_x L^2_{k_0+\gamma/2}}+\langle R \rangle^{ - 2(k-k_0)+\gamma}\|f\|^2_{L^2_x L^2_{k}}.
\]
Integrating the differential inequality, we obtain
\begin{equation*}
\begin{aligned} 
\|S_B(t)f\|^2_{L^2_x L^2_{k_0}}\lesssim  e^{-c\langle R\rangle^\gamma t}\|f\|^2_{L^2_x L^2_{k_0}}+\langle R\rangle^{- 2 ( k - k_0)}\| f \|^2_{L^2_x L^2_k} \lesssim & \inf_{R>0}(e^{-c\langle R\rangle^\gamma t}+\langle R\rangle^{ - 2(k-k_0)})\|f\|^2_{L^2_x L^2_k} 
\\
\lesssim & \langle t \rangle^{-\frac {2(k-k_*)} {|\gamma|}}  \|f\|^2_{L^2_x L^2_k} , \quad \forall k_* \in (k_0, k),
\end{aligned}
\end{equation*}
where we choose $\langle R\rangle =\langle t\rangle^{-1/\gamma} [\log(1+t)^{-\frac 2 c(k - k_0 )} ]^{1/\gamma}$.  Moreover, thanks to Duhamel's formula
\begin{equation*}
\begin{aligned} 
\Vert (I - P) S_L(t) \Vert_{L^2_xL^2_k \to L^2_x L^2_{k_0} } \lesssim& \Vert (I - P) S_B(t) \Vert_{L^2_x L^2_k \to L^2_xL^2_{k_0} }  
\\
&+  \int_0^t \Vert (I - P) S_L(t-s) \Vert_{L^2_x L^2(\mu^{-3/4}) \to L^2_x L^2(\mu^{-1/2}) }  \Vert A \Vert_{L^2_x L^2_k \to L^2_x L^2(\mu^{-3/4}) }  \Vert S_B(s) \Vert_{L^2_x L^2_k \to L^2_x L^2_{k_0} }  ds 
\\
\lesssim & \langle t \rangle^{-\frac {k-k_*} {|\gamma|}} , \quad \forall k_* \in (k_0, k),
\end{aligned}
\end{equation*}
so the proof for the polynomial case is thus finished. For the exponential weight case,  if $M, R_1$ is large we have
\[
(Bf, f)_{L^2_x L^2_{k, a, b}} \le  -c \Vert  f \Vert_{L^2_x L^2_{k +\gamma/2, a, b}}^2  , \quad \Vert A f  \Vert_{L^2_x L^2(\mu^{-3/4})} \lesssim  \Vert f \Vert_{L^2_x L^2_v}.
\]
So first we have
\[
\Vert S_B(t)f \Vert_{L^2_x L^2_{k, a, b}} \lesssim \Vert f\Vert_{L^2_x L^2_{k, a , b}}. 
\]
On the other hand, for any $a_0 < a$, for any constant $R>0$ we have
\[
\frac{d}{dt} \|S_B(t)f\|_{L^2_x L^2_{k , a_0 , b}}^2 \le -c\|S_B(t)f\|_{L^2_x L^2_{k +\gamma/2, a_0 , b}}^2\leq  -c\langle R \rangle^{\gamma}\|S_B(t)f\|_{L^2_x L^2_{k , a_0 , b}}^2 +   \langle R \rangle^\gamma  e^{-(a - a_0) \langle  R \rangle^b }  \|S_B(t)f\|_{L^2_x L^2_{k , a , b}}^2,
\]
where we use the following interpolation
\[
\langle R\rangle^{\gamma}\|f\|_{L^2_x L^2_{k , a_0 , b}}^2 \le \| f\|_{L^2_x L^2_{k +\gamma/2 , a_0 , b}}^2  +   \langle R \rangle^\gamma  e^{-( a - a_0) \langle  R \rangle^b } \|f\|_{L^2_x L^2_{k , a , b}}^2,
\]
we can deduce
\begin{equation*}
\begin{aligned} 
\|S_B(t)f\|^2_{L^2_x L^2_{k, a, b}}\lesssim& e^{-c\langle R\rangle^\gamma t}\|f\|^2_{L^2_x L^2_{k, a, b}} + e^{-( a- a_0) \langle R \rangle^b }  \| f \|^2_{L^2_x L^2_{k, a, b}} 
\\
\lesssim & \inf_{R>0}(e^{-c\langle R\rangle^\gamma t} +  e^{-( a - a_0) \langle R \rangle^b } )\|f\|^2_{L^2_x L^2_{k, a, b}} \lesssim e^{- \lambda t ^{ \frac {b} {b -\gamma}} } \|f\|^2_{L^2_x L^2_{k, a, b}} ,
\end{aligned}
\end{equation*}
for some constant $\lambda>0$, where we choose $\langle R \rangle = t^{\frac {1} {b- \gamma}}$.  Using Duhamel's formula again we have
\begin{equation*}
\begin{aligned} 
&\Vert (I - P) S_L(t) \Vert_{L^2_x L^2_{k, a, b} \to L^2_x L^2_{k, a_0, b} } 
\\
\lesssim& \int_0^t \Vert (I - P) S_L(t-s) \Vert_{L^2_x L^2(\mu^{-3/4}) \to L^2_x L^2(\mu^{-1/2}) }  \Vert A \Vert_{L^2_x L^2_{k, a_0, b} \to L^2_x L^2(\mu^{-3/4}) } \Vert S_B(s) \Vert_{L^2_x L^2_{k, a, b} \to L^2_x L^2_{k, a_0, b } } ds  
\\
&+  \Vert (I - P) S_B(t) \Vert_{ L^2_xL^2_{k, a, b} \to L^2_x L^2_{k, a_0, b} }  
\lesssim   e^{- \lambda t ^{ \frac {b } { b  -\gamma}} },
\end{aligned}
\end{equation*}
so the proof is thus finished. 
\end{proof}

\section{Estimates for the inhomogeneous equation}\label{section4}

In this section we prove the estimates for the inhomogeneous Boltzmann equation. 
Recall $N$ is defined in \eqref{N}, $w(\alpha, \beta)$ is defined in \eqref{weight function}, $X_k, X_{k, a, b}$ is defined in \eqref{X k 1} and \eqref{X k 2}, and  $Y_k:=X_{k+\gamma/2}, Y_{k, a, b} :=X_{k+\gamma/2, a, b}$. We first prove an estimate for the nonlinear term. 

\begin{lem}\label{L41}
Suppose $f, g, h$ smooth function. For the polynomial weight case,  for any $k \ge 4$ large, if $\gamma \in (-\frac 3 2, 1]$, then for any indices $|\alpha| \le 2$, we have
\[
|(\partial^\alpha Q(f, g), \partial^\alpha h \langle v \rangle^{2k})_{L^2_{x, v}} | \le C_k \Vert f \Vert_{H^2_xL^2_4} \Vert g  \Vert_{Y_k} \Vert h \Vert_{Y_k} + C_k\Vert g \Vert_{H^2_xL^2_4} \Vert f  \Vert_{Y_k} \Vert h \Vert_{Y_k},
\]
for some constant $C_k>0$. If $\gamma \in (-3, -\frac 3 2]$, for any indices $|\alpha| + |\beta| \le N$ we have
\[
|(\partial^\alpha_\beta Q(f, g), \partial^\alpha_\beta h w^2 (\alpha, \beta) )_{L^2_{x, v}} | \le C_k \Vert f  \langle v \rangle^4 \Vert_{H^N_{x, v} } \Vert g  \Vert_{Y_k} \Vert h \Vert_{Y_k} + C_k\Vert g \langle v \rangle^4 \Vert_{H^N_{x, v}} \Vert f  \Vert_{Y_k} \Vert h \Vert_{Y_k},
\]
for some constant $C_k>0$. For the exponential weight case, for any $k \in \R, a >0, b \in (0, 2)$, if $\gamma \in (-\frac 3 2, 1]$, then for any indices $|\alpha| \le 2$, we have
\[
|(\partial^\alpha Q(f, g), \partial^\alpha h \langle v \rangle^{2k} e^{2a \langle v \rangle^b}  )_{L^2_{x, v}} | \le C_k \Vert f \Vert_{X_{k, a, b}} \Vert g  \Vert_{Y_{k, a, b}} \Vert h \Vert_{Y_{k, a, b}} + C_k\Vert g \Vert_{X_{k, a, b}} \Vert f  \Vert_{Y_{k, a, b}} \Vert h \Vert_{Y_{k, a, b}},
\]
for some constant $C_{k, a, b}>0$. If $\gamma \in (-3, -\frac 3 2]$ for any indices $|\alpha| + |\beta| \le N$ we still have
\[
|(\partial^\alpha_\beta Q(f, g), \partial^\alpha_\beta h w^2 (\alpha, \beta) e^{2a \langle v \rangle^b} )_{L^2_{x, v}} | \le C_{k,a,  b} \Vert f  \Vert_{X_{k, a, b}  } \Vert g  \Vert_{Y_{k,a ,b}} \Vert h \Vert_{Y_{k, a, b}} + C_{k,a ,b} \Vert g  \Vert_{X_{k, a, b} } \Vert f  \Vert_{Y_{k, a, b}} \Vert h \Vert_{Y_{k, a, b}}, 
\]
for some constant $C_{k, a, b}>0$.
\end{lem}

\begin{proof} We first prove the polynomial case, for the case $\gamma \in (-\frac 3 2, 1]$, by
\[
(\partial^\alpha Q(f, g), \partial^\alpha h \langle v \rangle^{2k})_{L^2_{x, v}}= \sum_{\alpha_1 \le \alpha} ( Q(\partial^{\alpha_1} f, \partial^{\alpha - \alpha_1} g), \partial^{\alpha} h \langle v \rangle^{2k})_{L^2_{x, v}},
\]
together with Lemma \ref{C35} we have
\begin{equation*}
\begin{aligned} 
& ( Q(\partial^{\alpha_1} f, \partial^{\alpha - \alpha_1} g), \partial^{\alpha} h \langle v \rangle^{2k})_{L^2_{x, v}}
\\
\le &C_k \int_{\T^3}    \Vert \partial^{\alpha - \alpha_1} f \Vert_{L^2_{4}} \Vert \partial^{\alpha_1} g \Vert_{L^2_{ k+\gamma/2 }}\Vert \partial^{\alpha}  h \Vert_{L^2_{ k + \gamma/2}} + \Vert \partial^{\alpha_1} g \Vert_{L^2_{4}} \Vert \partial^{\alpha - \alpha_1} f \Vert_{L^2_{ k + \gamma/2}}\Vert \partial^{\alpha} h \Vert_{L^2_{ k + \gamma/2}} dx.
\end{aligned}
\end{equation*}
By symmetry we only estimate the first term, we easily compute
\begin{equation*}
\begin{aligned} 
\int_{\T^3}  \Vert \partial^{\alpha - \alpha_1} f \Vert_{L^2_{4}} \Vert \partial^{\alpha_1} g \Vert_{L^2_{ k+\gamma/2 }}\Vert \partial^{\alpha}  h \Vert_{L^2_{ k + \gamma/2}} dx
\lesssim&  \Vert \partial^{\alpha - \alpha_1} f \Vert_{H^{2-|\alpha -\alpha_1| }_x L^2_{4}} \Vert \partial^{\alpha_1} g \Vert_{H^{|\alpha - \alpha_1|}_x L^2_{ k+\gamma/2 }}\Vert \partial^{\alpha}  h \Vert_{L^2_x L^2_{ k + \gamma/2}} 
\\
\lesssim & \Vert  f \Vert_{H^2_x L^2_{4}} \Vert g \Vert_{H^2_x L^2_{ k+\gamma/2 }}\Vert   h \Vert_{H^2_x  L^2_{ k + \gamma/2}} ,
\end{aligned}
\end{equation*}
so the case $\gamma \in (-\frac 3 2 , 1]$ is proved. We then prove the case $ (-3, - \frac 3 2]$, we only prove the case $\gamma \in (-3, - \frac 5 2]$, the case $\gamma \in (- \frac 5 2, -\frac  3 2] $ can be proved similarly. First we have
\[
(\partial^\alpha_\beta Q(f, g), \partial^\alpha h w^2(\alpha, \beta))_{L^2_{x, v}}= \sum_{\alpha_1 \le \alpha, \beta_1 \le \beta} ( Q(\partial^{\alpha_1}_{\beta_1} f, \partial^{\alpha - \alpha_1}_{\beta- \beta_1} g), \partial^{\alpha}_\beta h w^2(\alpha, \beta))_{L^2_{x, v}},
\]
together with Lemma \ref{C35} we have
\begin{equation*}
\begin{aligned} 
&  ( Q(\partial^{\alpha_1}_{\beta_1} f, \partial^{\alpha - \alpha_1}_{\beta- \beta_1} g), \partial^{\alpha}_\beta h w^2(\alpha, \beta))_{L^2_{x, v}}
\\
\le &C_k \int_{\T^3} \min_{m + n =  2} \{ \Vert \partial^{\alpha - \alpha_1}_{\beta - \beta_1} f \Vert_{H^m_{4}} \Vert \partial^{\alpha_1}_{\beta_1} g w(\alpha, \beta)\Vert_{H^n_{ \gamma/2 }}    \}   \Vert \partial^{\alpha}_{\beta}  h w(\alpha, \beta) \Vert_{L^2_{  \gamma/2}} dx
\\
&+C_k\int_{\T^3}  \min_{m+n =  2 }  \{ \Vert \partial^{\alpha_1}_{\beta_1} g \Vert_{H^n_{4}} \Vert \partial^{\alpha - \alpha_1}_{\beta - \beta_1} f w(\alpha, \beta)\Vert_{H^m_{  \gamma/2}}    \}   \Vert \partial^{\alpha}_{\beta}  h w(\alpha, \beta) \Vert_{L^2_{ \gamma/2}}   dx.
\end{aligned}
\end{equation*}
By symmetry we only need to prove that 
\[
 \int_{\T^3}  \min_{m+n =2} \{ \Vert \partial^{\alpha - \alpha_1}_{\beta - \beta_1} f \Vert_{H^m_{4}} \Vert \partial^{\alpha_1}_{\beta_1} g w(\alpha, \beta)\Vert_{H^n_{ \gamma/2 }}     \}  \Vert \partial^{\alpha}_{\beta}  h  w(\alpha, \beta) \Vert_{L^2_{  \gamma/2}} dx \lesssim \Vert \langle v \rangle^{4} f \Vert_{H^4_{x, v} } \Vert g \Vert_{Y_k}\Vert h \Vert_{Y_k}, \quad \forall \alpha_1 \le \alpha, \beta_1 \le \beta.
\]
First we split it into three cases, $|\alpha-\alpha_1| + |\beta-\beta_1| \le 2 $, $|\alpha-\alpha_1| + |\beta-\beta_1| =3 $, $|\alpha-\alpha_1| + |\beta-\beta_1| = 4 $. For the case $|\alpha-\alpha_1| + |\beta-\beta_1| \le 2 $, take $m=2, n=0$ we have
\begin{equation*}
\begin{aligned} 
& \int_{\T^3}  \min_{m+n =2}  \{ \Vert \partial^{\alpha - \alpha_1}_{\beta - \beta_1} f \Vert_{H^m_{4}} \Vert \partial^{\alpha_1}_{\beta_1} g w(\alpha, \beta)\Vert_{H^n_{ \gamma/2 }}    \}   \Vert \partial^{\alpha}_{\beta}  h  w(\alpha, \beta) \Vert_{L^2_{  \gamma/2}} dx 
\\
\le &   \Vert \partial^{\alpha - \alpha_1}_{\beta - \beta_1} f \Vert_{H^c_x H^2_{4}} \Vert \partial^{\alpha_1}_{\beta_1} g w(\alpha, \beta)\Vert_{H^d_x L^2_{ \gamma/2 }}\Vert \partial^{\alpha}_{\beta}  h  w(\alpha, \beta) \Vert_{L^2_x L^2_{  \gamma/2}},
\end{aligned}
\end{equation*}
with $c, d$ nonnegative integers satisfying  $c+d =2$.  Taking 
\[
c = 2- |\beta-\beta_1| -  |\alpha-\alpha_1|, \quad d = |\beta-\beta_1| + |\alpha-\alpha_1|,
\]
such that
\[
|\beta-\beta_1| + |\alpha-\alpha_1| + 2 + c =4, \quad |\alpha_1 | +|\beta_1| +d =|\alpha| +|\beta|,
\]
so first we have
\[
\Vert \partial^{\alpha - \alpha_1}_{\beta - \beta_1} f \Vert_{H^2_x H^2_{4}}  \le  \Vert \langle v \rangle^{4} f  \Vert_{H^{4}_{x, v}} , \quad \Vert \partial^{\alpha}_{\beta}  h  w(\alpha, \beta) \Vert_{L^2_x L^2_{  \gamma/2}}  \le \Vert  h \Vert_{Y_k}.
\]
Using the fact that 
\[
|\alpha_1| + |\beta_1| = |\alpha|+|\beta| \quad \&\quad |\beta_1| \le |\beta| \quad \Rightarrow \quad  w(\alpha, \beta) \le w(\alpha_1, \beta_1),
\]
we deduce 
\[
\Vert \partial^{\alpha_1}_{\beta_1} g w(\alpha, \beta)\Vert_{H^d_x L^2_{ \gamma/2 }} = \sum_{|d|\le |\beta-\beta_1| + |\alpha-\alpha_1| } \Vert \partial^{\alpha_1}_{\beta_1} g w(\alpha, \beta)\Vert_{H^d_x L^2_{ \gamma/2 }} \le \Vert g \Vert_{Y_k}.
\]
For the case $|\alpha-\alpha_1| + |\beta-\beta_1| = 3 $, by $\Vert f g\Vert_{L^2_x} \le \Vert  f \Vert_{L^6_x} \Vert g \Vert_{L^3_x}$
and taking $ m = 0,  n = 2$ we have
\begin{equation*}
\begin{aligned} 
& \int_{\T^3}  \min_{m+n =2}  \{ \Vert \partial^{\alpha - \alpha_1}_{\beta - \beta_1} f \Vert_{H^m_{4}} \Vert \partial^{\alpha_1}_{\beta_1} g w(\alpha, \beta)\Vert_{H^n_{ \gamma/2 }}   \} \Vert \partial^{\alpha}_{\beta}  h  w(\alpha, \beta) \Vert_{L^2_{  \gamma/2}} dx 
\\
\le &   \Vert \partial^{\alpha - \alpha_1}_{\beta - \beta_1} f \Vert_{L^6_x L^2_{4}} \Vert \partial^{\alpha_1}_{\beta_1} g w(\alpha, \beta)\Vert_{L^3_x H^2_{ \gamma/2 }}\Vert \partial^{\alpha}_{\beta}  h  w(\alpha, \beta) \Vert_{L^2_x L^2_{  \gamma/2}} 
\\
\le &   \Vert \partial^{\alpha - \alpha_1}_{\beta - \beta_1} f \Vert_{H^1_x L^2_{4}} \Vert \partial^{\alpha_1}_{\beta_1} g w(\alpha, \beta)\Vert_{L^3_x H^2_{ \gamma/2 }}\Vert \partial^{\alpha}_{\beta}  h  w(\alpha, \beta) \Vert_{L^2_x L^2_{  \gamma/2}}.
\end{aligned}
\end{equation*}
Since $|\alpha-\alpha_1| + |\beta-\beta_1| = 3 $, we have
\[
\Vert \partial^{\alpha - \alpha_1}_{\beta - \beta_1} f \Vert_{H^1_x H^a_{4}}  \le  \Vert \langle v \rangle^{4} f  \Vert_{H^{4}_{x, v}} , \quad \Vert \partial^{\alpha}_{\beta}  h  w(\alpha, \beta) \Vert_{L^2_x L^2_{  \gamma/2}}  \le \Vert  h \Vert_{Y_k}.
\]
For the $g$ term, since $|\alpha|+|\beta| \ge 3$, we split it into two cases $|\alpha|+|\beta| = 3$ and  $|\alpha|+|\beta| = 4$. For the case  $|\alpha-\alpha_1| + |\beta-\beta_1| = |\alpha|+|\beta| =3$ we have $|\alpha_1| =|\beta_1| =0 $, by \eqref{ab3}
\[
\max_{|\alpha|+|\beta| =3}w^2 (\alpha, \beta) \le w(0, 2) w(1, 2), 
\]
together with \eqref{L3} we have
\[
\Vert  g w(\alpha, \beta)\Vert_{L^3_x H^2_{ \gamma/2 }} \le \Vert g w(0, 2)\Vert_{L^2_x H^2_{ \gamma/2 }} + \Vert g w( 1 , 2)\Vert_{H^1_x H^2_{ \gamma/2 }}\lesssim \Vert g \Vert_{Y_k}.
\]
For the case $|\alpha-\alpha_1| + |\beta-\beta_1| = 3, |\alpha| +|\beta|=4$, this time we have $|\alpha_1| =1, |\beta_1|=0$ or $|\alpha_1| =0, |\beta_1| =1$.  For the first case by \eqref{ab4}
\[
\max_{|\alpha|+|\beta| =4}w^2 (\alpha, \beta) \le w(1, 2) w(2, 2), 
\]
together with \eqref{L3} we have
\[
\Vert  \partial^{\alpha_1}g w(\alpha, \beta)\Vert_{L^3_x H^2_{ \gamma/2 }} \lesssim \Vert g w(1, 2)\Vert_{H^1_x H^2_{ \gamma/2 }} + \Vert g w( 2 , 2)\Vert_{H^2_x H^2_{ \gamma/2 }}\lesssim \Vert g \Vert_{Y_k}.
\]
For the second case by \eqref{ab4}
\[
\max_{|\alpha|+|\beta| =4}w^2 (\alpha, \beta) \le w(0, 3) w(1, 3), 
\]
together with \eqref{L3} we deduce 
\[
\Vert  \partial_{\beta_1}g w(\alpha, \beta)\Vert_{L^3_x H^2_{ \gamma/2 }} \lesssim \Vert g w(0, 3)\Vert_{L^2_x H^3_{ \gamma/2 }} + \Vert g w( 1 , 3)\Vert_{H^1_x H^3_{ \gamma/2 }}\lesssim \Vert g \Vert_{Y_k}.
\]
Finally for  $|\alpha-\alpha_1| + |\beta-\beta_1| = 4 $, we have $|\alpha_1| = |\beta_1|=0$, taking $ m = 0, n = 2$ we have
\begin{equation*}
\begin{aligned} 
&  \int_{\T^3} \min_{m+n =2}   \{   \Vert \partial^{\alpha - \alpha_1}_{\beta - \beta_1} f \Vert_{H^m_{4}} \Vert \partial^{\alpha_1}_{\beta_1} g w(\alpha, \beta)\Vert_{H^n_{ \gamma/2 }}   \}    \Vert \partial^{\alpha}_{\beta}  h  w(\alpha, \beta) \Vert_{L^2_{  \gamma/2}} dx 
\\
\le &   \Vert \partial^{\alpha }_{\beta} f \Vert_{L^2_x L^2_{4}} \Vert  g w(\alpha, \beta)\Vert_{L^\infty_x H^2_{ \gamma/2 }}\Vert \partial^{\alpha}_{\beta}  h  w(\alpha, \beta) \Vert_{L^2_x L^2_{  \gamma/2}} 
\\
\le &  \Vert \langle v \rangle^{4} f \Vert_{H^4_{x, v}}\Vert  g w(\alpha, \beta)\Vert_{L^\infty_x H^2_{ \gamma/2 }}\Vert   h   \Vert_{Y_k} ,
\end{aligned}
\end{equation*}
by \eqref{ab4}
\[
\max_{|\alpha|+|\beta| =4}w^2 (\alpha, \beta) \le  w(1, 2)^{4/5} w(2, 2)^{6/5}, 
\]
together with \eqref{Linfty} we deduce
\[
\Vert  g w(\alpha, \beta)\Vert_{L^\infty_x H^2_{ \gamma/2 }} \lesssim \Vert g w(1, 2)\Vert_{H^1_x H^2_{ \gamma/2 }} + \Vert g w(2 , 2)\Vert_{H^2_x H^2_{ \gamma/2 }}\lesssim \Vert g \Vert_{Y_k},
\]
so the polynomial case is thus proved by gathering all the case. For the exponential weight case, for $\gamma \in (-\frac 3 2, 1]$ by Lemma \ref{L37} we have 
\begin{equation*}
\begin{aligned} 
& ( Q(\partial^{\alpha_1} f, \partial^{\alpha - \alpha_1} g), \partial^{\alpha} h \langle v \rangle^{2k} e^{2a \langle v \rangle^b } )_{L^2_{x, v}}
\\
\le &C_{k, a, b} \int_{\T^3}   \Vert \partial^{\alpha - \alpha_1} f \Vert_{L^2_{k, a, b}} \Vert \partial^{\alpha_1} g \Vert_{L^2_{ k+\gamma/2, a, b }}\Vert \partial^{\alpha}  h \Vert_{L^2_{ k + \gamma/2, a, b}} +  \Vert \partial^{\alpha_1} g \Vert_{L^2_{k, a, b}} \Vert \partial^{\alpha - \alpha_1} f \Vert_{L^2_{ k + \gamma/2, a, b}}\Vert \partial^{\alpha} h \Vert_{L^2_{ k + \gamma/2, a , b}} dx.
\end{aligned}
\end{equation*}
For the first term we have
\begin{equation*}
\begin{aligned} 
&\int_{\T^3}  \Vert \partial^{\alpha - \alpha_1} f \Vert_{L^2_{k, a, b}} \Vert \partial^{\alpha_1} g \Vert_{L^2_{ k+\gamma/2, a, b }}\Vert \partial^{\alpha}  h \Vert_{L^2_{ k + \gamma/2, a, b}} dx
\\
\lesssim &\Vert \partial^{\alpha - \alpha_1} f \Vert_{H^{2-|\alpha -\alpha_1| }_x L^2_{k, a, b}} \Vert \partial^{\alpha_1} g \Vert_{H^{|\alpha - \alpha_1|}_x L^2_{ k+\gamma/2, a, b }}\Vert \partial^{\alpha}  h \Vert_{L^2_x L^2_{ k + \gamma/2, a, b}} 
\\
\lesssim  &\Vert  f \Vert_{H^2_x L^2_{k, a, b}} \Vert g \Vert_{H^2_x L^2_{ k+\gamma/2, a, b }}\Vert   h \Vert_{H^2_x  L^2_{ k + \gamma/2, a, b}}  \lesssim \Vert f \Vert_{X_{k, a, b} }  \Vert g \Vert_{Y_{k, a, b} } \Vert h \Vert_{Y_{k, a, b} },
\end{aligned}
\end{equation*}
and the second term follows by symmetry. For the case $\gamma \in (-3, -\frac 3 2]$, for simplicity we only prove the case $\gamma \in (-3, -\frac 5 2]$, by Lemma \ref{L37} we have
\begin{equation*}
\begin{aligned} 
&  ( Q(\partial^{\alpha_1}_{\beta_1} f, \partial^{\alpha - \alpha_1}_{\beta- \beta_1} g), \partial^{\alpha}_\beta h w^2(\alpha, \beta) e^{2 a \langle v \rangle^b})_{L^2_{x, v}}
\\
\le &C_{k, a, b} \int_{\T^3} \min_{m + n =  2} \{ \Vert \partial^{\alpha - \alpha_1}_{\beta - \beta_1}  f w(\alpha, \beta) \Vert_{H^m_{-96, a, b}} \Vert \partial^{\alpha_1}_{\beta_1} g w(\alpha, \beta)\Vert_{H^n_{ \gamma/2, a, b }}     \}    \Vert \partial^{\alpha}_{\beta}  h w(\alpha, \beta) \Vert_{L^2_{  \gamma/2, a, b}} dx
\\
&+C_{k, a, b} \int_{\T^3}  \min_{m+n =  2 }  \{ \Vert \partial^{\alpha_1}_{\beta_1} g w(\alpha, \beta) \Vert_{H^n_{-96, a, b}} \Vert \partial^{\alpha - \alpha_1}_{\beta - \beta_1} f w(\alpha, \beta)\Vert_{H^m_{  \gamma/2, a, b }}   \} \Vert \partial^{\alpha}_{\beta}  h w(\alpha, \beta) \Vert_{L^2_{ \gamma/2, a, b}}    dx.
\end{aligned}
\end{equation*}
By symmetry we only need to prove that 
\begin{equation*}
\begin{aligned} 
& \int_{\T^3} \min_{m+n =2}  \{ \Vert \partial^{\alpha - \alpha_1}_{\beta - \beta_1} f w(\alpha, \beta)  \Vert_{H^m_{- 9 6, a, b}} \Vert \partial^{\alpha_1}_{\beta_1} g w(\alpha, \beta)\Vert_{H^n_{ \gamma/2, a, b}}    \} \Vert \partial^{\alpha}_{\beta}  h  w(\alpha, \beta) \Vert_{L^2_{  \gamma/2. a, b}} dx 
\\
\lesssim & \Vert  f \Vert_{X_{k, a, b}  } \Vert g \Vert_{Y_{k, a, b} }\Vert h \Vert_{Y_{k, a, b} }, \quad \forall \alpha_1 \le \alpha, \beta_1 \le \beta.
\end{aligned}
\end{equation*}
By \eqref{weight function} we have
\[
w(\alpha, \beta)  \langle v \rangle^{-96} \le w( \alpha_2, \beta_2), \quad \forall |\alpha|  + |\beta| \le 4 , \quad \forall  |\alpha_2| +|\beta_2| \le 4 ,
\]
so for any nonnegative integer $m, n$ satisfies $m  + n + |\alpha-\alpha_1| + |\beta -\beta_1| \le 4$ we have
\[
\Vert \partial^{\alpha - \alpha_1}_{\beta - \beta_1} f w(\alpha, \beta)  \Vert_{H^m_x H^n_{ - 96 , a, b}}  \le \Vert f\Vert_{X_{k, a, b}},
\]
the remaining proof is the same as the polynomial case thus omitted, so the proof is thus finished. 
\end{proof}

Then we come to prove estimate for the linearized part, we first prove the polynomial case. 

\begin{lem}\label{L42}  For any smooth function $f$, if $\gamma \in (-\frac 3 2, 1 ]$, for any indices  $|\alpha| \le 2$ we have
\begin{equation*}
\begin{aligned} 
|(\partial^\alpha Q (f, \mu),  \partial^\alpha g \langle  v \rangle^{2k} )_{L^2_{x, v} }| \le & \Vert b(\cos \theta) \sin^{k - 2} \frac \theta 2 \Vert_{L^1_\theta}\Vert \partial^\alpha f \Vert_{L^2_x L^2_{k+\gamma/2, * }}\Vert \partial^\alpha g \Vert_{L^2_x L^2_{k+\gamma/2, * }}
\\
&+ C_k \Vert \partial^\alpha f \Vert_{L^2_x L^2_{k+\gamma/2-1/2}}\Vert \partial^\alpha g \Vert_{L^2_x L^2_{k+\gamma/2-1/2}} ,
\end{aligned}
\end{equation*}
for some constant $C_k >0$. For the case $\gamma \in (-3, -\frac 3 2]$, for any indices  $|\alpha| +|\beta| \le N$ we have
\begin{equation*}
\begin{aligned} 
|(\partial^\alpha_\beta Q (f, \mu),  \partial^\alpha_\beta g  w^2(\alpha, \beta)  )_{L^2_{x, v} }| \le & \Vert b(\cos \theta) \sin^{k- 2} \frac \theta 2 \Vert_{L^1_\theta}\Vert \partial^\alpha_\beta f w(\alpha, \beta)  \Vert_{L^2_x L^2_{\gamma/2, *}}\Vert \partial^\alpha_\beta g w(\alpha, \beta)  \Vert_{L^2_x L^2_{\gamma/2, *}} 
\\
&+ C_k \Vert \partial^\alpha_\beta f w(\alpha, \beta)  \Vert_{L^2_x L^2_{\gamma/2-1/2}}\Vert \partial^\alpha_\beta g w(\alpha, \beta)  \Vert_{L^2_x L^2_{\gamma/2-1/2}} 
\\
&+ C_k \sum_{\beta_1 < \beta} \Vert \partial^{\alpha}_{\beta_1} f w(\alpha, \beta_1)  \Vert_{L^2_x L^2_{\gamma/2}}\Vert \partial^\alpha_\beta g w(\alpha, \beta)  \Vert_{L^2_x L^2_{\gamma/2}} .
\end{aligned}
\end{equation*}
\end{lem}
\begin{proof} Since $\partial^{\alpha} \mu =0$ if $ |\alpha| > 0$, the case $\gamma \in (-\frac 3 2, 1]$ is just Lemma \ref{L31} and thus omitted. For the 
case $\gamma \in (-3, -\frac 3 2]$, by
\[
(\partial^\alpha_\beta Q(f, \mu), \partial^\alpha_\beta g w^2(\alpha, \beta) )_{L^2_{x, v} } = \sum_{\beta_1 \le \beta} ( Q(\partial^\alpha_{\beta_1} f, \partial_{\beta - \beta_1} \mu), \partial^\alpha_\beta g w^2(\alpha, \beta) )_{L^2_{x, v} }   ,
\]
we split it into two cases. For the case $\beta_1 = \beta$, by Lemma \ref{L31} we have
\begin{equation*}
\begin{aligned} 
 |( Q(\partial^\alpha_\beta f,  \mu), \partial^\alpha_\beta g w^2(\alpha, \beta))_{L^2_{x, v} }|  \le  &\Vert b(\cos \theta) \sin^{k-2 } \frac \theta 2 \Vert_{L^1_\theta}\Vert \partial^\alpha_\beta f w(\alpha, \beta) \Vert_{L^2_x L^2_{\gamma/2, *}}\Vert \partial^\alpha_\beta g w(\alpha, \beta) \Vert_{L^2_x L^2_{\gamma/2, *}} 
\\
&+ C_k \Vert \partial^\alpha_\beta    f w(\alpha, \beta) \Vert_{L^2_x L^2_{\gamma/2-1/2}}\Vert \partial^\alpha_\beta g w(\alpha, \beta) \Vert_{L^2_x L^2_{k+\gamma/2-1/2}}.
\end{aligned}
\end{equation*}
For the case $\beta_1 < \beta$, by Corollary \ref{C36}   we have
\[
|  ( Q(\partial^\alpha_{\beta_1} f, \partial_{\beta - \beta_1} \mu), \partial^\alpha_\beta g w^2(\alpha, \beta) )_{L^2_{x, v} }   |  \le   C_k \sum_{\beta_1 < \beta} \Vert \partial^\alpha_{\beta_1} f w(\alpha, \beta) \Vert_{L^2_x L^2_{\gamma/2}}\Vert \partial^\alpha_\beta g w(\alpha, \beta) \Vert_{L^2_x L^2_{\gamma/2}} ,
\]
so  the proof is thus finished since $w(\alpha, \beta) \le w(\alpha, \beta_1)$ if $|\beta_1| \le |\beta|$.  
\end{proof}

\begin{lem}\label{L43} For any smooth function $f$, for any $k \ge 4$. If $\gamma \in (-\frac 3 2, 1 ]$, for any indices $|\alpha| \le 2$ we have
\begin{equation*}
\begin{aligned} 
( \partial^\alpha Q ( \mu, f),  \partial^\alpha f \langle  v \rangle^{2k} )_{L^2_{x, v}} 
\le - \Vert  b(\cos \theta)  \sin^2 \frac \theta 2  \Vert_{L^1_\theta}\Vert \partial^\alpha f \Vert_{L^2_x L^2_{k+\gamma/2, *}}^2 + C_{k}  \Vert   \partial^\alpha f \Vert_{L^2_x L^2_{k+\gamma/2-1/2}}^2,
\end{aligned}
\end{equation*}
for some constant $C_k \ge 0$. If $\gamma \in (-3, -\frac 3 2 ]$, for any indices $|\alpha| +|\beta| \le N$ we have
\begin{equation*}
\begin{aligned} 
(\partial^\alpha_\beta Q ( \mu, f),  \partial^\alpha_\beta f w^2(\alpha, \beta) )_{L^2_{x, v} }   &\le  - \Vert  b(\cos \theta)  \sin^2 \frac \theta 2  \Vert_{L^1_\theta} \Vert  \partial^\alpha_\beta f w(\alpha, \beta)\Vert_{L^2_x   L^2_{\gamma/2, *}}^2 + C_{k}  \Vert \partial^\alpha_\beta f w(\alpha, \beta) \Vert_{L^2_x   L^2_{\gamma/2-1/2}}^2
\\
& + C_k \sum_{\beta_1 < \beta} \Vert \partial^\alpha_{\beta_1}  f w(\alpha, \beta_1) \Vert_{L^2_x   L^2_{\gamma/2}}\Vert \partial^\alpha_\beta f w(\alpha, \beta)\Vert_{L^2_x   L^2_{\gamma/2}}, 
\end{aligned}
\end{equation*}
for some constant $C_k \ge 0$. 
\end{lem}
\begin{proof}  Since $\partial^{\alpha} \mu =0$ if $ |\alpha| > 0$, the case $\gamma \in (-\frac 3 2, 1]$ is just Lemma \ref{L32} and thus omitted.  For the case $\gamma \in (-3, -\frac 3 2] $, first we have
\[
(\partial^\alpha_\beta Q( \mu , f), \partial^\alpha f w^2(\alpha, \beta))_{L^2_{x, v}}= \sum_{\beta_1 \le \beta} ( Q(\partial_{\beta - \beta_1} \mu, \partial^{\alpha }_{ \beta_1} f), \partial^{\alpha}_\beta f w^2(\alpha, \beta))_{L^2_{x, v}},
\]
we split it into two cases. For the case $\beta_1 =\beta$, by Lemma \ref{L32} we have
\[
( Q( \mu, \partial^{\alpha }_\beta f), \partial^{\alpha}_\beta f w^2(\alpha, \beta))_{L^2_{x, v}}  \le - \Vert  b(\cos \theta) \sin^2  \frac \theta 2  \Vert_{L^1_\theta}\Vert \partial^\alpha_\beta f w(\alpha, \beta)\Vert_{L^2_x L^2_{\gamma/2, *}}^2 + C_{k}  \Vert \partial^\alpha_\beta f w(\alpha, \beta) \Vert_{L^2_x L^2_{\gamma/2-1/2}}^2.
\]
For the case $|\beta_1| < |\beta|$, by Corollary \ref{C36}  we have
\[
|( Q(\partial_{ \beta - \beta_1} \mu, \partial^\alpha_{\beta_1} f), \partial^{\alpha}_{\beta} f w^2 (\alpha, \beta))_{L^2_{x, v}} |\le C_k \Vert \partial^\alpha_{  \beta_1}  f w(\alpha, \beta) \Vert_{L^2_{\gamma/2}}\Vert \partial^\alpha_\beta f w(\alpha, \beta) \Vert_{L^2_{\gamma/2}} ,
\]
so the lemma is thus proved since $w(\alpha, \beta) \le w(\alpha, \beta_1) $ if $|\beta_1| \le |\beta|$.  
\end{proof}

Then we prove the estimate for the exponential weight case. 

\begin{lem}\label{L44} For any smooth function $f$, for any $k \in \R, a >0, b \in (0, 2)$. If $\gamma \in (-\frac 3 2, 1 ]$, for any indices $|\alpha| \le 2$ we have
\[
(\partial^\alpha Q (f, \mu),  \partial^\alpha f \langle  v \rangle^{2k} e^{2 a  \langle v \rangle^b }     )_{L^2_{x, v} } + (\partial^\alpha Q (\mu, f),  \partial^\alpha f \langle  v \rangle^{2k} e^{2 a  \langle v \rangle^b }     )_{L^2_{x, v} }
\le -C_1 \Vert \partial^\alpha f  \Vert_{L^2_x L^2_{k+\gamma/2, a, b}}^2 + C_{k, a, b}  \Vert \partial^\alpha f \Vert_{L^2_{x, v}}^2,
\]
for some constants $C_1, C_{k, a, b}>0$. If $\gamma \in (-3, -\frac 3 2 ]$, for any indices $|\alpha| +|\beta| \le N$ we have
\begin{equation*}
\begin{aligned} 
&(\partial^\alpha_\beta Q (f, \mu),  \partial^\alpha_\beta f w^2(\alpha, \beta) e^{2 a  \langle v \rangle^b }     )_{L^2_{x, v} } + (\partial^\alpha_\beta Q (\mu, f),  \partial^\alpha_\beta f w^2(\alpha, \beta) e^{2 a  \langle v \rangle^b }     )_{L^2_{x, v} }
\\
\le &-C_1 \Vert \partial^\alpha_\beta f w(\alpha, \beta) \Vert_{L^2_x L^2_{\gamma/2, a, b}}^2 + C_{k, a, b}  \Vert \partial^\alpha_\beta f \Vert_{L^2_{x, v}}^2 +  C_{k, a, b} \sum_{\beta_1 < \beta} \Vert \partial^{\alpha}_{\beta_1} f w(\alpha, \beta_1)  \Vert_{L^2_x L^2_{\gamma/2, a, b }}\Vert \partial^\alpha_\beta f w(\alpha, \beta) \Vert_{L^2_x L^2_{\gamma/2, a, b  }} ,
\end{aligned}
\end{equation*}
for some constants $C_1, C_{k, a, b} >0$.
\end{lem}
\begin{proof}
Since $\partial^{\alpha} \mu =0$ if $ |\alpha| > 0$, the case $\gamma \in (-\frac 3 2, 1]$ is just Corollary \ref{C311} and thus omitted.  For the case $\gamma \in (-3, -\frac 3 2] $, first we have
\begin{equation*}
\begin{aligned} 
&(\partial^\alpha_\beta Q(f, \mu), \partial^\alpha_\beta f w^2(\alpha, \beta) e^{2 a  \langle v \rangle^b  }   )_{L^2_{x, v} }   + (\partial^\alpha_\beta Q (\mu, f),  \partial^\alpha_\beta f w^2 (\alpha, \beta)e^{2 a  \langle v \rangle^b }     )_{L^2_{x, v} }  
\\
= & \sum_{\beta_1 \le \beta} ( Q(\partial^\alpha_{\beta_1} f, \partial_{\beta - \beta_1} \mu), \partial^\alpha_\beta f w^2(\alpha, \beta) e^{2 a \langle v \rangle^b }  )_{L^2_{x, v} }   + ( Q(\partial_{\beta - \beta_1} \mu,  \partial^\alpha_{\beta_1} f), \partial^\alpha_\beta f w^2(\alpha, \beta)  e^{2 a\langle v \rangle^b  }  )_{L^2_{x, v} },
\end{aligned}
\end{equation*}
we split it into two cases. For the case $\beta_1 = \beta$, by Corollary \ref{C311}  we have
\begin{equation*}
\begin{aligned} 
&( Q(\partial^\alpha_\beta f,  \mu), \partial^\alpha_\beta f w^2(\alpha, \beta)  e^{2 a  \langle v \rangle^b }   )_{L^2_{x, v} }  +  ( Q(\mu, \partial^\alpha_\beta f), \partial^\alpha_\beta f w^2(\alpha, \beta) e^{2 a \langle v \rangle^b }   )_{L^2_{x, v} }  
\\
 \le & -C_1 \Vert \partial^\alpha_\beta f w(\alpha, \beta) \Vert_{L^2_x L^2_{\gamma/2, a, b }}^2   + C_{k, a, b} \Vert \partial^\alpha_\beta    f  \Vert_{L^2_{x, v}}^2.
\end{aligned}
\end{equation*}
For the case $\beta_1 < \beta$, by Corollary \ref{C38}   we have
\begin{equation*}
\begin{aligned} 
&| ( Q(\partial^\alpha_{\beta_1} f, \partial_{\beta - \beta_1} \mu), \partial^\alpha_\beta f w^2 (\alpha, \beta) e^{2 a  \langle v \rangle^b }  )_{L^2_{x, v} }   |  + | ( Q(\partial_{\beta - \beta_1} \mu, \partial^\alpha_{\beta_1} f), \partial^\alpha_\beta f w^2(\alpha, \beta) e^{2 a  \langle v \rangle^b }  )_{L^2_{x, v} }   |  
\\
\le & C_{k, a, b}  \sum_{\beta_1 < \beta} \Vert \partial^\alpha_{\beta_1} f w(\alpha, \beta)  \Vert_{L^2_x L^2_{\gamma/2, a, b}}\Vert \partial^\alpha_\beta f    w(\alpha, \beta) \Vert_{L^2_x L^2_{ \gamma/2, a, b}},
\end{aligned}
\end{equation*}
so the lemma is thus proved since $w(\alpha, \beta) \le w(\alpha, \beta_1) $ if $\beta_1 \le \beta$.  
\end{proof}

For the transport  term $v \cdot \nabla_x f$ we need the following estimate.

\begin{lem}\label{L45}
Suppose $|\beta|>0$, for any smooth function $f$ we have
\[
(\partial^\alpha_{\beta} (v \cdot \nabla_x f)  ,\partial^\alpha_{\beta}f  w^2(\alpha, \beta)  )_{L^2_{x, v}} \le C(\sum_{|\beta_2|=|\beta|-1,  |\alpha_2| =|\alpha|+1} \Vert \partial^{\alpha_2}_{\beta_2}  f w(\alpha_2, \beta_2)\Vert_{L^2_xL^2_{\gamma/2}}) \Vert \partial^\alpha_{\beta} f w(\alpha, \beta) \Vert_{L^2_{x}L^2_{\gamma/2} }.
\]
\end{lem}
\begin{proof}
By \eqref{w v nabla x} and 
\[
\partial_{v_i} (v \cdot \nabla_x f) = \partial_{x_i} f + v \cdot \nabla_x \partial_{v_i} f, \quad \forall i =1, 2, 3,
\]
the theorem is thus proved.
\end{proof}

Recall the definition of $\bar {X}_0$. $\bar {X}_0 :=H^2_xL^2_v$ if $\gamma \in (-\frac 3 2, 1]$, $ \bar{X}_0= H^N_{x, v} $,  if $\gamma \in (-3 , -\frac 3 2]$, and $Y_{k, *}$ is defined in \eqref{Y k *}. Gathering the estimates above, we obtian following estimate.

\begin{lem}\label{L46}
For any smooth function $f, g, h$ smooth, for the polynomial case,  for any $k>4$ large, for the nonlinear term we have
\begin{equation}
\label{estimate nonlinear X k}
|(Q(f, g), h)_{X_k}| \lesssim   (\Vert f \Vert_{X_{4} } \Vert g \Vert_{Y_k} +\Vert f \Vert_{Y_k} \Vert g \Vert_{ X_{4} } )\Vert h \Vert_{Y_k}.
\end{equation}
For the linearized term we have
\begin{equation}
\label{estimate linearized X k 1}
|( Q (f, \mu),   g  )_{X_k}| \le \Vert b(\cos \theta) \sin^{k- 2  -\frac {k-4} 3} \frac \theta 2 \Vert_{L^1_\theta}\Vert f \Vert_{Y_{k, *} }\Vert  g \Vert_{Y_{k, *}} 
+ C_k \Vert  f \Vert_{Y_{k-1/2}}\Vert  g \Vert_{Y_{k-1/2}},
\end{equation}
and
\begin{equation}
\label{estimate linearized X k 2}
( Q ( \mu, f),  f )_{X_k} \le  - \Vert  b(\cos \theta) \sin^{2 + \frac {k-4} 3} \frac \theta 2  \Vert_{L^1_\theta}\Vert  f \Vert_{Y_{k, *} }^2 + C_{k}  \Vert  f \Vert_{Y_{k-1/2}}^2.
\end{equation}
In particular gathering the two terms we have 
\begin{equation}
\label{estimate linearized X k 3}
( Q ( \mu, f),  f )_{X_k} + ( Q (f,  \mu),  f )_{X_k}  \le  - \Vert  b(\cos \theta)\sin^2 \frac \theta 2 (  \sin^{\frac {k-4} 3}\frac \theta 2  -  \sin^{\frac {2(k-4)} 3}\frac \theta 2) \Vert_{L^1_\theta} \Vert  f \Vert_{Y_{k, *} }^2 + C_{k}  \Vert  f \Vert_{Y_{k-1/2}}^2 .
\end{equation}
For the exponential weight case, for the nonlinear term we have 
\begin{equation}
\label{estimate nonlinear X k a b}
|(Q(f, g), h)_{X_{k, a, b}}| \lesssim   (\Vert f \Vert_{X_{k, a, b} } \Vert g \Vert_{Y_{k, a ,b}} +\Vert f \Vert_{Y_{k, a, b}} \Vert g \Vert_{ X_{k, a, b} } )\Vert h \Vert_{Y_{k, a, b}},
\end{equation}
for the linearized term we have
\begin{equation}
\label{estimate linearized X k a b}
( Q ( \mu, f),  f )_{X_{k, a, b} } + ( Q (f,  \mu),  f )_{X_{k, a, b} } \le  - C_2 \Vert  f \Vert_{Y_{k, a, b}}^2 + C_k  \Vert  f \Vert_{\bar{X_{0}} }^2.
\end{equation}
For the $v \cdot \nabla_x f$ term, for the polynomial case we have
\begin{equation}
\label{estimate polynomial v nabla x}
|(- v \cdot \nabla_x f, f )_{X_k}| \le \frac 1 4 \Vert  b(\cos \theta)\sin^2 \frac \theta 2 (  \sin^{\frac {k-4} 3}\frac \theta 2  -  \sin^{\frac {2(k-4)} 3}\frac \theta 2) \Vert_{L^1_\theta} \Vert  f \Vert_{Y_{k, *}}^2, 
\end{equation}
for the exponential weight case we have
\begin{equation}
\label{estimate exponential v nabla x}
|(- v \cdot \nabla_x f, f )_{X_{k, a, b}}| \le \frac {C_2} {4} \Vert  f \Vert_{Y_{k, a, b}  }^2 .
\end{equation}
\end{lem}

\begin{proof}
\eqref{estimate nonlinear X k} and \eqref{estimate nonlinear X k a b} can be  proved by summing on $|\alpha| +|\beta| \le N$ ($|\alpha| \le 2$ if $\gamma \in(-\frac 3 2 ,1]$) in  Lemma \ref{L41}. For any $\eta>0$ we have
\begin{equation*}
\begin{aligned} 
&C_{|\alpha|, |\beta|}^2C_k \sum_{|\beta_1| < |\beta|} \Vert \partial^\alpha_{\beta_1} f  w(\alpha, \beta_1) \Vert_{L^2_{\gamma/2}}\Vert \partial^\alpha_{\beta} g  w(\alpha, \beta) \Vert_{L^2_{\gamma/2}} 
\\
\le & \sum_{|\beta_1| < |\beta|}  \frac  {C_k C_{|\alpha|, |\beta|} } {\eta} \Vert \partial^\alpha_{\beta_1} f  w(\alpha, \beta_1) \Vert_{L^2_{\gamma/2}}  \eta C_{|\alpha|, |\beta|} \Vert \partial^\alpha_{\beta} g  w(\alpha, \beta) \Vert_{\gamma/2}.
\end{aligned}
\end{equation*}
By \eqref{constant} we have $C_{|\alpha|, |\beta|} \ll C_{|\alpha|, |\beta_1|}$, so \eqref{estimate linearized X k 1} follows from summing on $|\alpha| +|\beta| \le N$ ($|\alpha| \le 2$ if $\gamma \in(-\frac 3 2 ,1]$) in Lemma \ref{L42} and  taking  suitable small constants $\eta>0$  such that
\[
\eta \ll   \Vert b(\cos \theta) (\sin^{k-2  -\frac {k-4} 3}  \frac \theta 2  -  \sin^{k-2}  \frac \theta 2) \Vert_{L^1_\theta}, \quad \frac  {C_k C_{|\alpha|, |\beta|} } {\eta} \ll C_{|\alpha|, |\beta_1|},
\]
the estimate \eqref{estimate linearized X k 2} and  \eqref{estimate linearized X k a b} can be proved similarly. For the $v \cdot \nabla_x f$ term, if $\gamma \in (-\frac 3 2, 1]$, it is easily seen that
\[
(v \cdot \nabla_x \partial^\alpha  f, \partial^\alpha f \langle v \rangle^{2k})_{L^2_{x, v}} = 0.
\]
For the case $\gamma \in (-3, -\frac 3 2 ]$ we have
\begin{equation*}
\begin{aligned} 
& C_{|\alpha|, |\beta|}^2 C_k(\sum_{|\beta_2|=|\beta|-1,  |\alpha_2| =|\alpha|+1} \Vert \partial^{\alpha_2}_{\beta_2}  f w(\alpha_2, \beta_2)\Vert_{L^2_xL^2_{\gamma/2}}) \Vert \partial^\alpha_{\beta} f w(\alpha, \beta) \Vert_{L^2_{x}L^2_{\gamma/2} } 
\\
\le&  (\frac {C_{|\alpha|, |\beta|} C_k} \eta   \sum_{|\beta_2|=|\beta|-1,  |\alpha_2| =|\alpha|+1} \Vert \partial^{\alpha_2}_{\beta_2}  f w(\alpha_2, \beta_2)\Vert_{L^2_xL^2_{\gamma/2}})  \eta C_{|\alpha|, |\beta|} \Vert \partial^\alpha_{\beta} f w(\alpha, \beta) \Vert_{L^2_{x}L^2_{\gamma/2} } .
\end{aligned}
\end{equation*}
By \eqref{constant} we have $C_{|\alpha|, |\beta|} \ll C_{|\alpha|+1, |\beta|-1}$, so \eqref{estimate polynomial v nabla x} follows by summing on $|\alpha| +|\beta| \le N$ in Lemma \ref{L45} and taking suitable $\eta$ such that 
\[
\frac {C_{|\alpha|, |\beta|} C_k} \eta  \ll C_{|\alpha|+1, |\beta|-1} , \quad \eta \ll \frac 1 2 \Vert  b(\cos \theta)\sin^2 \frac \theta 2 (  \sin^{\frac {k-4} 3}\frac \theta 2  -  \sin^{\frac {2(k-4)} 3}\frac \theta 2) \Vert_{L^1_\theta}.
\]
And \eqref{estimate exponential v nabla x} can be proved similarly.
\end{proof}

Taking $g=f$ in Lemma \ref{L46} we can easily obtain the following estimate. 

\begin{cor}\label{C47}
Suppose that $-3 < \gamma \le 1$, $f$ smooth. For the polynomial case, for any $k > 4$ large, there exists constants $c_0, C_k>0$ such that
\begin{equation*}
\begin{aligned}
( Q(\mu+f, \mu+f) , f  )_{X_k} &\le -2c_0 \Vert f \Vert_{Y_k}^2 + C_{k}  \Vert f \Vert_{Y_{k-1/2}}^2 + C_k\Vert f \Vert_{\bar{X}_4} \Vert f \Vert_{Y_k}^2
\\
&\le -c_0 \Vert f \Vert_{Y_k}^2 + C_{k}  \Vert f \Vert_{\bar{X_0}}^2 + C_k\Vert f \Vert_{X_4} \Vert f \Vert_{Y_k}^2.
\end{aligned}
\end{equation*}
For the exponential weight case, for any $k \in \R, a>0, b \in (0, 2)$ we have
\[
( Q(\mu+f, \mu+f) , f  )_{X_{k, a, b}} \le -c_0 \Vert f \Vert_{Y_{k, a, b} }^2 + C_{k, a, b}    \Vert f \Vert_{\bar{X_0}}^2  + C_{k, a, b}\Vert f \Vert_{X_{k, a, b}} \Vert f \Vert_{Y_{k, a, b}}^2,
\]
for some constants $c_0, C_{k, a, b} > 0$.
\end{cor}

\begin{cor}\label{C48} Suppose $\gamma \in (-3, 1]$. For any smooth function $f$, suppose $f$ is the solution of 
\[
\partial_t f  = \bar{L} f :=  -v \cdot \nabla_x f + Q(f, \mu) + Q(\mu, f) , \quad f|_{t=0} = f_0.
\]
If $\gamma \in [0, 1]$, for any $k >4$ we have
\[
\Vert f - P f\Vert_{X_k} \lesssim e^{-\lambda t}   \Vert f_0 - P f_0 \Vert_{X_k},
\]
for some constant $\lambda>0$.  If $\gamma \in (-3, 0)$, for any $4 < k_0 <  k$  we have
\[
\Vert f - P f\Vert_{X_{k_0}} \lesssim \langle  t \rangle^{-\frac {k -k_* } {|\gamma| }}\Vert f_0 - P f_0 \Vert_{X_k} ,\quad \forall k_* \in (k_0, k).
\]
For the exponential weight, if $\gamma \in [0, 1]$, for any $k \in \R, 0< a, b \in (0, 2)$ we have
\[
\Vert f - P f\Vert_{X_{k, a , b}} \lesssim  e^{- \lambda t }\Vert f_0 - P f_0 \Vert_{X_{k, a, b}},
\]
for some constant $\lambda>0$.  If $\gamma \in (-3, 0)$, for any $k \in \R, 0 < a_0 < a, b \in (0, 2)$ we have
\[
\Vert f - P f\Vert_{X_{k, a_0 , b}} \lesssim  e^{- \lambda t^{\frac b {b - \gamma}  } }\Vert f_0 - P f_0 \Vert_{X_{k, a, b}},
\]
for some constant $ \lambda >0$.
\end{cor}
\begin{proof}

The proof is similar as Lemma \ref{L312} thus omitted. 
\end{proof}

\begin{cor}\label{C49} Denote $Z_k=X_{k - \gamma / 2}$, then we have
\[
\Vert Q(f, g) \Vert_{Z_k} \lesssim   \Vert f \Vert_{X_4} \Vert g \Vert_{Y_k} +\Vert f \Vert_{Y_k} \Vert g \Vert_{X_4} .
\]
\end{cor}
\begin{proof}
It's easily seen that $Z_k$ is the dual of $Y_k$ with respect to $X_k$, so the corollary follows by \eqref{estimate nonlinear X k} in Lemma \ref{L46}.
\end{proof}

\section{Global existence and convergence}\label{section5}

The proof of local existence is standard once we have established estimates in Lemma \ref{L46}, we refer to  \cite{AMSY} for example.

\begin{thm} (Local existence) For any $k  > 4$, there exists $\epsilon_0, \epsilon_1, T>0$ such that if $f_0 \in X_k$ and
\[
\Vert f_0 \Vert_{X_{k}} \le \epsilon_0 ,\quad  \mu + f_0 \ge 0,
\]
then the Cauchy problem
\[
\partial_t f + v \cdot \nabla_x f= Q(\mu+f, \mu+f), \quad f|_{t=0}  =f_0(x, v),
\]
admits a unique weak solution $f \in L^\infty([0, T]; X_k )$ satisfying 
\[
\Vert f \Vert_{ L^\infty ([0, T]; X_k )} \le \epsilon_1, \quad \mu+f \ge 0, \quad  \Vert f \Vert_{L^2([0, T] ; Y_k) } \le \epsilon_1.
\]
\end{thm}

\begin{thm}
Recall $\bar{L} f =  -v \cdot \nabla_x f + Q(f, \mu) + Q(\mu, f)$. For any function $f$ satisfies $P f =0$,  for any $k \ge  6  $, define the norm $||| f |||_{X_k}$ and the associate scalar product $((f, g))_{X_k}$ by
\[
||| f |||_{X_k}^2  = \eta \Vert f \Vert_{X_k}^2 + \int_0^\infty \Vert S_{\bar{L}}  (\tau ) f \Vert_{\bar{X_0}}^2 d \tau,  \quad (( f, g ))_{X_k}  = \eta (f, g )_{X_k} + \int_0^\infty (S_{\bar{L}} (\tau ) f,  S_{\bar{L}} (\tau ) g)_{\bar{X_0}} d \tau.
\]
Similarly for any $k \in \R, a > 0, b \in (0, 2)$, define the norm $||| f |||_{X_{k, a, b}}$ and the associate scalar product $((f, g))_{X_{k, a, b}}$ by
\[
||| f |||_{X_{k, a, b}}^2  = \eta \Vert f \Vert_{X_{k, a , b}}^2 + \int_0^\infty \Vert S_{\bar{L}} (\tau ) f \Vert_{\bar{X_0}}^2 d \tau, \quad (( f, g ))_{X_{k, a, b}}  = \eta (f, g )_{X_{k, a, b}} + \int_0^\infty (S_{\bar{L}} (\tau ) f,  S_{\bar{L}} (\tau ) g)_{\bar{X_0}} d \tau.
\]
Then there exists some $\eta>0$, such that the norm $||| \cdot |||_{X_k}$ $(||| \cdot |||_{X_{k, a, b}})$ is equivalent to $\Vert \cdot \Vert_{X_k}$ $(\Vert \cdot \Vert_{X_{k, a, b}})$ on the space $\{ f \in X_k| P f =0 \}$ $(\{ f \in X_{k, a, b}| P f =0 \})$. Moreover there exists  some constants $C, K>0$ such that any smooth solution to the following equation 
\begin{equation}
\label{Boltzmann equation with perturbation}
\partial_t f  = \bar{L} f +Q(f, f), \quad f(0) =f_0, \quad P f_0 = 0  ,
\end{equation}
satisfies for the polynomial case 
\begin{equation}
\label{polynomial}
\frac d {dt} ||| f |||_{X_k}^2  \le (C ||| f  |||_{X_6} - K ) \Vert f \Vert_{Y_k}^2,
\end{equation}
and for the exponential weight case
\begin{equation}
\label{exp}
\frac d {dt} ||| f |||_{X_{k, a, b}}^2  \le (C ||| f |||_{X_{k, a, b}} - K ) \Vert f \Vert_{Y_{k, a, b}}^2.
\end{equation}
As a consequence, if $ ||| f_0 |||_{X_6} \le \frac {K}{2C}$, then there exists a global solution $f \in L^\infty([0, \infty), X_6), \mu +f \ge 0$ to the Boltzmann equation \eqref{Boltzmann equation with perturbation}. Moreover for any $k \ge 6$, if we assume $\Vert f_0 \Vert_{X_k} < +\infty$, then for the case $\gamma \in [0, 1]$ we have
\[
 ||| f |||_{X_k}  \lesssim e^{-\lambda t}  ||| f |||_{X_k}, 
\]
for some constant $\lambda>0$. For the case $\gamma \in (-3, 0)$,  for any $6 \le k_1 < k$ we have
\[
 ||| f |||_{X_{k_1}}  \lesssim \langle t \rangle^{\frac {k- k_*} {|\gamma|}} ||| f |||_{X_{k}}, \quad \forall k_* \in (k_1, k).
\]
For the exponential weight case if we assume $||| f_0 |||_{X_{k, a, b}} \le \frac {K}{2C}$, then there exists a global solution $f \in L^\infty([0, \infty),   X_{k, a, b})$, $\mu +f \ge 0$ to the Boltzmann equation \eqref{Boltzmann equation with perturbation}. Moreover for the case $\gamma \in [0, 1]$ 
\[
 ||| f |||_{X_{k, a, b}}  \lesssim e^{-\lambda t}  ||| f |||_{X_{k, a, b}}, 
\]
for some constant $\lambda>0$. For the case $\gamma \in (-3, 0)$,  for any $ 0 < a_0 < a$ we have
\[
 ||| f |||_{X_{k, a_0, b}}  \lesssim e^{-\lambda t^{\frac {b} { b -\gamma} }}  ||| f |||_{X_{k, a, b}},
\]
for some constant $\lambda>0$. 
\end{thm}
\begin{proof} During the proof, we will denote $X = X_k, Y=Y_k$ for the polynomial weight case and $X =X_{k, a , b}, Y =Y_{k, a, b}$ for the exponential weight case. Since $k \ge 6$ for the polynomial weight case,  by Corollary \ref{C48}, for both cases we have
\[
\Vert S_{\bar{L}}   (\tau) f\Vert_{\bar{X_0}} \le \theta(\tau) \Vert f\Vert_{X} ,    \quad \lim_{\tau \to \infty}\theta(\tau) =0, \quad  \int_0^\infty \theta^2(\tau)  d\tau < +\infty,
\]
for some function $\theta(\tau) $, which implies
\[
\int_0^\infty \Vert S_{\bar{L}}(\tau ) f \Vert_{\bar{X_0}}^2 d \tau \lesssim \Vert f \Vert_X^2  \int_0^\infty \theta^2(\tau) d\tau,
\]
the equivalence between two norms is thus proved. Then we compute
\[
\frac {d} {dt}  \frac 1 2 ||| f(t) |||^2_X  =\eta (Q(\mu + f, \mu+f) ,f )_X + \int_0^\infty ( S_{\bar{L}} (\tau)\bar{L} f, S_{\bar{L}}   (\tau) f )_{\bar{X_0}  } d\tau + \int_0^\infty ( S_{\bar{L}}   (\tau)Q(f, f ) , S_{\bar{L}}  (\tau) f )_{\bar{X_0} } d\tau.
\]
We will estimate the terms separately, first by Corollary \ref{C47} we have for the polynomial weight case
\[
(Q(\mu + f, \mu+f), f)_X  \le -c_0 \Vert f \Vert_{Y}^2 + C_{k}  \Vert f \Vert_{\bar{X_0} }^2 + C_k\Vert f \Vert_{{X_4} } \Vert f \Vert_{Y}^2,
\]
and for the exponential weight case 
\[
(Q(\mu + f, \mu+f), f)_X  \le -c_0 \Vert f \Vert_{Y}^2 + C_{k}  \Vert f \Vert_{\bar{X_0} }^2 + C_k\Vert f \Vert_{{X} } \Vert f \Vert_{Y}^2.
\]
Recall that
\[
\Vert S_L(\tau) f(t) \Vert_{\bar {X_0} } \le \theta(\tau+t) \Vert f_0\Vert_{X} ,  \quad \lim_{\tau \to \infty}\theta(\tau+t) =0, \quad \forall  t \ge 0.
\]
For the second term we have
\[
\int_0^\infty ( S_{\bar{L}}   (\tau)   \bar{L}f, S_{\bar{L}}  (\tau) f )_{ \bar {X_0} } d\tau =\int_0^\infty \frac {d} {d\tau} \Vert S_{\bar{L}} (\tau) f(t)\Vert_{ \bar {X_0} }^2 d\tau = \lim_{\tau \to \infty}\Vert S_{\bar{L}}   (\tau ) f(t)\Vert_{ \bar {X_0}}^2 - \Vert f(t) \Vert_{ \bar {X_0} }^2 = -\Vert f(t) \Vert_{ \bar {X_0} }^2.
\]
For the last term we have
\[
\int_0^\infty ( S_L(\tau)Q(f, f ) , S_L(\tau) f )_{\bar {X_0} } d\tau \le\int_0^\infty \Vert S_L(\tau)Q(f, f )\Vert_{\bar {X_0} }   \Vert S_L(\tau) f \Vert_{ \bar {X_0} } d\tau.
\]
For the case $\gamma \in [0, 1]$, by Corollary \ref{C48} and Corollary \ref{C49} we have
\[
\int_0^\infty \Vert S_{\bar{L}}  (\tau)Q(f, f )\Vert_{\bar {X_0} }   \Vert S_L(\tau) f \Vert_{ \bar {X_0} } d\tau \lesssim \Vert Q(f, f )\Vert_{{X_ 5 }}    \Vert  f \Vert_{X_{5 }}  \int_0^{\infty} e^{-\lambda \tau } d\tau  \lesssim \Vert Q(f, f )\Vert_{{Z_ 6 }}    \Vert  f \Vert_{Y_6}  \lesssim \Vert f \Vert_{X_6} \Vert  f \Vert_{Y_6}^2.
\]
For the case $\gamma \in (-3, 0)$, since $ 6 > k_1 + \frac { |\gamma| }  2 $ for some $k_1 >4$, by Corollary \ref{C48} and Corollary \ref{C49}  so we have
\begin{equation*}
\begin{aligned}
\int_0^\infty \Vert S_{\bar{L}}  (\tau)Q(f, f )\Vert_{\bar {X_0} }   \Vert S_{\bar{L}}  (\tau) f \Vert_{ \bar {X_0} } d\tau & \lesssim \Vert Q(f, f )\Vert_{{X_{ 6 + |\gamma| / 2 }}}    \Vert  f \Vert_{X_{6 - |\gamma| / 2 }}  \int_0^{\infty} \langle t  \rangle^{ - \frac {6+|\gamma|/2 - k_1  + 6 -|\gamma|/2 -k_1} {|\gamma|}  }   d\tau
\\
&  \lesssim \Vert Q(f, f )\Vert_{{Z_ 6 }}    \Vert  f \Vert_{Y_6}  \lesssim \Vert f \Vert_{X_6} \Vert  f \Vert_{Y_6}^2,
\end{aligned}
\end{equation*}
by taking a suitable $\eta$ and combining all the terms, \eqref{polynomial} and \eqref{exp} is thus proved. For the global existence and convergence rate,  if $||| f_0 |||_{X_6} \le \frac {K} {2C}$, then 
\[
\frac d {dt} ||| f |||_{X_6}^2  \le (C|||f |||_{X_6} - K ) \Vert f \Vert_{Y_6}^2,
\]
we deduce that $||| f |||_{X_6}$ is decreasing over time for all $t \ge 0$. Together with the local existence we know that there exists a global solution $f \in L^\infty ((0, \infty), X_6)$. Now we come to prove the convergence rate, for the polynomial case, for all $ k \ge 6 $ we have
\[
\frac d {dt} ||| f |||_{X_k}^2  \le (C|||f |||_{X_6} - K ) \Vert f \Vert_{Y_k}^2  \le -\frac {K} 2 \Vert f \Vert_{Y_k}^2.
\]
Thus the convergence rate can be proved similarly as Lemma \ref{L312}, the exponential weight case can be proved similarly. 
\end{proof}

\section{Global existence for the Boltzmann equation with large amplitude initial data}\label{section6}

In this section we prove the global existence for the Boltzmann equation with large amplitude initial data.  We first prove some useful lemmas. 
\begin{lem}\label{L61}
For any $\gamma \in (-3, 1], k >\max \{3+\gamma, 3\} , \epsilon>0$ small enough we have
\[
\int_{ \{|v-v_*| >\frac {\langle v \rangle} \epsilon \cup  |v-v_*| < \epsilon \langle v \rangle \} }  |v-v_*|^\gamma \langle v_* \rangle^{-k} dv_*\le C_{k, \epsilon} \langle v \rangle^\gamma , \quad \lim_{\epsilon \to 0} C_{k, \epsilon } =0,
\]
for any $v \in \R^d$. 
\end{lem}

\begin{proof}
If $|v| \le \frac 1 2$, then $|v_*| +\frac 1 2 \le 1 + |v-v_*|$, so we have
\begin{equation*}
\begin{aligned}
\int_{ \{|v-v_*| >\frac {\langle v \rangle} \epsilon \cup  |v-v_*| < \epsilon \langle v \rangle \}        } |v-v_*|^\gamma \langle v_* \rangle^{-k} dv_* = &\int_{\{|v_*| >\frac {\langle v \rangle} \epsilon \cup  |v_*| < \epsilon \langle v \rangle \} } |v_*|^\gamma \langle v- v_* \rangle^{-k} dv_*
\\
\le  & C_k \int_{\{|v_*| >\frac {\langle v \rangle} \epsilon \cup  |v_*| < \epsilon \langle v \rangle \}  } |v_*|^\gamma \langle  v_* \rangle^{-k} dv_*.
\end{aligned}
\end{equation*}
We easily compute that
\[
C_k \int_{ \{ |v_*| < \epsilon \langle v \rangle \}  } |v_*|^\gamma \langle  v_* \rangle^{-k} dv_* \le C_k \int_{ \{ |v_*| < \epsilon \langle v \rangle \}  } |v_*|^\gamma dv_*\le C_k \epsilon^{\gamma+3} \langle v \rangle^{\gamma+3} \le C_k \epsilon^{\gamma+3} \langle v \rangle^{\gamma},
\]
and
\[
C_k \int_{ \{ |v_*| >\frac {\langle v \rangle} \epsilon  \}  } |v_*|^\gamma \langle  v_* \rangle^{-k} dv_* \le C_k \int_{ \{   |v_*| >\frac {1} \epsilon\}  } |v_*|^{\gamma-k} dv_*\le C_k \epsilon^{ k -\gamma -3  } \le C_k \epsilon^{ k-\gamma -3} \langle v \rangle^{\gamma},
\]
so the case $|v| \le \frac 1 2$ is thus proved. Consider now $|v| > 1/2$, we split the integral into two regions $|v - v_*| > \langle v \rangle /\epsilon$ and $|v - v_*| \le \epsilon \langle v \rangle $. For the first region, since $|v| \le \frac 1 2 |v_*|$ implies $|v-v_*| \ge \frac 1 2 |v_*|$, so we have
\begin{equation*}
\begin{aligned}
&\int_{ \{|v-v_*| >\frac {\langle v \rangle} \epsilon \}        } |v-v_*|^\gamma \langle v_* \rangle^{-k} dv_* =\int_{\{|v_*| >\frac {\langle v \rangle} \epsilon \} } |v_*|^\gamma \langle v- v_* \rangle^{-k} dv_*
\\
\le  & C_k \int_{\{|v_*| >\frac {\langle v \rangle} \epsilon  \}  } |v_*|^\gamma \langle  v_* \rangle^{-k} dv_*  \le C_k \int_{\{|v_*| >\frac {\langle v \rangle} \epsilon  \}  } |  v_*  |^{-k+\gamma} dv_*  \le C_k\epsilon^{k-\gamma-3} \langle v \rangle^{-k+\gamma+3 } \le C_k\epsilon^{k-\gamma-3} \langle v \rangle^{\gamma }. 
\end{aligned}
\end{equation*}
For the second region, since $|v| > 1/2$ and $|v - v_*| \le \epsilon \langle v \rangle $ imply $| v_*| \ge \langle v \rangle /4$, hence
\[
\int_{|v-v_*| \le \epsilon \langle v \rangle } |v-v_*|^\gamma \langle v_* \rangle^{-k} dv_* \le C_k \langle v \rangle^{-k} \int_{|v-v_*| \le \epsilon \langle v \rangle } |v-v_*|^\gamma dv_*   \le C_k \epsilon^{3+\gamma} \langle v \rangle^{-k+\gamma+3} ,
\]
so the theorem is thus proved by gathering all the terms together. 
\end{proof}

For  the linearized part of the polynomial case we are able to prove a better estimate.
\begin{lem}\label{L62}
For any $-3 < \gamma \le 1$, for any constant $k> \max\{3, 3+\gamma\}$ we have
\[
I := \int_{\R^3} \int_{\mathbb{S}^2} |v-v_*|^\gamma \frac {\langle v \rangle^k } {\langle v' \rangle^k  } e^{-\frac 1 2 |v_*'|^2 } dv_* d\sigma \le \frac {c} {k^{\frac {\gamma+3} 4}} \langle v \rangle^{\gamma}+ C_k \langle v \rangle^{\gamma-2},
\]
for some constant $c>0$ (independent of k) and for all $v \in \R^d$. Moreover for any $\epsilon>0$ small we have
\[
J = \int_{\R^3} \int_{\mathbb{S}^2} 1_{ \{|v- v'| >\frac {\langle v \rangle} \epsilon \cup  |v- v'| < \epsilon \langle v \rangle \} }   |v-v_*|^\gamma \frac {\langle v \rangle^k } {\langle v' \rangle^k  } e^{-\frac 1 2 |v_*'|^2 } dv_* d\sigma \le C_{k, \epsilon} \langle v \rangle^{\gamma}, \quad \lim_{\epsilon \to 0} C_{k ,\epsilon } =0.
\]
We also have
\[
K := \int_{\R^3} \int_{\mathbb{S}^2} |v-v_*|^\gamma \frac {\langle v \rangle^k } {\langle v' \rangle^k  } e^{-\frac 1 2 |v_*'|^2 }  \langle v' \rangle^{-2} dv_* d\sigma \le  C_k \langle v \rangle^{\gamma-2}. 
\]
\end{lem}

\begin{proof}
Since $\gamma -1 \le 0$, by Lemma  \ref{L210} we have
\begin{equation*}
\begin{aligned}
I = & 4\int_{\R^3} \frac 1 {|v' -v|} \frac {\langle v \rangle^k } {\langle v' \rangle^k  }  \int_{ \{ \omega: \omega \cdot   ( v'-v)  =0 \}  }  \frac 1 {\sqrt{|v'-v|^2 +|w|^2 }} (\sqrt{|v'-v|^2 + |w|^2})^{\gamma}  e^{- \frac {|v+w|^2} 2}  d w  dv'
\\
\le & 4\int_{\R^3} \frac 1 {|v' -v|^{\frac {3-\gamma} 2 }} \frac {\langle v \rangle^k } {\langle v' \rangle^k  }   \int_{ \{ \omega: \omega \cdot   ( v'-v)  =0 \}  } |w|^{\frac {\gamma-1} 2 } e^{- \frac {|v+w|^2} 2}  d w dv'.
\end{aligned}
\end{equation*}
Recall the decomposition \eqref{decomposition v} we have
\[
I \le  4\int_{\R^3} \frac 1 {|v' -v|^{\frac {3-\gamma} 2 }} \frac {\langle v \rangle^k } {\langle v' \rangle^k  } e^{- \frac {|v_\perp|^2}  2 }  \int_{ \{ \omega: \omega \cdot   ( v'-v)  =0 \}  } |w|^{\frac {\gamma-1} 2 } e^{- \frac {|v_\parallel +w|^2}  2 } d w  dv'.
\]
Since $\frac {\gamma-1} 2 > -2$, we have
\[
\int_{ \{ \omega: \omega \cdot   ( v'-v)  =0 \}  } |w|^{\frac {\gamma-1} 2 } e^{- \frac {|v_\parallel +w|^2}  2 }  d w = \int_{\R^2 } |w - v_\parallel |^{\frac {\gamma-1} 2 } e^{- \frac {|w|^2}  2 }  d  w \le C \langle v_\parallel \rangle^{\frac {\gamma-1} 2 },
\]
hence
\[
I \le  C \int_{\R^3}  |v' -v|^{\frac {\gamma-3} 2 } \frac {\langle v \rangle^k } {\langle v' \rangle^k  } e^{- \frac {|v_\perp|^2}  2 } \langle v_\parallel \rangle^{\frac {\gamma-1} 2 } dv' . 
\]
Similarly we have
\[
J \le  C\int_{\{|v- v'| >\frac {\langle v \rangle} \epsilon \cup  |v- v'| < \epsilon \langle v \rangle \} }    |v' -v|^{\frac {\gamma-3} 2 } \frac {\langle v \rangle^k } {\langle v' \rangle^k  } e^{- \frac {|v_\perp|^2}  2 } \langle v_\parallel \rangle^{\frac {\gamma-1} 2 } dv' . 
\]
We split into two regions $|v| \le 1$ and $|v| > 1$. If $|v| \le 1$, since $\frac {\gamma-3} 2  > -3$, hence
\[
I \le  C_k \int_{\R^3}  |v' -v|^{\frac {\gamma-3} 2 }\frac {1 } {\langle v' \rangle^k  }dv' \le C_k \langle v \rangle^{\frac {\gamma-3} 2 } \le  C_k \langle v \rangle^{-100}.
\]
Similarly by  Lemma \ref{L61} we have
\[
J \le  C_k \int_{\{|v- v'| >\frac {\langle v \rangle} \epsilon \cup  |v- v'| < \epsilon \langle v \rangle \} }  |v' -v|^{\frac {\gamma-3} 2 }\frac {1 } {\langle v' \rangle^k  }dv' \le C_{k, \epsilon} \langle v \rangle^\gamma. 
\]
For the case $|v|>1$, since  $|v_\perp| \le |v|$, we split it into two cases $|v_\perp| > \frac {|v|} k$ and $|v_\perp| \le \frac {|v|} k$. For the case $|v_\perp| > \frac {|v|} k$ we have
\[
I \le  C_k \int_{\R^3}  |v' -v|^{\frac {\gamma-3} 2 }  \frac {\langle v \rangle^k } {\langle v' \rangle^k  } e^{ -\frac {|v|^2} {2k^2}} dv' 
\le C_k  e^{ -\frac {|v|^2} {4k^2}} \int_{\R^3}  |v' -v|^{\frac {\gamma-3} 2 } \langle v '\rangle^{-k} d v' \le  C_k  \langle v \rangle^{-100},
\]
and by Lemma \ref{L61}
\[
J \le  C_k e^{ -\frac {|v|^2} {4k^2}} \int_{\{|v- v'| >\frac {\langle v \rangle} \epsilon \cup  |v- v'| < \epsilon \langle v \rangle \} }  |v' -v|^{\frac {\gamma-3} 2 }\frac {1 } {\langle v' \rangle^k  }dv' \le C_{k, \epsilon} \langle v \rangle^\gamma.
\]
For the case $|v_\perp| \le \frac {|v|} k$, $|v|>1$ we have
\[
|v'|^2 = |v-v'|^2 +|v|^2 +2 v\cdot(v'-v ) \ge  |v-v'|^2 +|v|^2 - 2 |v-v'| |v_\perp|  \ge  (1-\frac 1 k) (|v-v'|^2 +|v|^2 ),
\]
and since
\[
(1+\frac 1 k)^k \le e, \quad |v_\parallel|^2 \ge |v|^2-\frac 1 {k^2} |v|^2 \ge \frac 1 2 |v|^2,
\]
we deduce
\begin{equation*}
\begin{aligned}
I \le&  C \int_{\R^3}  |v' -v|^{\frac {\gamma-3} 2 } \frac {\langle v \rangle^k } {\langle v' \rangle^k  }  e^{- \frac {|v_\perp|^2}  2 } \langle v_\parallel \rangle^{\frac {\gamma-1} 2 } dv' 
\\
\le& C   \langle v \rangle^{\frac {\gamma-1} 2 } \int_{\R^3}  |v' -v|^{\frac {\gamma-3} 2 } \frac {\langle v \rangle^{k} }  {\langle |v-v'|^2 +|v|^2 \rangle^{k}} e^{- \frac {|v_\perp|^2}  2 }   dv' 
\\
\le &  C   \langle v\rangle^{\frac {\gamma-1} 2 }    \int_{\R^3}  |v' -v|^{\frac {\gamma-3} 2 }\frac {1} { (1+  \frac   {|v-v'|^2} {1+|v|^2}   )^{\frac {k} 2 } }   e^{- \frac {|v_\perp|^2}  2 }   dv' .
\end{aligned}
\end{equation*}
If we take the change of variables \eqref{change of variable r theta} we have
\[
I \le C 2\pi \langle v \rangle^{\frac {\gamma-1} 2 }  \int_{0}^\infty r^{\frac {\gamma+ 1} 2} \frac {1} { (1+  \frac   {r^2} {1+|v|^2}   )^{\frac {k} 2} }   \int_0^\pi e^{ - \frac {|v|^2\cos^2\theta} 2 } \sin \theta dr d\theta. 
\]
Taking another change of variables \eqref{change of variable x y}, recall  \eqref{beta function} we deduce
\begin{equation*}
\begin{aligned}
I \le& C \frac 1 {|v|} \langle v \rangle^{\gamma+1 } \int_{0}^\infty x^{\frac {\gamma+1 } {2}} \frac {1} { (1+ |x|^2 )^{\frac {k} 2} } \int_{-|v|}^{|v|} e^{-\frac {|y|^2} 2} dy dx 
\\
\le & C \langle v \rangle^{\gamma }   \int_{0}^\infty  x^{\frac {\gamma+1 } {2}}   \frac {1} { (1+ |x|^2 )^{\frac {k} 2} } dx \int_{-\infty}^{+\infty} e^{-\frac {|y|^2} 2} dy 
\\
\le & C\langle v \rangle^{\gamma }  \int_{0}^\infty  z^{\frac {\gamma-1 } {4}}   \frac {1} { (1+ z )^{\frac {k} 2} } dx
\\
\le & C \langle v \rangle^{\gamma }   k^{ -\frac {\gamma+3} 4},
\end{aligned}
\end{equation*}
so the term $I$ is estimated.  By \eqref{change of variable x y} and \eqref{change of variable r theta} we have $x = \frac {|v-v'|} {\langle v \rangle}$, which implies
\[
\{|v- v'| >\frac {\langle v \rangle} \epsilon \cup  |v- v'| < \epsilon \langle v \rangle \} = \{ x \le \epsilon \cup x \ge \frac 1 \epsilon \},
\]
thus  for $J$  we have
\[
J \le C \frac 1 {|v|} \langle v \rangle^{\gamma+1 } \left( \int_{0}^\epsilon + \int_{\frac 1 \epsilon}^\infty \right) x^{\frac {\gamma+1 } {2}} \frac {1} { (1+ |x|^2 )^{\frac {k} 2} } \int_{-|v|}^{|v|} e^{-\frac {|y|^2} 2} dy dx  \le C_{k, \epsilon} \langle v \rangle^\gamma,
\]
so the proof for $J$ is thus finished. For the $K$ term since
\[
K  =\langle v \rangle^{- 2}  \int_{\R^3} \int_{\mathbb{S}^2} |v-v_*|^\gamma \frac {\langle v \rangle^{k+2} } {\langle v' \rangle^{k+2}  } e^{-\frac 1 2 |v_*'|^2 }  dv_* d\sigma ,
\]
the estimate for term $K$  just follows by the estimate for term $I$. 
\end{proof}

We introduce the mild solution to the Boltzmann equation. For any $k \ge 0$, let $f(t, x, v) =  \langle v \rangle^{-k} (F(t, x, v) -\mu (v))$ in \eqref{Boltzmann equation}, then $f$ satisfies
\begin{equation}
\label{Boltzmann equation f}
\partial_t f + v \cdot \nabla_v f  +L_k  f = \Gamma_k (f, f),
\end{equation}
where 
\[
L_k f := \langle v \rangle^k Q(\mu,  f \langle v \rangle^{-k}  ) + \langle v \rangle^k Q(  f \langle v \rangle^{-k}, \mu), 
\]
and
\[
\Gamma^{\pm}_k (f , f) := \langle v \rangle^k Q^{\pm} (  f \langle v \rangle^{-k},  f \langle v \rangle^{-k} ) , \quad \Gamma_k   (f, f) := \Gamma^{+}_k (f, f) - \Gamma^{-}_k (f, f).
\]
We also have
\[
L_k f = K_k f - \nu(v) f ,\quad \nu(v) :  = \int_{\R^3} \int_{\mathbb{S}^2} |v-v_*|^\gamma b (\cos \theta) \mu(v_*)  dv_* d\sigma \ \sim \langle v \rangle^\gamma, 
\]
where $K_k : =K_{2, k} -K_{1, k}$ is defined as
\[
(K_{1, k} f) (v)  := \mu(v) \langle v \rangle^k \int_{\R^3} \int_{\mathbb{S}^2} |v-v_*|^\gamma b (\cos \theta) f(v_* ) \langle v_* \rangle^{-k}  dv_* d\sigma,
\]
and
\begin{equation*}
\begin{aligned}
(K_{2, k} f) (v) :=& \int_{\R^3} \int_{\mathbb{S}^2} |v-v_*|^\gamma b (\cos \theta) \langle v \rangle^k f(v_*' ) \langle v_*' \rangle^{-k}  \mu(v')dv_* d\sigma 
\\
&+ \int_{\R^3} \int_{\mathbb{S}^2} |v-v_*|^\gamma b (\cos \theta) \langle v \rangle^k \mu(v_*')  f(v' ) \langle v' \rangle^{-k}  dv_* d\sigma 
\\
=&2\int_{\R^3} \int_{\mathbb{S}^2} |v-v_*|^\gamma b (\cos \theta) \langle v \rangle^k \mu(v_*')  f(v' ) \langle v' \rangle^{-k}  dv_* d\sigma .
\end{aligned}
\end{equation*}
Thus the mild solution of \eqref{Boltzmann equation f} is given by 
\begin{align}
\label{mild solution}
\nonumber
f(t, x, v) =& e^{-\nu(v) t} f_0(x-vt, v) + \int_0^t e^{-\nu(v) (t-s) } (K_k f) (s, x-v(t-s), v) ds 
\\
&+  \int_0^t e^{-\nu(v) (t-s) } \Gamma_k (f, f) (s, x-v(t-s), v) ds.
\end{align}
If we define  $l_k (v, v') $ is the kernel with respect to $K_k$ such that 
\[
K_k f(v) = \int_{\R^3} l_k (v, v') f(v') d v'.
\]
For the kernel $l_k$ we have the following estimate. 
\begin{lem}\label{L63}
For any $\gamma \in (-3, 1]$, for any $k > \max \{3, 3+\gamma \}$ we have
\[
\int_{\R^3} | l_k (v, v')| dv' \le \frac {c} {k^{\frac {\gamma+3} 4}} \langle v \rangle^{\gamma}+ C_k \langle v \rangle^{\gamma-2}, \quad  \int_{\R^3} |l_k (v, v')|\langle v' \rangle^{-2} dv' \le  C_k \langle v \rangle^{\gamma-2},
\]
for some constant $C_k >0$. Moreover, for $\epsilon >0$ small enough, we have
\[
\int_{ \{|v-v'| >\frac {\langle v \rangle} \epsilon \cup  | v- v' | < \epsilon \langle v \rangle \} } |l_k (v, v')| dv' \le C_{k, \epsilon} \langle v \rangle^{\gamma }, \quad \lim_{\epsilon \to 0} C_{k, \epsilon } =0.
\]
\end{lem}
\begin{proof}
It is easily seen that
\begin{equation*}
\begin{aligned}
\int_{\R^3} | l_k (v, v')| dv' \le& 2 \int_{\R^3} \int_{\mathbb{S}^2}    |v-v_*|^\gamma b(\cos \theta) \frac {\langle v \rangle^k } {\langle v' \rangle^k  } e^{-\frac 1 2 |v_*'|^2 } dv_* d\sigma
\\
&+\mu(v) \langle v \rangle^k \int_{\R^3} \int_{\mathbb{S}^2} |v-v_*|^\gamma b (\cos \theta)  \langle v_* \rangle^{-k}  dv_* d\sigma,
\end{aligned}
\end{equation*}
we easily conclude by Lemma \ref{L29}, Lemma \ref{L61} and Lemma \ref{L62}.
\end{proof}

For the mild solution $f$,  we have the following lemma on local existence. 

\begin{lem}\label{L64}
(Local existence) Suppose $\gamma \in (-3, 1]$, $F_0 =\mu +f_0 \ge 0$.  For any $k> \max \{3, 3+\gamma \}$ suppose $\Vert \langle v \rangle^k f_0 \Vert_{L^\infty} < +\infty$. Then there exists a positive time 
\[
t_1 = C_k (1+\Vert \langle v \rangle^k f_0 \Vert_{L^\infty})^{-1},
\]
such that the Boltzmann equation \eqref{Boltzmann equation} has a unique mild solution $F = \mu +f \ge 0$ in $[0, t_1]$ satisfies 
\[
\sup_{0 \le s \le t_1}\Vert \langle v \rangle^k f (s) \Vert_{L^\infty} \le 2 \Vert \langle v \rangle^k f_0 \Vert_{L^\infty}.
\]
\end{lem}
\begin{proof}
The proof is similar to Proposition 2.1 in \cite{DHWY} thus omitted. 
\end{proof}

We   give an upper bound for the nonlinear term.

\begin{lem}\label{L65} Let $\gamma \in (-3, 1] $, for any $k > \max \{3, 3+\gamma\}$, $\alpha \ge 0$, for any $s \ge 0, y \in \T^3$, for any smooth function $f$  it holds that
\[
|\langle v \rangle^\alpha \Gamma_k^{-}(f, f)(s, y, v) | \le C \nu(v)\Vert \langle v \rangle^{\alpha}  f(s)\Vert_{L^\infty_{x, v}} \Vert f(s) \Vert_{L^\infty_{x, v}}^{\frac {p+1} {2p}} \left (\int_{\R^3} |f(s, y, v')| dv' \right)^{\frac {p-1} {2p}},
\]
similarly
\[
|\langle v \rangle^\alpha \Gamma_k^{+}(f, f)(s, y, v) | \le C \nu(v)\Vert \langle v \rangle^{\alpha}  f(s)\Vert_{L^\infty_{x, v}} \Vert \langle v \rangle f(s) \Vert_{L^\infty_{x, v}}^{\frac {p+1} {2p}} \left (\int_{\R^3} |f(s, y, v')| dv' \right)^{\frac {p-1} {2p}},
\]
for some constant $p>1$  close to 1 which only depends on $\gamma$ (such p is fixed and used later).
\end{lem}

\begin{proof}
Fix $p>1$ close enough to 1 and $\epsilon>0$ small enough such that
\begin{equation}
\label{requirement p}
-3< p\gamma <\frac 3 2, \quad \frac {4(p-1)} {p+1} \le 1, \quad \frac {p-1} {2p} \le \frac 1 2 + \frac \gamma 6, \quad -3<  p\gamma+\epsilon \gamma \le -2 ,\quad \epsilon \frac {2p} {p-1} \le 1.
\end{equation}
For the term $\Gamma_k^-(f, f) $, we easily compute
\begin{equation*}
\begin{aligned}
|\langle v \rangle^\alpha \Gamma_k^{-}(f, f)(s, y, v) | \le & C \Vert \langle v \rangle^\alpha f(s)\Vert_{L^\infty_{x, v}} \int_{\R^3} |v-v_*|^\gamma \langle v \rangle^{-k}  f(s, y, v_*) dv_*
\\
\le &  C  \Vert \langle v \rangle^\alpha f(s)\Vert_{L^\infty_{x, v}} \left(\int_{\R^3} |v-v_*|^{p\gamma} \langle v_* \rangle^{-k}  dv_*  \right)^{\frac 1 p} \left(\int_{\R^3}  \langle v_* \rangle^{-k} f(s, y, v_*)^{\frac {p} {p-1}} dv_*\right)^{\frac {p -1} p}
\\
\le &  C  \Vert \langle v \rangle^\alpha f(s)\Vert_{L^\infty_{x, v}} \nu(v)   \left(\int_{\R^3}  \langle v_* \rangle^{-2k} dv_*\right)^{\frac {p -1} {2p}}  \left(\int_{\R^3}   f(s, y, v_*)^{\frac {2p} {p-1}} dv_*\right)^{\frac {p -1} {2p}}
\\
\le &C \nu(v)\Vert \langle v \rangle^\alpha  f(s)\Vert_{L^\infty_{x, v}} \Vert f(s) \Vert_{L^\infty_{x, v}}^{\frac {p+1} {2p}} \left (\int_{\R^3} |f(s, y, v')| dv' \right)^{\frac {p-1} {2p}}.
\end{aligned}
\end{equation*}
For the term $\Gamma_k^+(f, f)$, since $\langle v \rangle^\alpha \le C_\alpha \langle v_* \rangle^\alpha +C_\alpha \langle v_*' \rangle^\alpha$, we  have
\begin{equation*}
\begin{aligned}
|\langle v \rangle^\alpha \Gamma_k^{+}(f, f)(s, y, v) | \le &  C \int_{\R^3} \int_{\mathbb{S}^2 } |v-v_*|^\gamma \frac {\langle v \rangle^k} {\langle v' \rangle^k \langle v_*' \rangle^k  } |f(s, y, v_*' ) \langle v' \rangle^\alpha f(s, y, v' )| dv_* d\sigma
\\
&+ C\int_{\R^3} \int_{\mathbb{S}^2 } |v-v_*|^\gamma \frac {\langle v \rangle^k} {\langle v' \rangle^k \langle v_*' \rangle^k  } | \langle v_*' \rangle^\alpha f(s, y, v_*' ) f(s, y, v' )| dv_* d\sigma : =I_1 +I_2.
\end{aligned}
\end{equation*}
Without loss of generality we only prove $I_2$ in the following, we have
\begin{equation*}
\begin{aligned}
I_2 =& \int_{\R^3} \int_{\mathbb{S}^2 } |v-v_*|^\gamma \frac {\langle v \rangle^k} {\langle v' \rangle^k \langle v_*' \rangle^k  } |\langle v_*' \rangle^\alpha   f(s, y, v_*' ) f(s, y, v' )| dv_* d\sigma
\\
\le & C \Vert \langle v \rangle^\alpha  f(s) \Vert_{L^\infty_{x, v}} \left(\int_{\R^3} \int_{\mathbb{S}^2} |v-v_*|^{p\gamma +\epsilon p} \frac {\langle v \rangle^{pk}} {\langle v' \rangle^{pk} \langle v_*' \rangle^{pk}  }dv d\sigma \right)^{\frac 1 p} \left(\int_{\R^3} \int_{\mathbb{S}^2} |v-v_*|^{-\epsilon \frac p {p-1} }| f(s, y, v') |^{\frac p {p-1}} dv_* d\sigma  \right)^{1-\frac 1 p}.
\end{aligned}
\end{equation*}
By Lemma \ref{L210}, for any function $g$ we have
\begin{equation*}
\begin{aligned}
&\int_{\R^3} \int_{\mathbb{S}^2} |v-v_*|^{-\epsilon \frac p {p-1} } | g (v')| dv_* d\sigma  
\\
\le &4 \int_{\R^3}  \frac 1 {|v' -v|} |g (v')| \int_{ \{ w: w \cdot   ( v'-v)  =0 \}  }  \frac 1 {\sqrt{|v'-v|^2 +|w|^2 }} (\sqrt{|v'-v|^2 + |w|^2})^{-\epsilon \frac p {p-1} }   d w  dv'
\\
\le &4 \int_{\R^3}  \frac 1 {|v' -v|} | g(v') | \int_{ \R^2 }  \frac 1 {(|v'-v|^2 +|w|^2 )^{\frac {1+ \epsilon \frac p {p-1} } 2} }   d w  dv'.
\end{aligned}
\end{equation*}
By a change of variable $w = |v-v'| x$, and since  $1+ \epsilon \frac p {p-1} >1$,  the integral is integrable and we have
\begin{align}
\label{inequality epsilon g}
\nonumber
\int_{\R^3} \int_{\mathbb{S}^2} |v-v_*|^{-\epsilon \frac p {p-1} } | g (v')| dv_* d\sigma   \le & C  \int_{\R^3}  |v' -v|^{-\epsilon \frac p {p-1}} |g(v')|   \int_{ \R^2 }  \frac 1 {(1+|x|^2 )^{\frac {1+ \epsilon \frac p {p-1} } 2} }   dx  dv' 
\\
\le& C \int_{\R^3}  |v' -v|^{-\epsilon \frac p {p-1} }  | g(v') |    dv'.
\end{align}
So using \eqref{inequality epsilon g} and Cauchy-Schwarz inequality we have
\begin{equation*}
\begin{aligned}
&\int_{\R^3} \int_{\mathbb{S}^2} |v-v_*|^{-\epsilon \frac p {p-1} }| f(s, y, v') |^{\frac p {p-1}} dv_* d\sigma  
\\
\le& C \left(\int_{\R^3}  |v-v'|^{-\epsilon \frac {2p} {p-1} } \langle v' \rangle^{-4} dv'  \right)^{\frac 1 2} \left( \int_{\R^3}  \langle v' \rangle^{4}  | f(s, y, v') |^{\frac {2p} {p-1}}  dv' \right)^{\frac 1 2} 
\\
\le & C\langle v \rangle^{-\epsilon \frac {p} {p-1} } \Vert \langle v \rangle^{4} |f|^{\frac {p+1} {p-1}} \Vert_{L^\infty_{x, v}}^{1/2}  \left( \int_{\R^3}  | f(s, y, v') | dv' \right)^{\frac 1 2}  
\\
\le& C\langle v \rangle^{-\epsilon \frac {p} {p-1} } \Vert \langle v \rangle^{\frac {4(p-1)} {p+1}} |f|\Vert_{L^\infty_{x, v}}^{\frac {p+1} {2(p-1)} }  \left( \int_{\R^3}  | f(s, y, v') | dv' \right)^{\frac 1 2} .
\end{aligned}
\end{equation*}
By Lemma \ref{L211} we have
\[
\left(\int_{\R^3} \int_{\mathbb{S}^2} |v-v_*|^{p\gamma +\epsilon p} \frac {\langle v \rangle^{pk}} {\langle v' \rangle^{pk} \langle v_*' \rangle^{pk}  }dv d\sigma \right)^{\frac 1 p}  \lesssim \langle v \rangle^{\gamma+\epsilon}.
\]
Gathering the terms two we have
\[
I_2 \le C \langle v \rangle^{\gamma+\epsilon} \langle v \rangle^{-\epsilon}\Vert \langle v \rangle^\alpha  f(s) \Vert_{L^\infty_{x, v}} \Vert \langle v \rangle f(s) \Vert_{L^\infty_{x, v}}^{\frac {p+1} {2p} } \left( \int_{\R^3}  | f(s, y, v') | dv' \right)^{\frac {p-1} {2p}}, 
\]
the proof  is thus finished. 
\end{proof}

For any $\beta \ge 0$, let  $h(t, x, v) = \langle v \rangle^\beta f(t, x, v)$, where $f$ is a solution to \eqref{Boltzmann equation f}. it is easily seen that $h$ satisfies 
\begin{equation}
\label{Boltzmann equation h}
\partial_t h + v \cdot \nabla_x h + \nu(v) h =\Gamma_{k+ \beta} (h, h) + K_{k+ \beta} h,
\end{equation}
with 
\[
K_{k + \beta}  h (v) = \langle v \rangle^\beta K_k(  \langle v \rangle^{-\beta} h )  (v), \quad \Gamma_{k + \beta} (h, h) = \langle v \rangle^\beta \Gamma_k (  \langle v \rangle^{- \beta} h , \langle v \rangle^{-\beta} h )= \langle v \rangle^\beta \Gamma_k  (f, f). 
\]
For the kernel of $K_{k +\beta}$ we have
\[
K_{k+\beta} f(v) = \int_{\R^3} l_{k +\beta} (v, v') f(v') d v',\quad  l_{k+ \beta} (v, v') = l_k (v, v') \langle v \rangle^\beta \langle v' \rangle^{-\beta},
\]
it is easily seen that $l_{k +\beta}$  still satisfies Lemma \ref{L65}, with $C_k$ replaced by $C_{k+\beta} =C_{k, \beta} $.  The mild solution to \eqref{Boltzmann equation h} is given by 
\begin{align}
\label{mild solution h}
\nonumber
h(t, x, v) =& e^{-\nu(v) t} h_0(x-vt, v) + \int_0^t e^{-\nu(v) (t-s) } (K_{k +\beta}  h) (s, x-v(t-s), v) ds 
\\ 
&+  \int_0^t e^{-\nu(v) (t-s) } \Gamma_{k + \beta} (h, h) (s, x-v(t-s), v) ds.
\end{align}

For the mild solution $h$ we have the following estimate.

\begin{lem}\label{L66}
Suppose $f$ and $h$ satisfy \eqref{Boltzmann equation f} and \eqref{Boltzmann equation h}. For any $\gamma \in (-3, 1]$, there exists a constant $k_0 > \max \{ 3, 3+\gamma \}$ such that  for any $k \ge k_0, \beta \ge \max \{3, 3+\gamma\}  $  it holds that
\begin{equation*}
\begin{aligned}
\sup_{0 \le s \le t} \Vert h(s) \Vert_{L^\infty_{x, v}}  \le& C_{k, \beta} (\Vert h_0 \Vert_{L^\infty_{x, v}}  +\Vert h_0 \Vert_{L^\infty_{x, v}} ^2 +  \sqrt{H (F_0 )}  +  H (F_0) ) 
\\
&+ C_{k, \beta} \sup_{t_1 \le s \le t, y \in \T^3 } \left \{ \Vert h(s) \Vert_{L^\infty_{x, v}}^{\frac {3p+1} {2p}}\int_{\R^3} \left ( |f(s, y, v')| dv' \right)^{\frac {p-1} {2p}} \right\},
\end{aligned}
\end{equation*}
for some constant $C_{k, \beta} \ge 1$, where $t_1$ is defined in Lemma \ref{L64}.
\end{lem}
\begin{proof} By \eqref{mild solution h} we have
\begin{equation*}
\begin{aligned}
|h(t, x, v)| \le & e^{-\nu(v) t} \Vert h_0\Vert_{L^\infty_{x, v}} + \int_0^t e^{-\nu(v) (t-s) } |(K_{k + \beta} h) (s, x-v(t-s), v) |ds 
\\ 
&+  \int_0^t e^{-\nu(v) (t-s) } |\Gamma_{k +\beta} (h, h) (s, x-v(t-s), v)| ds :=  e^{-\nu(v) t} \Vert h_0\Vert_{L^\infty_{x, v}} +J_2 +J_3.
\end{aligned}
\end{equation*}
For the $J_3$ term, since $\beta \ge 1$, by Lemma \ref{L65} we have
\[
|\Gamma_{k +\beta} (h, h)(s, y, v)| = |\langle v \rangle^\beta \Gamma_k ( f, f)(s, y, v)|  \le C_{k, \beta}  \langle v \rangle^\gamma  \sup_{0 \le s \le t, y \in \T^3 } \left \{ \Vert h(s) \Vert_{L^\infty_{x, v}}^{\frac {3p+1} {2p}} \left ( \int_{\R^3}  |f(s, y, v')| dv' \right)^{\frac {p-1} {2p}} \right\},
\]
hence
\begin{equation*}
\begin{aligned}
J_3 \le & C_{k, \beta} \int_0^t e^{-\nu(v) (t-s)} \langle v \rangle^\gamma ds  \sup_{0 \le s \le t, y \in \T^3 } \left \{ \Vert h(s) \Vert_{L^\infty_{x, v}}^{\frac {3p+1} {2p}}    \left ( \int_{\R^3} |f(s, y, v')| dv' \right)^{\frac {p-1} {2p}} \right\} 
\\
\le & C_{k, \beta}   \sup_{0 \le s \le t, y \in \T^3} \left \{ \Vert h(s) \Vert_{L^\infty_{x, v}}^{\frac {3p+1} {2p}}  \left (   \int_{\R^3}  |f(s, y, v')| dv' \right)^{\frac {p-1} {2p}} \right\}. 
\end{aligned}
\end{equation*}
For the $J_2$ term, denote $\tilde{x} = x - v(t-s) $, we have
\[
J_2 \le \int_0^t e^{-\nu(v) (t-s) }\int_{\R^3}| l_{k +\beta} (v, v') h(s, \tilde{x}, v')| dv' ds ,
\] 
by  \eqref{mild solution h} again we have 
\begin{equation*}
\begin{aligned}
J_2 \le& \int_0^t e^{-\nu(v) (t-s) }\int_{\R^3}| l_{k +\beta} (v, v')|  e^{-\nu(v') s}  | h_0 (\tilde{x} -v' s , v') | dv' ds  
\\
&+  \int_0^t e^{-\nu(v) (t-s) }\int_{\R^3} | l_{k+ \beta} (v, v')| \int_0^s  e^{-\nu(v') (s-\tau )} | \Gamma_{k +\beta} (h, h)| (\tau, \tilde{x} -v' (s-\tau), v') d \tau dv' ds 
\\
&+  \int_0^t e^{-\nu(v) (t-s) }  \int_{\R^3}  \int_{\R^3} | l_{k +\beta} (v, v')  l_{k +\beta} (v', v'')  |  \int_0^s  e^{-\nu(v') (s-\tau )}   | h(\tau, \tilde{x} -v' (s-\tau), v'') | d v''d \tau dv' ds 
\\
:=& J_{21} +J_{22} +J_{23}.
\end{aligned}
\end{equation*}
For the $J_{21}$ term by Lemma \ref{L63} we have
\[
J_{21}  \le C_{k, \beta} \Vert h_0 \Vert_{L^\infty_{x, v}} \int_0^t e^{-\nu(v) (t-s) } \langle v \rangle^\gamma ds \le C_{k, \beta}  \Vert h_0 \Vert_{L^\infty_{x, v}}. 
\]
For the $J_{22}$ term  by Lemma \ref{L65} we have
\begin{equation*}
\begin{aligned}
J_{22} \le & C_{k, \beta}  \sup_{0 \le s \le t, y \in \T^3 }  \left \{ \Vert h(s) \Vert_{L^\infty_{x, v}}^{\frac {3p+1} {2p}}  \left (   \int_{\R^3} |f(s, y, v')| dv' \right)^{\frac {p-1} {2p}} \right\}    \int_0^t e^{-\nu(v) (t-s) }\int_{\R^3} | l_{k +\beta} (v, v')|  \int_0^s e^{-\nu(v') ( s - \tau )} \langle v'  \rangle^\gamma d\tau dv' ds 
\\
\le &  C_{k, \beta} \sup_{0 \le s \le t, y \in \T^3 }  \left \{ \Vert h(s) \Vert_{L^\infty_{x, v}}^{\frac {3p+1} {2p}} \left (  \int_{\R^3}  |f(s, y, v')| dv' \right)^{\frac {p-1} {2p}} \right\} \int_0^t e^{-\nu(v) (t-s)} \langle v \rangle^\gamma ds 
\\
\le& C_{k, \beta} \sup_{0 \le s \le t, y \in \T^3 }  \left \{ \Vert h(s) \Vert_{L^\infty_{x, v}}^{\frac {3p+1} {2p}}   \left ( \int_{\R^3}  |f(s, y, v')| dv' \right)^{\frac {p-1} {2p}} \right\}. 
\end{aligned}
\end{equation*}
For  term $J_{23}$, we first split it into two parts $|v| \le N$ and $|v| \ge N$ for some constant $N>0$ large to be fixed later. For the case $|v| \ge N$, by Lemma \ref{L63} we have
\[
\int_{\R^3} |l_{k + \beta} (v', v'')| dv'' \le \frac {c} {(k+\beta)^{\frac {\gamma+3} 4}} \langle v' \rangle^{\gamma}+ C_{k, \beta}  \langle v'  \rangle^{\gamma-2},
\]
which implies
\[
\int_0^s  e^{-\nu(v') (s-\tau )}  \int_{\R^3} |l_{k + \beta} (v', v'')| dv'' d\tau \le  \frac {c} {k^{\frac {\gamma+3} 4}} + C_{k, \beta}  \langle v'  \rangle^{-2}.
\]
Since $|v| \ge N$ implies $\langle v \rangle^{-2} \le \frac 1 {N^2}$, using Lemma \ref{L63} again we have
\begin{align}
\label{v ge N}
\nonumber
&\int_{\R^3}  |l_{k + \beta} (v, v')| \int_0^s  e^{-\nu(v') (s-\tau )}  \int_{\R^3} |l_{k +\beta} (v', v'')| dv'' d\tau dv' 
\\
\le& \frac {c} {k^{\frac {\gamma+3} 4}} \int_{\R^3}  |l_{k +\beta} (v, v')| dv' +  C_{k, \beta}  \int_{\R^3} |l_{k +\beta} (v, v')| \langle v' \rangle^{-2} dv' 
\le  \frac {c^2} {k^{\frac {\gamma+3} 2}}   \langle v \rangle^{\gamma}+ C_{k, \beta}  \langle v \rangle^{\gamma-2} 
\le \langle v \rangle^\gamma \left ( \frac {c^2} {k^{\frac {\gamma+3} 2}}   + \frac {C_{k, \beta}  }{N^2}  \right), 
\end{align}
we deduce 
\[
J_{23} \le \left( \frac {c^2} {k^{\frac {\gamma+3} 2}}   + \frac {C_{k, \beta}}{N^2} \right)  \sup_{0 \le s \le t }\Vert h(s) \Vert_{L^\infty_{x, v}}\int_0^t e^{-\nu(v) (t-s) } \langle v \rangle^\gamma ds \le \left( \frac {c^2} {k^{\frac {\gamma+3} 2}}   + \frac {C_{k, \beta} }{N^2} \right) \sup_{0 \le s \le t }\Vert h(s) \Vert_{L^\infty_{x, v}}.
\]
For the case $|v| \le N$, since $ l_{k + \beta}  (v, v')$ is unbounded, by  Lemma \ref{L63} we have for any $N, k,\beta$  we can find a bounded compact support function $l_{k, N, \beta}$ such that
\begin{equation}
\label{l k N beta}
l_{k, N, \beta} (v, v')  : = l_{k+ \beta} (v, v') 1_{\frac {\langle v \rangle} {C_{k, N, \beta}} \le  |v-v'|  \le C_{k, N, \beta} \langle v \rangle  }, \quad \int_{\R^3} | l_{k+ \beta} (v,  v') -l_{k, N, \beta} (v, v')|  dv' \le \frac {C_{k, \beta}} {N } \langle v \rangle^\gamma,\quad \forall v \in \R^3,
\end{equation}
for some large constant $C_{k, N, \beta} >0 $. By
\begin{align}
\label{l k beta decomposition}
\nonumber
l_{k + \beta} (v, v') l_{k + \beta}  (v', v'') =& ( l_{k + \beta} (v, v') - l_{k, N, \beta}(v, v') ) l_{k + \beta} (v', v'') +l_{k, N, \beta} (v, v')  ( l_{k+ \beta} (v', v'') - l_{k, N, \beta}(v', v'') )  
\\
&+  l_{k, N, \beta}(v, v')  l_{k, N, \beta}(v', v''),
\end{align}
we split $J_{23}$ into three terms respectively. For the first term we have
\begin{equation*}
\begin{aligned}
&\int_0^t e^{-\nu(v) (t-s) }  \int_{\R^3}  \int_{\R^3} |( l_{k +\beta} (v, v') - l_{k, N, \beta}(v, v') ) l_{k +\beta} (v', v'')  |  \int_0^s  e^{-\nu(v') (s-\tau )}   | h(\tau, \tilde{x} -v' (s-\tau), v'') | d v''d \tau dv' ds 
\\
\le &\frac {C_{k, \beta}} {N} \sup_{0 \le s \le t }\Vert h(s) \Vert_{L^\infty_{x, v}}\int_0^t e^{-\nu(v) (t-s) }  \langle v \rangle^\gamma ds \int_0^{s} e^{-\nu(v') (s-\tau) }  \langle v' \rangle^\gamma d\tau   \le \frac {C_{k, \beta}} {N} \sup_{0 \le s \le t }\Vert h(s) \Vert_{L^\infty_{x, v}},
\end{aligned}
\end{equation*}
the second term can be estimated similarly.  For the third term, by \eqref{l k N beta} we have $l_{k, N, \beta}(v, v')$ and  $l_{k, N, \beta }(v', v'')$ is supported in 
\[
\frac {\langle v \rangle} {C_{k, N, \beta}} \le  |v-v'|  \le C_{k, N, \beta}  \langle v \rangle , \quad \frac {\langle v' \rangle} {C_{k, N, \beta}} \le  |v'-v''|  \le C_{k, N, \beta} \langle v' \rangle,
\]
since $|v| \le N$, which implies $l_{k, N, \beta}(v', v'') l_{k, N, \beta}(v', v'')$ is supported in  $|v| \le N, |v'|\le C_{k, N, \beta}',  |v''|\le C_{k, N, \beta}'$
for some constant $C_{k, N, \beta}'>0$. We split it into two parts, $\tau \in [s-\lambda, s]$ and $\tau \in [0, s-\lambda ]$, where $\lambda>0$ is a small constant to be fixed later. For the case $\tau \in [s-\lambda, s]$, since $|v'| \le C_{k, N, \beta}'$,   we have 
\begin{equation*}
\begin{aligned}
&\int_0^t e^{-\nu(v) (t-s) }  \int_{\R^3}  \int_{\R^3} | l_{k, N, \beta} (v, v')  l_{k, N, \beta}(v', v'')  |  \int_{s-\lambda}^s  e^{-\nu(v') (s-\tau )}   | h(\tau, \tilde{x} -v' (s-\tau), v'') | d v''d \tau dv' ds 
\\
\le &C_{k, N, \beta} \sup_{0 \le s \le t }\Vert h(s) \Vert_{L^\infty_{x, v}} \int_0^t e^{-\nu(v) (t-s) }  \langle v \rangle^\gamma ds \int_{s-\lambda}^{s} e^{-\nu(v') (s-\tau) }  \langle v' \rangle^\gamma d\tau 
\\
\le &C_{k, N,  \beta} \sup_{0 \le s \le t }\Vert h(s) \Vert_{L^\infty_{x, v}} (1- e^{ - \nu(v') \lambda})  \le C_{k, N, \beta} \lambda \sup_{0 \le s \le t }\Vert h(s) \Vert_{L^\infty_{x, v}}.
\end{aligned}
\end{equation*}
For the case $\tau \in [0, s-\lambda ]$, first we have
\begin{equation}
\label{l k N beta bound}
l_{k, N, \beta} (v, v') \le C_{k, N,\beta},\quad l_{k, N, \beta}  (v', v'') \le C_{k, N, \beta},\quad \frac 1 {C_{k, N, \beta}} \le \nu(v) \le C_{k, N, \beta},\quad \frac 1 {C_{k, N, \beta}} \le  \nu(v' ) \le C_{k, N, \beta},
\end{equation}
and by Lemma \ref{L213} we have
\begin{align}
\label{entropy for h}
\nonumber
&\int_{|v'| \le C_{k, N, \beta}', |v''| \le C_{k, N, \beta}' } | h(\tau, \tilde{x} -v' (s-\tau), v'') | dv' dv''
\\ \nonumber
 = &  \int_{|v'| \le C_{k, N, \beta}', |v''| \le C_{k, N, \beta}' }  \frac {| F(\tau, \tilde{x} -v' (s-\tau), v'') -\mu(v'') | } {\langle v'' \rangle^{k+\beta}} dv' dv''
\\ \nonumber
\le &C_{k, N, \beta} \int_{|v'| \le C_{k, N, \beta}', |v''| \le C_{k, N, \beta}' }  \frac {| F(\tau, \tilde{x} -v' (s-\tau), v'') -\mu(v'') | } {\sqrt{\mu(v'' )  }} I_{ \{ | F(\tau, \tilde{x} -v' (s-\tau), v'') -\mu(v'') | \le \mu (v'') \}} dv' dv''
\\ \nonumber
& + C_{k, N, \beta} \int_{|v'| \le C_{k, N, \beta}', |v''| \le C_{k, N, \beta}' }  | F(\tau, \tilde{x} -v' (s-\tau), v'') -\mu(v'') |  I_{ \{ | F(\tau, \tilde{x} -v' (s-\tau), v'') -\mu(v'') | \ge \mu (v'') \}} dv' dv''
\\ \nonumber
\le & C_{k, N, \beta} \frac 1 {(s-\tau)^\frac 3 2}  \left( \int_{\T^3 } \int_{|v''| \le C_{k, N, \beta}' } \frac {| F(\tau, y, v'') -\mu(v'') |^2 } {\mu(v'' )  } I_{ \{ | F(\tau, y, v'') -\mu(v'') | \le \mu (v'') \}} dv'' dy \right)^{\frac 1 2}
\\ \nonumber
&+ C_{k, N, \beta} \frac 1 {(s-\tau)^3} \int_{\T^3 } \int_{|v''| \le C_{k, N, \beta}' }  | F(\tau, y, v'') -\mu(v'') | I_{ \{ | F(\tau, y, v'') -\mu(v'') | \ge \mu (v'') \}} dv'' dy
\\ 
\le & C_{k, N, \beta} \frac 1 {(s-\tau)^\frac 3 2} \sqrt{ H (F_0)} +C_{k, N, \beta} \frac 1 {(s-\tau)^3}  H(F_0),
\end{align}
where we have made a change of variable $ y = \tilde{x} -v' (s-\tau) $.  Since $s-\tau \ge \lambda$, so we have
\begin{equation*}
\begin{aligned}
&\int_0^t e^{-\nu(v) (t-s) }  \int_{\R^3}  \int_{\R^3} | l_{k, N, \beta} (v, v')  l_{k, N, \beta}(v', v'')  |  \int_{0}^{s-\lambda}  e^{-\nu(v') (s-\tau )}   | h(\tau, \tilde{x} -v' (s-\tau, v'')) | d v''d \tau dv' ds 
\\
\le &C_{k, N, \beta} \int_0^t e^{-c(t-s) } \int_{0}^{s-\lambda}  e^{-  c(s-\tau )}  \int_{|v'| \le C_{k, N, \beta}', |v''| \le C_{k, N, \beta}' }   | h  (\tau, \tilde{x} -v' (s-\tau), v'') | d v''d \tau dv' ds 
\\
\le &C_{k, N, \beta} \lambda^{-\frac 3 2} \sqrt{ H (F_0)} +C_{k, N, \beta} \lambda^{-3}   H (F_0). 
\end{aligned}
\end{equation*}
Gathering all the terms and taking supremum we have
\begin{equation*}
\begin{aligned}
\sup_{0 \le s \le t} \Vert h(s) \Vert_{L^\infty_{x, v}}  \le & C_{k, \beta} \Vert h_0 \Vert_{L^\infty_{x, v}}  +   \left ( \frac {c^2} {k^{\frac {\gamma+3} 2}}   + \frac {C_{k, \beta}} {N} + C_{k, N, \beta} \lambda \right)\sup_{0 \le s \le t }\Vert h(s) \Vert_{L^\infty_{x, v}}
\\
& + C_{k, N, \beta} \lambda^{-\frac 3 2} \sqrt{H ( F_0 )} +C_{k, N, \beta} \lambda^{-3}   H (F_0) +  C_{k, \beta}   \sup_{0 \le s \le t, y \in \T^3  } \left \{ \Vert h(s) \Vert_{L^\infty_{x, v}}^{\frac {3p+1} {2p}}\int_{\R^3} \left ( |f(s, y, v')| dv' \right)^{\frac {p-1} {2p}} \right\}.
\end{aligned}
\end{equation*}
First fix $\beta \ge 0$, then choose $k$ large, then let $N$ be sufficiently large and finally let $\lambda$ be sufficiently small such that
\[
\frac {c^2} {k^{\frac {\gamma+3} 2}}   + \frac {C_{k, \beta}} {N} + C_{k, N, \beta} \lambda \le \frac 1 2, 
\]
which implies
\[
\sup_{0 \le s \le t} \Vert h(s) \Vert_{L^\infty_{x, v}}  \le  C_{k, \beta} (\Vert h_0 \Vert_{L^\infty_{x, v}}  +  \sqrt{ H (F_0)} +   H (F_0) )+  C_{k, \beta}   \sup_{0 \le s \le t, y \in \T^3 } \left \{ \Vert h(s) \Vert_{L^\infty_{x, v}}^{\frac {3p+1} {2p}}\int_{\R^3} \left ( |f(s, y, v')| dv' \right)^{\frac {p-1} {2p}} \right\},
\]
using Lemma \ref{L64} we have
\[
\sup_{0 \le s \le t_1, y \in \T^3} \left \{ \Vert h(s) \Vert_{L^\infty_{x, v}}^{\frac {3p+1} {2p}}\int_{\R^3} \left ( |f(s, y, v')| dv' \right)^{\frac {p-1} {2p}}  \right\} \le C \sup_{0 \le s \le t_1} \Vert h(s)\Vert_{L^\infty_{x, v}}^2 \le C \Vert h_0\Vert_{L^\infty_{x, v}}^2,
\]
so the proof is thus finished. 
\end{proof}

\begin{lem}\label{L67} Suppose $\gamma \in (-3, 1]$ and $k, \beta >\max \{3, 3+\gamma \}$, then for any smooth function  $f$ and $h$ satisfy \eqref{Boltzmann equation f} and \eqref{Boltzmann equation h} we have
 \begin{equation*}
\begin{aligned}
\int_{\R^3} |f(t, x, v)|dv \le & \int_{\R^3} e^{-\nu(v)t} |f_0(x-vt, v ) | dv +C_{k, N, \beta} \lambda^{-\frac 3 2 } \sqrt{H (F_0)} +C_{k, N, \beta} \lambda^{- 3 } H (F_0) 
\\
&+ C_{k, \beta}(\lambda + \frac {1} {N^{\frac {\beta-3} 2}} ) (\sup_{0 \le s \le t} \Vert h(s) \Vert_{L^\infty_{x, v}} + \sup_{0 \le s \le t} \Vert h(s) \Vert_{L^\infty_{x, v}}^2)
\\
&+ C_{k, N, \beta} \lambda^{- 3}  (\sqrt{ H(F_0)}  + H (F_0) )^{1-\frac 1 p} \sup_{0 \le s \le t} \Vert h(s) \Vert_{L^\infty_{x, v}}^{1+\frac 1 p},
\end{aligned}
\end{equation*}
where $\lambda>0, N \ge 1$ are to be chosen later. Recall that $p>1$ is defined in \eqref{requirement p}. 
\end{lem}
\begin{proof} By \eqref{mild solution} we have
\begin{equation*}
\begin{aligned}
\int_{\R^3} |f(t, x, v)|dv  \le& \int_{\R^3} e^{-\nu(v) t} |f_0(x-vt, v ) | dv  + \int_{0}^t \int_{\R^3} e^{-\nu(v) (t-s)} |(K_k f) (s, x-v (t-s),  v ) | dv ds
\\
& + \int_{0}^t \int_{\R^3} e^{-\nu(v) (t-s)} |\Gamma_k ( f, f) (s, x-v (t-s),  v ) | dv ds 
\\
:=& \int_{\R^3} e^{-\nu(v) t} |f_0(x-vt, v ) | dv  + H_1+H_2 .
\end{aligned}
\end{equation*}
For the term $H_1$, recall
\[
h(t, x, v) = \langle v \rangle^\beta f(t, x, v), \quad  l_{k+ \beta}  (v, v') = l_k(v, v') \langle v \rangle^\beta \langle v' \rangle^{-\beta}, 
\]
hence
\[
H_1 \le \int_{0}^t \int_{\R^3} e^{-\nu(v)(t-s)} \langle v \rangle^{-\beta} \int_{\R^3}| l_{k +\beta} (v, v')  h(s, x-v(t-s), v' )   | dv' dv ds.
\]
We split it into two  case $s \in [ t-\lambda, t]$ and $s \in [0, t-\lambda]$, where $\lambda$ is a small constant to be fixed later. For the case $s \in [ t-\lambda, t]$, since $\beta -\gamma>3$ we have
\begin{equation*}
\begin{aligned}
&\int_{t-\lambda}^t \int_{\R^3} e^{-\nu(v)(t-s)} \langle v \rangle^{-\beta} \int_{\R^3}| l_{k+ \beta}  (v, v')  h(s, x-v(t-s), v' )   | dv' dv ds 
\\
\le&\sup_{0\le s \le t} \Vert h(s) \Vert_{L^\infty_{x, v}} \int_{t-\lambda}^t  \int_{\R^3} e^{-\nu(v )(t-s)} \langle v \rangle^{-\beta}  \int_{\R^3}| l_{k+ \beta} (v, v')    | dv' dv ds 
\\
\le&C_{k, \beta} \sup_{0\le s \le t} \Vert h(s) \Vert_{L^\infty_{x, v}} \int_{t-\lambda}^t  \int_{\R^3}  \langle v \rangle^{-\beta}  \langle v \rangle^\gamma dv ds \le C_{k, \beta}  \lambda\sup_{0\le s \le t} \Vert h(s) \Vert_{L^\infty_{x, v}}. 
\end{aligned}
\end{equation*}
For the case $s \in [0, t-\lambda]$,  we split it into two cases $|v|\ge N$ and $|v|\le N$ for some large constant $N$ to be fixed later. For the case $|v| \ge N$, since $\beta>3$ we have
\begin{equation*}
\begin{aligned}
&\int_{0}^{t-\lambda} \int_{|v| \ge N} e^{-\nu(v)(t-s)} \langle v \rangle^{-\beta} \int_{\R^3}| l_{k +\beta}  (v, v')  h(s, x-v(t-s), v' )   | dv' dv ds 
\\
\le&\sup_{0\le s \le t} \Vert h(s) \Vert_{L^\infty_{x, v}} \int_{ 0 }^t  \int_{|v| \ge N} e^{-\nu(v )(t-s)} \langle v \rangle^{-\beta}  \int_{\R^3}| l_{k + \beta} (v, v')    | dv' dv ds 
\\
\le&C_{k, \beta} \sup_{0\le s \le t} \Vert h(s) \Vert_{L^\infty_{x, v}} \int_{0}^t  e^{-\nu(v )(t-s)}  \langle v \rangle^\gamma ds  \int_{|v| \ge N}  \langle v \rangle^{-\beta}  dv ds \le C_{k, \beta}  \frac 1 {N^{\beta-3}}\sup_{0\le s \le t} \Vert h(s) \Vert_{L^\infty_{x, v}} .
\end{aligned}
\end{equation*}
For the case $|v| \le N$, using decomposition \eqref{l k N beta} we split it into two terms respectively. For the first term since $\beta >3$ we have
\begin{equation*}
\begin{aligned}
 &\int_{0}^{t-\lambda} \int_{|v| \le N} e^{-\nu(v)(t-s)}   \langle v \rangle^{-\beta}\int_{\R^3}| l_{k + \beta}  (v, v') - l_{k, N, \beta}(v, v') || h(s, x-v(t-s), v' )   | dv' dv ds 
\\
\le&C_{k, \beta} \frac{1} N \sup_{0\le s \le t} \Vert h(s) \Vert_{L^\infty_{x, v}} \int_{0}^t  e^{-\nu(v )(t-s)}  \langle v \rangle^\gamma ds  \int_{\R^3}  \langle v \rangle^{-\beta}  dv ds \le C_{k, \beta}  \frac 1 {N}\sup_{0\le s \le t} \Vert h(s) \Vert_{L^\infty_{x, v}}.
\end{aligned}
\end{equation*}
For the last term since $|v| \le N$, by \eqref{l k N beta bound} and \eqref{entropy for h} we have
\begin{equation*}
\begin{aligned}
 &\int_{0}^{t-\lambda} \int_{|v| \le N} e^{-\nu(v)(t-s)}   \langle v \rangle^{-\beta}\int_{\R^3}|  l_{k, N, \beta}(v, v') || h(s, x-v(t-s), v' )   | dv' dv ds 
\\
\le&C_{k, N, \beta} \int_{0}^{t-\lambda}  e^{c(t-s)} \int_{|v| \le N}   \int_{|v'| \le C_{k, N, \beta}'} | h(s, x-v(t-s), v' )   | dv' dv ds \le C_{k, N, \beta} \lambda^{-\frac 3 2} \sqrt{ H (F_0)} +C_{k, N, \beta} \lambda^{-3}   H (F_0). 
\end{aligned}
\end{equation*}
Then we come to the $H_2$ term, we have
\begin{equation*}
\begin{aligned}
|H_2| \le& \int_{0}^t \int_{\R^3} e^{-\nu(v) (t-s)} |\Gamma_k^- ( f, f) (s, x-v (t-s),  v ) | dv ds 
\\
&+  \int_{0}^t \int_{\R^3} e^{-\nu(v) (t-s)} |\Gamma_k^+ ( f, f) (s, x-v (t-s),  v ) | dv ds := H_{21}+ H_{22}.
\end{aligned}
\end{equation*}
For the  $H_{21} $ term, we split it into four terms for some constant $\lambda, N>0$ to be fixed later
\begin{equation*}
\begin{aligned}
H_{21} \le& C\int_0^t \int_{\R^3} e^{-\nu(v)(t - s) } \int_{\R^3} |v-v_*|^\gamma \langle v_* \rangle^{-k}  |f (s, x-v(t-s), v_* ) || f (s, x-v(t-s), v ) |  dv_*  dv ds
\\
\le &C \sup_{0\le s \le t} \Vert h(s) \Vert_{L^\infty_{x, v}}\int_0^t \int_{\R^3} e^{-\nu(v)(t - s) } \int_{\R^3} |v-v_*|^\gamma \langle v_* \rangle^{-k-\beta}  \langle v \rangle^{-\beta} | h (s, x-v(t-s), v_* )| dv_*  dv ds
\\
=&C \sup_{0\le s \le t} \Vert h(s) \Vert_{L^\infty_{x, v}} \left( \int_{t-\lambda}^{t} \int_{\R^3} \int_{\R^3}  + \int_{0}^{t-\lambda} \int_{|v| \ge N} \int_{\R^3} + \int_0^{t-\lambda} \int_{\R^3} \int_{|v_*| \ge N}  +\int_0^{t-\lambda} \int_{|v| \le N} \int_{|v_*| \le N}  \right)\{ \cdot \cdot \cdot\} dv_* dv ds
\\
:=& I_1 +I_2+I_3+I_4. 
\end{aligned}
\end{equation*}
For the term $I_1$, since $\beta, k >\max \{ 3, 3+\gamma \}$
\[
I_1 \le  C \sup_{0 \le s \le t} \Vert h(s) \Vert_{L^\infty_{x, v}}^2 \int_{t-\lambda}^t \int_{\R^3}  \int_{\R^3} |v-v_*|^\gamma \langle v_* \rangle^{-k-\beta}  \langle v \rangle^{-\beta}dv_*  dv ds \le C_{k, \beta} \lambda  \sup_{0 \le s \le t} \Vert h(s) \Vert_{L^\infty_{x, v}}^2. 
\]
For the term $I_2$, since  $\beta, k >\max \{ 3, 3+\gamma \}$
\[
I_2 \le  C \sup_{0 \le s \le t} \Vert h(s) \Vert_{L^\infty_{x, v}}^2 \int_{0}^t e^{-\nu(v)(t - s) }  \langle v \rangle^{\gamma}  d s \int_{|v| \ge N}    \langle v \rangle^{-\beta}  dv \le C_{k, \beta } \frac 1 {N^{\beta-3}}  \sup_{0 \le s \le t} \Vert h(s) \Vert_{L^\infty_{x, v}}^2. 
\]
For the term $I_3$, since $\beta, k >\max \{ 3, 3+\gamma \}$
\begin{equation*}
\begin{aligned}
I_3 \le& C \sup_{0 \le s \le t} \Vert h(s) \Vert_{L^\infty_{x, v}}^2 \int_0^t \int_{\R^3} e^{-\nu(v)(t - s) } \int_{|v_*| \ge N} |v-v_*|^\gamma \langle v_* \rangle^{-k-\beta}  \langle v \rangle^{-\beta} dv_*  dv ds 
\\
\le& C  \frac 1 {N^{\beta}}  \sup_{0 \le s \le t} \Vert h(s) \Vert_{L^\infty_{x, v}}^2 \int_0^t \int_{\R^3} e^{-\nu(v)(t - s) } \int_{|v_*| \ge N} |v-v_*|^\gamma \langle v_* \rangle^{-k}  \langle v \rangle^{-\beta} dv_*  dv ds 
\\
\le&   C_k \frac 1 {N^{\beta}}    \sup_{0 \le s \le t} \Vert h(s) \Vert_{L^\infty_{x, v}}^2 \int_{0}^t e^{-\nu(v)(t - s) }  \langle v \rangle^{\gamma}  d s \int_{\R^3 }    \langle v \rangle^{-\beta}  dv 
\le C_{k, \beta } \frac 1 {N^{\beta}}   \sup_{0 \le s \le t} \Vert h(s) \Vert_{L^\infty_{x, v}}^2.
\end{aligned}
\end{equation*}
For the $I_4$ term, since $\beta, k >\max \{ 3, 3+\gamma \}$, similar as \eqref{entropy for h} we have
\begin{equation*}
\begin{aligned}
I_4 \le & C \sup_{0 \le s \le t} \Vert h(s) \Vert_{L^\infty_{x, v}} \int_0^{t-\lambda}  e^{- c (t - s) } \int_{|v| \le N}\int_{|v_*| \le N} |v-v_*|^\gamma \langle v_* \rangle^{-k-\beta}  \langle v \rangle^{-\beta} |h (s, x-v(t-s), v_* ) |dv_*  dv ds
\\
\le &\sup_{0 \le s \le t} \Vert h(s) \Vert_{L^\infty_{x, v}} \int_0^{t-\lambda}  e^{- c (t - s) } \left(\int_{|v| \le N}\int_{|v_*| \le N} |v-v_*|^{\gamma p}  \langle v_* \rangle^{- k p -\beta p}  \langle v \rangle^{-\beta p}  dv_*  dv  \right)^{\frac 1 p} 
\\
&\left(\int_{|v| \le N}\int_{|v_*| \le N}    |f (s, x-v(t-s), v_* ) |^{\frac p {p-1} }  dv_*  dv  \right)^{1-\frac 1 p}
\\
 \le &C_{k, N, \beta}  \sup_{0 \le s \le t} \Vert h(s) \Vert_{L^\infty_{x, v}}^{1+\frac 1 p}\left(\int_{|v| \le N}\int_{|v_*| \le N}    |f (s, x-v(t-s), v_* ) |dv_*  dv  \right)^{1-\frac 1 p}
\\
 \le &C_{k, N, \beta} \lambda^{- 3}  (\sqrt{ H (F_0) }  + H (F_0) )^{1-\frac 1 p} \sup_{0 \le s \le t} \Vert h(s) \Vert_{L^\infty_{x, v}}^{1+\frac 1 p}, 
\end{aligned}
\end{equation*}
where $p$ is defined in \eqref{requirement p}. For the  $H_{22}$ term, we split it into four terms for some constant $\lambda, N>0$ to be fixed later
\begin{equation*}
\begin{aligned}
H_{22} \le& C\int_0^t \int_{\R^3} e^{-\nu(v)(t - s) } \int_{\R^3} \int_{\mathbb{S}^2}  |v-v_*|^\gamma \frac { \langle v \rangle^{k}} { \langle v' \rangle^{k} \langle v_*' \rangle^{k}   }  | f (s, x-v(t-s), v_*' ) | | f (s, x-v(t-s), v' ) | dv_* d\sigma dv ds
\\
 \le& C\int_0^t \int_{\R^3} e^{-\nu(v)(t - s) } \int_{\R^3} \int_{\mathbb{S}^2}  |v-v_*|^\gamma \frac { \langle v \rangle^{k}} { \langle v' \rangle^{k +\beta } \langle v_*' \rangle^{k+\beta}   }  | h (s, x-v(t-s), v_*' )  | | h (s, x-v(t-s), v' ) | dv_* d\sigma dv ds
\\
=&C \left( \int_{t-\lambda}^{t} \int_{\R^3} \int_{\R^3}  + \int_{0}^{t-\lambda} \int_{|v| \ge N} \int_{\R^3} + \int_0^{t-\lambda} \int_{\R^3} \int_{|v_*| \ge N}  +\int_0^{t-\lambda} \int_{|v| \le N} \int_{|v_*| \le N}  \right)\{ \cdot \cdot \cdot \} dv_* dv ds:= I_1 +I_2+I_3+I_4.
\end{aligned}
\end{equation*}
For the $I_1$ term, since $\beta, k >\max \{ 3, 3+\gamma \}$, by Lemma \ref{L211}
\[
I_1 \le  C \sup_{0 \le s \le t} \Vert h(s) \Vert_{L^\infty_{x, v}}^2 \int_{t-\lambda}^t \int_{\R^3}  \langle v \rangle^{-\beta + \gamma}   dv ds \le C_{k, \beta} \lambda  \sup_{0 \le s \le t} \Vert h(s) \Vert_{L^\infty_{x, v}}^2.
\]
For the $I_2$ term, since  $\beta, k >\max \{ 3, 3+\gamma \}$ still by Lemma \ref{L211} 
\[
I_2 \le  C \sup_{0 \le s \le t} \Vert h(s) \Vert_{L^\infty_{x, v}}^2 \int_{0}^t e^{-\nu(v)(t - s) }  \langle v \rangle^{\gamma}  d s \int_{|v| \ge N}    \langle v \rangle^{-\beta}  dv \le C_{k, \beta } \frac 1 {N^{\beta-3}}  \sup_{0 \le s \le t} \Vert h(s) \Vert_{L^\infty_{x, v}}^2.
\]
For the $I_3$ term, since $\beta, k >\max \{ 3, 3+\gamma \}$, since $\langle v_* \rangle \le \langle v' \rangle\langle v_*' \rangle$, still by Lemma \ref{L211} apply for $k+\beta- \frac {\beta-3} 2$ we have
\begin{equation*}
\begin{aligned}
I_3 \le & C \sup_{0 \le s \le t} \Vert h(s) \Vert_{L^\infty_{x, v}}^2\int_0^t \int_{\R^3} e^{-\nu(v)(t - s) } \int_{|v_*| \ge N} |v-v_*|^\gamma    \frac { \langle v \rangle^{k}} { \langle v' \rangle^{k +\beta -\frac {\beta-3} 2 } \langle v_*' \rangle^{k+\beta -\frac {\beta-3} 2}  \langle v_* \rangle^{\frac {\beta-3} 2}  } dv_*  dv ds 
\\
\le & C_{k, \beta } \frac 1 {N^{\frac {\beta-3} 2}}  \sup_{0 \le s \le t} \Vert h(s) \Vert_{L^\infty_{x, v}}^2 \int_{0}^t e^{-\nu(v)(t - s) }  \langle v \rangle^{\gamma}  d s \int_{\R^3}    \langle v \rangle^{-\beta + \frac {\beta-3} 2}  dv \le  C_{k, \beta } \frac 1 {N^{\frac {\beta-3} 2 }} \sup_{0 \le s \le t} \Vert h(s) \Vert_{L^\infty_{x, v}}^2.
\end{aligned}
\end{equation*}
For the $I_4$ term, if $\beta, k >\max \{ 3, 3+\gamma \}$,  similar as \eqref{entropy for h} we have
\begin{equation*}
\begin{aligned}
I_4 \le & C \sup_{0 \le s \le t} \Vert h(s) \Vert_{L^\infty_{x, v}} \int_0^{t-\lambda}  e^{-c(t - s) } \int_{|v| \le N}\int_{|v_*| \le N} \int_{\mathbb{S}^2 } |v-v_*|^\gamma\frac { \langle v \rangle^{  k}} { \langle v' \rangle^{  k +\beta} \langle v_*' \rangle^{ k + \beta  }   } | h (s, x-v(t-s), v' )| dv_*  dv ds
\\
\le & C_{k, N, \beta} \sup_{0 \le s \le t} \Vert h(s) \Vert_{L^\infty_{x, v}} \int_0^{t-\lambda}  e^{-c(t - s) } \left(\int_{|v| \le N}\int_{|v_*| \le N}\int_{\mathbb{S}^2 }  |v-v_*|^{\gamma p}  \frac { \langle v \rangle^{ p k}} { \langle v' \rangle^{ p k} \langle v_*' \rangle^{p k}   }  d \sigma dv_*  dv  \right)^{\frac 1 p} 
\\
&\left(\int_{|v| \le N}\int_{|v_*| \le N}  \int_{\mathbb{S}^2 }  |f (s, x-v(t-s), v' ) |^{\frac p {p-1} } d \sigma dv_*  dv  \right)^{1-\frac 1 p}
\\
\le & C_{k, N, \beta} \sup_{0 \le s \le t} \Vert h(s) \Vert_{L^\infty_{x, v}}  \int_0^{t-\lambda}  e^{-c(t - s) } \left(\int_{|v' | \le 2 N}\int_{|v_*'| \le 2 N}  \int_{\mathbb{S}^2 }  |f (s, x-v(t-s), v' ) |^{\frac p {p-1} } d \sigma dv_*'  dv'  \right)^{1-\frac 1 p}
\\
 \le &C_{k, N, \beta} \lambda^{- 3}  (\sqrt{H (F_0)}  + H (F_0) )^{1-\frac 1 p} \sup_{0 \le s \le t} \Vert h(s) \Vert_{L^\infty_{x, v } }^{1+\frac 1 p},
\end{aligned}
\end{equation*}
since 
\[
\{ |v| \le N, |v_*| \le N,  \sigma \in \mathbb{S}^2  \} \quad \subset \quad  \{ |v'| \le 2N, |v_*'| \le 2N, \sigma \in \mathbb{S}^2  \},
\]
and $p$ is defined in \eqref{requirement p}. The theorem is thus proved by gathering all the terms. 
\end{proof}

Then we come to the proof for the global existence for the Boltzmann equation with large amplitude initial data. 

\begin{proof} ({\bf Proof of Theorem \ref{T13}} ) Fix $\beta, k >\max \{3, 3+\gamma\}$ satisfies the assumption in Lemma \ref{L66} and Lemma \ref{L67}. By the assumption in Theorem \ref{T13} we have $\Vert h_0 \Vert_{L^\infty_{x, v}} \le M$. We first assume that
\begin{equation}
\label{priori assumption}
\Vert h(t) \Vert_{L^\infty_{x, v}} \le 2A_0 := 2C_{k, \beta} (2M^2   +\sqrt{ H (F_0) } + H (F_0)  ),
\end{equation}
where $C_{k, \beta}$ is defined in Lemma \ref{L66}. By Lemma \ref{L66} and the priori assumption \eqref{priori assumption} we have
\begin{equation}
\label{priori assumption 2}
 \Vert h(t) \Vert_{L^\infty_{x, v}} \le A_0 + C_{k, \beta} (2A_0)^{\frac {3p+1} {2p}} \cdot \sup_{t_1 \le s \le t , y \in \T^3 } \left( \int_{\R^3} |f(s, y, \eta) | d \eta \right) ^{\frac {p-1} {2p}}. 
\end{equation}
Since $x \in \T^3$, for any $t \ge t_1$ we have
\begin{equation*}
\begin{aligned}
\int_{\R^3} e^{-\nu(t)} |f_0(x-vt, v) | dv  \le& \left( \int_{|v| \ge \lambda} + \int_{|v| \le \lambda}    \right)  |f_0(x-vt, v) | dv
\\
\le& C \Vert w_\beta f_0\Vert_{L^\infty_{x, v}}^{\frac 3 \beta} \Vert f_0 \Vert_{L^1_x L^\infty_v}^{1-\frac 3 \beta } + \frac C {t_1^3}\Vert f_0 \Vert_{L^1_x L^\infty_v}
\\
\le& C M^{\frac 3 \beta} \Vert f_0 \Vert_{L^1_x L^\infty_v}^{1-\frac 3 \beta } + C  M^3 \Vert f_0 \Vert_{L^1_x L^\infty_v}, 
\end{aligned}
\end{equation*}
where we have chosen  $\lambda = \Vert w_\beta f_0\Vert_{L^\infty_{x, v}}^{\frac 1 \beta} \Vert f_0 \Vert_{L^1_x L^\infty_v}^{-\frac 1\beta } $ and $t_1$ is defined in Lemma \ref{L64}. By Lemma \ref{L67} and the priori assumption \eqref{priori assumption} we have
\begin{equation*}
\begin{aligned}
 \sup_{t_1 \le s \le t , y \in \T^3  }  \int_{\R^3} |f(s, y, \eta) | d \eta \le & CM^3 \Vert f_0 \Vert_{L^1_xL^\infty_v} + C M^{\frac 3 \beta} \Vert f_0 \Vert_{L^1_x L^\infty_v}^{1-\frac 3 \beta }  +C_{k, N, \beta} \lambda^{-\frac 3 2 } \sqrt{H (F_0)} +C_{k, N, \beta} \lambda^{- 3 } H (F_0) 
\\
&+ C_{k, \beta}(\lambda + \frac {1} {N^{\frac {\beta-3} 2}} ) (2A_0)^2 + C_{k, N, \beta} \lambda^{- 3}  (\sqrt{H (F_0)}  + H (F_0) )^{1-\frac 1 p}   (2A_0)^{1+\frac 1 p}. 
\end{aligned}
\end{equation*}
First choose $N$ large then choose $\lambda$ small, finally choose $H(F_0)$  and $\Vert f_0\Vert_{L^1_x L^\infty_v}$ small we deduce
\begin{equation}
\label{small term large amplitude}
4C_{k, \beta} A_0^{\frac {p+1} {2p} }\sup_{t_1 \le s \le t }  \left( \int_{\R^3} |f(s, y, \eta) | d \eta \right) ^{\frac {p-1} {2p}} \le \frac 1 2,
\end{equation}
together with \eqref{priori assumption 2} implies that 
\[
\Vert h(t ) \Vert_{L^\infty_{x, v}} \le \frac 7 4 A_0, \quad \forall t \ge 0,
\]
hence we have closed the priori assumption \eqref{priori assumption},  the proof for  the global existence is thus finished.
\end{proof}

\section{Convergence rate for the Boltzmann equation with large amplitude initial data}\label{section7}
 
In this section, we consider the time-decay estimates for the global solution we obtained in Section \ref{section6}. In the whole section we are under the assumption in Theorem \ref{T13} such that all the results  in Section \ref{section6} remain true. In this section we will denote $f$ the solution to \eqref{Boltzmann equation f}  and denote $h$ the solution to \eqref{Boltzmann equation h}. 
\subsection{Convergence rate for Hard Potentials }
In this subsection we consider the decay estimate for hard potential. 
We first  recall a lemma on the convolution of semigroups. 
\begin{lem}\label{L71}
For $\lambda_2> \lambda_1 > 0, t >0$. we have 
\[
\int_0^t e^{-\lambda_1 (t-s) }  e^{ - \lambda_2s} ds = \int_0^t e^{-\lambda_1 s }  e^{ - \lambda_2 (t-s ) } ds  = e^{-\lambda_1 t } \int_{0}^t e^{-(\lambda_2-\lambda_1) s} ds \le  \frac {1} {\lambda_2 - \lambda_1}  e^{ - \lambda_1 t  }.
\]
\end{lem}

For the case $\gamma \ge 0$, we first prove that the linearized equation converges. We consider  the following linearized equation
\begin{equation}
\label{linearized equation large amplitude}
\xi_t + v \cdot \nabla_x \xi + \nu(v) \xi + K_k \xi = 0, \quad \xi(0, x, v ) = \xi_0(x, v).
\end{equation}
For the linearized equation \eqref{linearized equation large amplitude}, the corresponding mild solution is 
\begin{equation}
\label{mild solution linear}
\xi (t, x, v) = e^{-\nu(v) t} \xi _0(x-vt, v) + \int_0^t e^{-\nu(v) (t-s) } (K_{k } \xi ) (s, x-v(t-s), v) ds.
\end{equation}

\begin{lem}\label{L72}
There exists a $k_0  \ge 6 $ large such that for any $k \ge k_0 , \gamma \in [0, 1]$, suppose $\xi(t)$ is the solution to the linearized equation \eqref{linearized equation large amplitude}, we have
\[
\Vert  \xi(t) \Vert_{L^\infty_{x, v}} \le C_k e^{-\frac {\lambda_2} 2 t} \Vert \xi_0 \Vert_{L^\infty_{x, v}},  \quad \forall t \ge 0,
\]
for some constants $C_k, \lambda_2 >0$. 
\end{lem} 
\begin{proof} The proof is similar to Lemma \ref{L66}. First  \eqref{mild solution linear} implies
\[
|\xi (t, x, v)| \le  e^{-\nu(v) t} \Vert \xi _0\Vert_{L^\infty_{x, v}} + \int_0^t e^{-\nu(v) (t-s) } |(K_{k } \xi ) (s, x-v(t-s), v) |ds  :=e^{-\nu(v) t} \Vert \xi_0\Vert_{L^\infty_{x, v}} +J_2.
\]
For the $J_2$ term,  denote $\tilde{x} = x - v(t-s)$, we have
\[
J_2 \le \int_0^t e^{-\nu(v) (t-s) }\int_{\R^3}| l_k (v, v') \xi (s, \tilde{x}, v')| dv' ds.
\] 
by  \eqref{mild solution linear}  again we have 
\begin{equation*}
\begin{aligned}
J_2 \le& \int_0^t e^{-\nu(v) (t-s) }\int_{\R^3}| l_k (v, v')|  e^{-\nu(v') s} | \xi_0 (\tilde{x} -v' s , v')| dv' ds  
\\
&+  \int_0^t e^{-\nu(v) (t-s) }  \int_{\R^3}  \int_{\R^3} | l_k (v, v')  l_k (v', v'')  |  \int_0^s  e^{-\nu(v') (s-\tau )}   | \xi (\tau, \tilde{x} -v' (s-\tau), v'') | d v''d \tau dv' ds :=J_{21} +J_{22}. 
\end{aligned}
\end{equation*}
Denote $\lambda_2 := \min \{\lambda_1, \nu(v)| v\in \R^d \}$, where $\lambda_1$ is the exponential convergence rate for the linearized semigroup in $L^2$ proved in Lemma \ref{L312}. For the $J_{21}$ term we have
\begin{equation}
\label{new J21}
J_{21} \le C_k e^{-\frac {\lambda_2} 2  t} \Vert  \xi_0 \Vert_{L^\infty_{x, v}} \int_0^t e^{-\frac {\nu(v)} 2 (t-s) } \langle v \rangle^\gamma ds \le C_k e^{-\frac {\lambda_2} 2  t}   \Vert \xi_0 \Vert_{L^\infty_{x, v}}.
\end{equation}
For the $J_{22}$ term, if $|v| \ge N$, by \eqref{v ge N} we have
\begin{equation*}
\begin{aligned}
J_{22} \le&  \sup_{ 0 \le \tau \le t} \{ e^{\frac {\lambda_2} 2  \tau} \Vert \xi (\tau )  \Vert_{L^\infty_{x, v}} \}  \int_0^t e^{-\nu(v) (t-s) }  \int_{\R^3}  \int_{\R^3} | l_k (v, v')  l_k (v', v'')  |  \int_0^s  e^{-\nu(v') (s-\tau )}  e^{-\frac {\lambda_2} 2 \tau} d v''d \tau dv' ds 
\\
\le &  e^{-\frac {\lambda_2} 2 t}\sup_{ 0 \le \tau \le t} \{ e^{\frac {\lambda_2} 2  \tau} \Vert \xi (\tau )  \Vert_{L^\infty_{x, v}} \} \int_0^t e^{-\frac {\nu(v)} 2 (t-s) }  \int_{\R^3}  \int_{\R^3} | l_k   (v, v')  l_k (v', v'')  |  \int_0^s  e^{-\frac {\nu(v')} 2 (s-\tau )}d v''d \tau dv' ds 
\\
\le &\left( \frac {c^2} {k^{\frac {\gamma+3} 2}}   + \frac {C_{k}}{N^2} \right) e^{-\frac {\lambda_2} 2 t}\sup_{ 0 \le \tau \le t} \{ e^{\frac {\lambda_2} 2  \tau} \Vert \xi (\tau )  \Vert_{L^\infty_{x, v}} \}. 
\end{aligned}
\end{equation*}
For the case $|v| \le N$, similarly as decomposition \eqref{l k beta decomposition}  we have, for any $N>0, k \ge 6$   we can find a bounded compact support function $l_{k, N}$ such that
\begin{equation}
\label{l k N}
l_{k, N} (v, v')  : = l_{k} (v, v') 1_{\frac {\langle v \rangle} {C_{k, N}} \le  |v-v'|  \le C_{k, N} \langle v \rangle  }, \quad \int_{\R^3} | l_{k} (v,  v') -l_{k, N} (v, v')|  dv' \le \frac {C_{k}} {N } \langle v \rangle^\gamma,\quad \forall v \in \R^3,
\end{equation}
for some large constant $C_{k, N} >0 $. By
\begin{equation}
\label{l k decomposition}
l_{k} (v, v') l_{k }  (v', v'') = ( l_{k} (v, v') - l_{k, N}(v, v') ) l_{k} (v', v'') +l_{k, N} (v, v')  ( l_{k} (v', v'') - l_{k, N}(v', v'') )  +  l_{k, N}(v, v')  l_{k, N}(v', v''),
\end{equation}
we split $J$ into three terms respectively. For the first term we have
\begin{equation*}
\begin{aligned}
&\int_0^t e^{-\nu(v) (t-s) }  \int_{\R^3}  \int_{\R^3} |( l_k (v, v') - l_{k, N}(v, v') ) l_k (v', v'')  |  \int_0^s  e^{-\nu(v') (s-\tau )}   | \xi (\tau, \tilde{x} -v' (s-\tau), v'') | d v''d \tau dv' ds 
\\
\le &\frac {C_{k}} {N} \sup_{ 0 \le \tau \le t} \{ e^{\frac {\lambda_2} 2  \tau} \Vert \xi (\tau )  \Vert_{L^\infty_{x, v} } \}    \int_0^t e^{-  \nu(v)  (t-s) }  \langle v \rangle^\gamma ds \int_0^{s} e^{-    \nu(v') (s-\tau) }  \langle v' \rangle^\gamma   e^{-\frac {\lambda_2} 2 \tau} d\tau  
\\
\le & \frac {C_{k}} {N} e^{ - \frac {\lambda_2} 2 t}  \sup_{ 0 \le \tau \le t} \{ e^{\frac {\lambda_2} 2  \tau} \Vert \xi (\tau )  \Vert_{L^\infty_{x, v} } \},
\end{aligned}
\end{equation*}
and the second term can be estimated similarly. For the last term,  since $l_{k, N}(v', v'') l_{k, N}(v', v'')$ is supported in 
$|v| \le N,  |v'|\le C_{k, N}', |v''|\le C_{k, N}'$ for some constant $C_{k, N}'>0$. We again split it into two cases, $\tau \in [s-\lambda, s]$ and $\tau \in [0, s-\lambda ]$,  where $\lambda>0$ is a small constant to be fixed later. For the case $\tau \in [s-\lambda, s]$, since $|v'| \le C_{k, N}'$,   we have 
\begin{equation*}
\begin{aligned}
&\int_0^t e^{-\nu(v) (t-s) }  \int_{\R^3}  \int_{\R^3} | l_{k, N} (v, v')  l_{k, N}(v', v'')  |  \int_{s-\lambda}^s  e^{-\nu(v') (s-\tau )}   | \xi (\tau, \tilde{x} -v' (s-\tau), v'') | d v''d \tau dv' ds 
\\
\le &C_{k, N}  \sup_{ 0 \le \tau \le t} \{ e^{\frac {\lambda_2} 2  \tau} \Vert \xi (\tau )  \Vert_{L^\infty_{x, v} } \}  \int_0^t e^{-\nu(v) (t-s) }  \langle v \rangle^\gamma ds \int_{s-\lambda}^{s} e^{-\nu(v') (s-\tau) }  \langle v' \rangle^\gamma e^{-\frac {\lambda_2} 2 \tau } d\tau 
\\
\le &C_{k, N}  e^{-\frac {\lambda_2} 2 t}   \sup_{ 0 \le \tau \le t} \{ e^{\frac {\lambda_2} 2  \tau} \Vert \xi (\tau )  \Vert_{L^\infty_{x, v} } \}  \int_0^t e^{-\frac {\nu(v)} 2 (t-s) }  \langle v \rangle^\gamma ds \int_{s-\lambda}^{s} e^{-\frac {\nu(v')} 2 (s-\tau) }  \langle v' \rangle^\gamma  d\tau 
\\
\le &C_{k, N}    e^{-\frac {\lambda_2} 2 t}   \sup_{ 0 \le \tau \le t} \{ e^{\frac {\lambda_2} 2  \tau} \Vert \xi (\tau )  \Vert_{L^\infty_{x, v} } \}    (1- e^{ - \frac {\nu(v')} 2 \lambda})  \le C_{k, N} \lambda  e^{-\frac {\lambda_2} 2 t}   \sup_{ 0 \le \tau \le t} \{ e^{\frac {\lambda_2} 2  \tau} \Vert \xi (\tau )  \Vert_{L^\infty_{x, v} } \}. 
\end{aligned}
\end{equation*}
For the case $\tau \in [0, s-\lambda ]$ since 
\begin{equation}
\label{l k N bound}
l_{k, N}(v, v') \le C_{k, N},\quad l_{k, N}(v', v'') \le C_{k, N},\quad \frac 1 { C_{k, N} } \le \nu(v) \le C_{k, N},\quad  \frac 1 { C_{k, N} }  \le  \nu(v' ) \le C_{k, N},
\end{equation}
denote $\xi_1(t, x, v) = \langle v \rangle^{-\frac 7 4} \xi(t, x, v) $, since $k- \frac 7 4 > 4$, by Lemma \ref{L312} apply for $\xi_1$ we have
\begin{equation*}
\begin{aligned}
&\int_{|v'| \le C_{k, N}', |v''| \le C_{k, N}' } | \xi (\tau, \tilde{x} -v' (s-\tau), v'') | dv' dv''
\\
\le &C_{k, N} \int_{|v'| \le C_{k, N}', |v''| \le C_{k, N}' } | \xi_1(\tau, \tilde{x} -v' (s-\tau), v'') | dv' dv''
\\
\le & C_{k, N} \frac 1 {(s-\tau)^\frac 3 2}  \left( \int_{\T^3} \int_{|v''| \le C_{k, N}' }  |\xi_1(\tau, y, v'') |^2   dv'' dy \right)^{\frac 1 2}
\\
\le & C_{k, N} \frac 1 {(s-\tau)^\frac 3 2} \Vert \xi_1(\tau) \Vert_{L^2_{x, v} } \le C_{k, N} \frac 1 {(s-\tau)^\frac 3 2} e^{-\lambda_1 \tau} \Vert \xi_1(0)  \Vert_{L^2_{x, v}}   \le C_{k, N} \frac 1 {(s-\tau)^\frac 3 2} e^{-\lambda_1 \tau} \Vert \xi_0 \Vert_{L^\infty_{x, v}},
\end{aligned}
\end{equation*}
where we have made a change of variable $ y = \tilde{x} -v' (s-\tau) $.  Since $s-\tau \ge \lambda$, we have
\begin{equation*}
\begin{aligned}
&\int_0^t e^{-\nu(v) (t-s) }  \int_{\R^3}  \int_{\R^3} | l_{k, N} (v, v')  l_{k, N}(v', v'')  |  \int_{0}^{s-\lambda}  e^{-\nu(v') (s-\tau )}   | \xi (\tau, \tilde{x} -v' (s-\tau, v'')) | d v''d \tau dv' ds 
\\
\le &C_{k, N} \int_0^t e^{-\lambda_2 (t-s) } \int_{0}^{s-\lambda}  e^{-  \lambda_2 (s-\tau )}  \int_{|v'| \le C_{k, N}', |v''| \le C_{k, N}' }   | \xi  (\tau, \tilde{x} -v' (s-\tau), v'') | d v''d \tau dv' ds 
\\
\le &C_{k, N} \lambda^{-\frac 3 2 }  \Vert \xi_0 \Vert_{L^\infty_{x, v}} \int_0^t e^{-\lambda_2 (t-s) } \int_{0}^{s-\lambda} e^{-  \lambda_2 (s-\tau )}   e^{-\lambda_1 \tau} d \tau ds  \le C_{k, N} \lambda^{-\frac 3 2 } e^{-\frac {\lambda_2} 2 t} \Vert \xi_0 \Vert_{L^\infty_{x, v}}.
\end{aligned}
\end{equation*}
Gathering all the terms we have
\begin{equation*}
\begin{aligned}
 \Vert \xi (t) \Vert_{L^\infty_{x, v}}  \le & C_{k, N, \lambda} e^{-\frac {\lambda_2} 2  t} \Vert \xi_0 \Vert_{L^\infty_{x, v}}  +  \left ( \frac {c^2} {k^{\frac {\gamma+3} 2}}    + \frac {C_{k}} {N} + C_{k, N} \lambda   \right)    e^{-\frac {\lambda_2} 2 t}   \sup_{ 0 \le \tau \le t} \{ e^{\frac {\lambda_2} 2  \tau} \Vert \xi (\tau  )  \Vert_{L^\infty_{x, v}} \}.
\end{aligned}
\end{equation*}
First choose $k$ large, then let $N$ be sufficiently large and finally let $\lambda$ be sufficiently small such that 
\[
\frac {c^2} {k^{\frac {\gamma+3} 2}}    + \frac {C_{k}} {N} + C_{k, N} \lambda \le \frac 1 2.
\]
Multiply both side by $e^{\frac {\lambda_2} 2  t}$ and taking supremum  we have
\[
\sup_{ 0 \le \tau \le t} \{ e^{\frac {\lambda_2} 2  \tau} \Vert \xi (\tau )  \Vert_{L^\infty_{x, v}} \} \le C_k \Vert \xi_0 \Vert_{L^\infty_{x, v}}  + \frac 1 2 \sup_{ 0 \le \tau \le t} \{ e^{\frac {\lambda_2} 2  \tau} \Vert \xi (\tau )  \Vert_{L^\infty_{x, v}} \} ,
\]
which implies
\[
\Vert \xi (t) \Vert_{L^\infty_{x, v}}   \le C_k e^{-\frac {\lambda_2} 2  t} \Vert \xi_0 \Vert_{L^\infty_{x, v}},
\]
so the convergence for the linear semigroup is thus proved. 
\end{proof}

Denote $g(t, x, v) = \langle v \rangle^{-2} h(t, x, v) = \langle v \rangle^{\beta-2} f(t, x, v)$, it is easily seen that $g$ satisfies 
\begin{equation}
\label{equation large amplitude ge 0 g}
\partial_t g + v \cdot \nabla_x g + \nu(v) g =\Gamma_{k + \beta-2} (g, g) + K_{k + \beta-2} g. 
\end{equation}
We first prove the rate of convergence for $g$. 

\begin{lem}\label{L73} There exists a $k_0 \ge  8 $ large such that for any $\gamma \in [0, 1], k \ge k_0 , \beta >\max \{ 3, 3+\gamma\}$,  suppose $g$ is the solution to the nonlinear equation \eqref{equation large amplitude ge 0 g}, we have
\[
\Vert g(t) \Vert_{L^\infty_{x, v}} \le C_{k, \beta} e^{-\frac {\lambda_2} 4  t} \Vert g \Vert_{L^\infty_{x, v}}, \quad \forall t \ge 0,
\]
for some constant $C_{k, \beta} >0 $, where $\lambda_2$ is defined in Lemma \ref{L72}.  
\end{lem}
\begin{proof}
Denote the solution to the linearized equation
\[
\partial_t g + v \cdot \nabla_x g + \nu(v) g = K_{k + \beta-2} g, \quad g(0, x, v) =g_0(x, v),
\]
by $g(t) =V(t) g_0$. By Lemma \ref{L72} we have
\[
\Vert V(t) g \Vert_{L^\infty_{x, v}} \le C_{k, \beta} e^{-\frac {\lambda_2} 2  t} \Vert g \Vert_{L^\infty_{x, v}}. 
\]
By Duhamel's principle we have
\[
g(t) = V(t) g_0 +\int_0^t V(t-s ) \{\Gamma_{k+\beta-2} (g, g) (s) \} ds = V(t) g_0 +\int_0^t V(t-s ) \{\langle v \rangle^{\beta-2} \Gamma_k (f, f) (s) \} ds.
\]
We easily compute
\[
\Vert g(t)  \Vert_{L^\infty_{x, v}} \le C_{k, \beta} e^{-\frac {\lambda_2} 2  t}  \Vert g_0 \Vert_{L^\infty_{x, v}} +  \left \Vert \int_0^t V(t-s ) \{\langle v \rangle^{\beta-2} \Gamma_k (f, f) (s) \} ds  \right \Vert_{L^\infty_{x, v}}. 
\]
For the second  term, using  Duhamel's principle again we have
\begin{equation*}
\begin{aligned}
\int_0^t V(t-s ) \{\langle v \rangle^{\beta-2} \Gamma_k (f, f) (s) \} ds =& \int_0^t e^{-\nu(v)(t-s)} \{\langle v \rangle^{\beta-2} \Gamma_k (f, f) (s) \} ds 
\\
&+ \int_0^t \int_s^t e^{-\nu(v) (t-\tau) } K_{k+ \beta-2} \{V(\tau-s) (\langle v \rangle^{\beta-2} \Gamma_k (f, f) (s)) \}d\tau  ds  
\\
=& \int_0^t e^{-\nu(v)(t-s)} \{\langle v \rangle^{\beta-2} \Gamma_k (f, f) (s) \} ds 
\\
&+ \int_0^t \int_0^\tau e^{-\nu(v) (t-\tau) } K_{k+ \beta-2} \{V(\tau-s) (\langle v \rangle^{\beta-2} \Gamma_k (f, f) (s)) \}ds  d\tau  : =A_1+A_2 . 
\end{aligned}
\end{equation*}
For the term $A_1$, since $\beta - 2 \ge 1$, by Lemma \ref{L65} we have
\begin{equation*}
\begin{aligned}
|A_1| &\le C_{k, \beta} \int_0^t e^{-\nu(v) (t-s) } \nu (v) \Vert g(s) \Vert_{L^\infty_{x, v} }^{\frac {3p+1} {2p}} \sup_{y \in \T^3 } \left ( \int_{\R^3} |f(s, y, v')| dv' \right)^{\frac {p-1} {2p}}  ds
\\
&\le  C_{k, \beta} \int_0^t e^{-\nu(v) (t-s) } \nu (v)  e^{-\frac {\lambda_2} 4 s } \sup_{0 \le s \le t, y \in \T^3 } \left \{ \left[e^{\frac {\lambda_2} 4 s  }\Vert g(s)\Vert_{L^\infty_{x, v} } \right]  \Vert g(s) \Vert_{L^\infty_{x, v} }^{\frac {p+1} {2p}} \left ( \int_{\R^3} |f(s, y, v')| dv' \right)^{\frac {p-1} {2p}}  \right\} ds
\\
 &\le  C_{k, \beta}   e^{-\frac {\lambda_2} 4 t } \sup_{0 \le s \le t, y \in \T^3 } \left \{ \left[e^{\frac {\lambda_2} 4 s  }\Vert g(s)\Vert_{L^\infty_{x, v} } \right]  \Vert g(s) \Vert_{L^\infty_{x, v} }^{\frac {p+1} {2p}} \left ( \int_{\R^3} |f(s, y, v')| dv' \right)^{\frac {p-1} {2p}}  \right\}. 
\end{aligned}
\end{equation*}
For the $A_2$ term, since $\gamma \le 1$ and $\frac {p+1} {2p} \ge \frac 1 2$ we have
\[
\Vert \langle v \rangle^{\gamma} g (s )\Vert_{L^\infty_{x, v}} \le \Vert \langle v \rangle^2 g(s) \Vert_{L^\infty_{x, v}}^{\frac {p+1} {2p}} \Vert g(s) \Vert_{L^\infty_{x, v}}^{\frac {p-1} {2p}} = \Vert  h(s) \Vert_{L^\infty_{x, v}}^{\frac {p+1} {2p}} \Vert g(s) \Vert_{L^\infty_{x, v}}^{\frac {p-1} {2p}},
\]
together with  Lemma \ref{L65} we have
\begin{equation*}
\begin{aligned}
\Vert V(\tau-s) (\langle v \rangle^{\beta-2} \Gamma_k (f, f) (s))  \Vert_{L^\infty_{x, v}} \le&C  e^{-\frac {\lambda_2} 2 (\tau- s)} \Vert\langle v \rangle^{\beta-2} \Gamma_k (f, f) (s)  \Vert_{L^\infty_{x, v}}  
\\
\le& C  e^{-\frac {\lambda_2} 2 (\tau- s)}  \Vert \langle v \rangle^{\gamma} g(s)  \Vert_{L^\infty_{x, v}} \Vert g(s) \Vert_{L^\infty_{x, v}}^{\frac {p+1} {2p}} \sup_{y \in \T^3 }\left ( \int_{\R^3} |f(s, y, v')| dv' \right)^{\frac {p-1} {2p}}
\\
\le& C  e^{-\frac {\lambda_2} 2 (\tau- s)}   \Vert h(s) \Vert_{L^\infty_{x, v}}^{\frac {p+1} {2p}} \Vert g(s) \Vert_{L^\infty_{x, v}}     \sup_{y \in \T^3 }\left ( \int_{\R^3} |f(s, y, v')| dv' \right)^{\frac {p-1} {2p}} ,
\end{aligned}
\end{equation*}
which implies
\begin{equation*}
\begin{aligned}
|A_2| \le& \int_0^t \int_0^\tau e^{-\nu(v) (t - \tau ) } \int_{\R^3} |l_{k + \beta-2} (v, v')| dv' \Vert  V(\tau-s) (\langle v \rangle^{\beta-2} \Gamma(f, f) (s)) \Vert_{L^\infty_{x, v}}  ds     d\tau 
\\
\le&  C_{k, \beta} \int_0^t \int_0^\tau e^{-\nu(v) (  t  -  \tau ) } \nu (v)     e^{-\frac {\lambda_2} 2 (\tau- s)}   \Vert h(s) \Vert_{L^\infty_{x, v}}^{\frac {p+1} {2p}} \Vert g(s) \Vert_{L^\infty_{x, v}}     \sup_{y \in \T^3 }\left ( \int_{\R^3} |f(s, y, v')| dv' \right)^{\frac {p-1} {2p}}   ds  d \tau 
\\
\le&  C_{k, \beta} \int_0^t \int_0^\tau e^{-\nu(v) (  t  -  \tau ) } \nu (v)     e^{-\frac {\lambda_2} 2 (\tau- s)}  e^{-\frac {\lambda_2} 4 s} \sup_{0 \le s \le t, y \in \T^3 } \left \{ \left[e^{\frac {\lambda_2} 4 s} \Vert g(s) \Vert_{L^\infty_{x, v} } \right]  \Vert h(s) \Vert_{L^\infty_{x, v} }^{\frac {p+1} {2p}}     \left ( \int_{\R^3} |f(s, y, v')| dv' \right)^{\frac {p-1} {2p}}   \right\}   ds  d \tau 
\\
\le &  C_{k, \beta} e^{-\frac {\lambda_2} 4 t} \sup_{0 \le s \le t, y \in \T^3 } \left \{ \left[e^{\frac {\lambda_2} 4 s} \Vert g(s) \Vert_{L^\infty_{x, v} } \right]  \Vert h(s) \Vert_{L^\infty_{x, v} }^{\frac {p+1} {2p}}     \left ( \int_{\R^3} |f(s, y, v')| dv' \right)^{\frac {p-1} {2p}}   \right\}.
\end{aligned}
\end{equation*}
Gathering the terms and taking supremum we have
\begin{equation*}
\begin{aligned}
\sup_{0 \le s \le t} \left[e^{\frac {\lambda_2} 4 s} \Vert g(s) \Vert_{L^\infty_{x, v} } \right] \le&  C_{k, \beta} \Vert g_0 \Vert_{L^\infty_{x, v} } +  C_{k, \beta} \sup_{0 \le s \le t, y \in \T^3 } \left \{ \left[e^{\frac {\lambda_2} 4 s} \Vert g(s) \Vert_{L^\infty_{x, v} } \right]  \Vert h(s) \Vert_{L^\infty_{x, v} }^{\frac {p+1} {2p}}     \left ( \int_{\R^3} |f(s, y, v')| dv' \right)^{\frac {p-1} {2p}}   \right\}  
\\
\le&  C_{k, \beta} \Vert g_0 \Vert_{L^\infty_{x, v} } +  C_{k, \beta} \sup_{0 \le s \le t_1}  \Vert h(s) \Vert_{L^\infty_{x, v} }^2  + C_{k, \beta} \sup_{0 \le s \le t} \left[e^{\frac {\lambda_2} 4 s} \Vert g(s) \Vert_{L^\infty_{x, v} } \right]   
\\
& \times \sup_{t_1 \le s \le t, y \in \T^3} \left \{ \Vert h(s) \Vert_{L^\infty_{x, v} }^{\frac {p+1} {2p}}     \left ( \int_{\R^3} |f(s, y, v')| dv' \right)^{\frac {p-1} {2p}}   \right\}  
\\
\le&  C_{k, \beta}  M^4 + C_{k, \beta} \sup_{0 \le s \le t} \left[e^{\frac {\lambda_2} 4 s} \Vert g(s) \Vert_{L^\infty_{x, v} } \right]   \sup_{t_1 \le s \le t, y \in \T^3 } \left \{ \Vert h(s) \Vert_{L^\infty_{x, v} }^{\frac {p+1} {2p}}     \left ( \int_{\R^3} |f(s, y, v')| dv' \right)^{\frac {p-1} {2p}}   \right\},
\end{aligned}
\end{equation*}
together with  \eqref{small term large amplitude} we deduce
\[
e^{\frac {\lambda_2} 4 t} \Vert g(t) \Vert_{L^\infty_{x, v} } \le 2 C_{k, \beta} M^4, 
\]
so the lemma is thus proved. 
\end{proof}

\begin{proof} ({\bf Proof of Theorem \ref{T13}})
Finally we come to prove the rate of convergence for $h$. By \eqref{mild solution h} we have
\begin{equation*}
\begin{aligned}
|h(t, x, v)| \le & e^{-\nu(v) t} \Vert h_0\Vert_{L^\infty_{x, v}} + \int_0^t e^{-\nu(v) (t-s) } |(K_{k + \beta} h) (s, x-v(t-s), v) |ds 
\\ 
&+  \int_0^t e^{-\nu(v) (t-s) } |\Gamma_{k +\beta} (h, h) (s, x-v(t-s), v)| ds :=  e^{-\nu(v) t} \Vert h_0\Vert_{L^\infty_{x, v}} +J_2 +J_3. 
\end{aligned}
\end{equation*}
For the $J_{3}$ term by Lemma \ref{L65} we have
\begin{equation*}
\begin{aligned}
J_3\le & C_{k, \beta} \int_0^t e^{-\nu(v) (t-s)} \langle v \rangle^\gamma    \Vert h(s) \Vert_{L^\infty_{x, v}}  \Vert h(s) \Vert_{L^\infty_{x, v} }^{\frac {p+1} {2p}}\sup_{y \in \T^3 }  \left ( \int_{\R^3}  |f(s, y, v')| dv' \right)^{\frac {p-1} {2p}}  ds
\\
\le & C_{k, \beta}  e^{-\frac  {\lambda_2 } 4 t } \sup_{0 \le s \le t, y \in \T^3 } \left \{ \left[e^{\frac {\lambda_2} 4 s} \Vert h(s) \Vert_{L^\infty_{x, v} } \right]  \Vert h(s) \Vert_{L^\infty_{x, v} }^{\frac {p+1} {2p}}     \left ( \int_{\R^3} |f(s, y, v')| dv' \right)^{\frac {p-1} {2p}}   \right\}.
\end{aligned}
\end{equation*}
For the $J_2$ term, by  \eqref{mild solution h} again we have 
\begin{equation*}
\begin{aligned}
J_2 \le& \int_0^t e^{-\nu(v) (t-s) }\int_{\R^3}| l_{k +\beta} (v, v')|  e^{-\nu(v') s}  | h_0 (\tilde{x} -v' s , v') | dv' ds  
\\
&+  \int_0^t e^{-\nu(v) (t-s) }\int_{\R^3} | l_{k+ \beta} (v, v')| \int_0^s  e^{-\nu(v') (s-\tau )} | \Gamma_{k +\beta} (h, h)| (\tau, \tilde{x} -v' (s-\tau), v') d \tau dv' ds 
\\
&+  \int_0^t e^{-\nu(v) (t-s) }  \int_{\R^3}  \int_{\R^3} | l_{k +\beta} (v, v')  l_{k +\beta} (v', v'')  |  \int_0^s  e^{-\nu(v') (s-\tau )}   | h(\tau, \tilde{x} -v' (s-\tau), v'') | d v''d \tau dv' ds 
\\
&:=J_{21} +J_{22} +J_{23}.
\end{aligned}
\end{equation*}
For the $J_{21}$ term, similarly as  \eqref{new J21}
\[
J_{21} \le C_k e^{-\frac {\lambda_2} 2  t} \Vert  h_0 \Vert_{L^\infty_{x, v}} \int_0^t e^{-\frac {\nu(v)} 2 (t-s) } \langle v \rangle^\gamma ds \le C_k e^{-\frac {\lambda_2} 2  t}   \Vert h_0 \Vert_{L^\infty_{x, v}}.
\]
For the $J_{22}$ term by Lemma \ref{L65} we have
\begin{equation*}
\begin{aligned}
J_{22}\le & C_{k, \beta}    \int_0^t e^{-\nu(v) (t-s) }\int_{\R^3} | l_{k+\beta} (v, v')|  \int_0^s e^{-\nu(v') (s -\tau )} \langle v'  \rangle^\gamma e^{-\frac  {\lambda_2 } 4 \tau } 
\\
& \times \sup_{0 \le s \le t, y \in \T^3 } \left \{ \left[e^{\frac {\lambda_2} 4 s} \Vert h(s) \Vert_{L^\infty_{x, v} } \right]  \Vert h(s) \Vert_{L^\infty_{x, v} }^{\frac {p+1} {2p}}     \left ( \int_{\R^3} |f(s, y, v')| dv' \right)^{\frac {p-1} {2p}}   \right\}   d\tau dv'   ds
\\
\le &  C_{k, \beta} e^{-\frac {\lambda_2} 4 t}    \sup_{0 \le s \le t, y \in \T^3 } \left \{ \left[e^{\frac {\lambda_2} 4 s} \Vert h(s) \Vert_{L^\infty_{x, v} } \right]  \Vert h(s) \Vert_{L^\infty_{x, v} }^{\frac {p+1} {2p}}     \left ( \int_{\R^3} |f(s, y, v')| dv' \right)^{\frac {p-1} {2p}}   \right\}  .
\end{aligned}
\end{equation*}
For the term $J_{23}$, we split into two case $|v| \ge N$ and $|v| \le N$, for the case $|v| \ge N$, by \eqref{v ge N} we have
\begin{equation*}
\begin{aligned}
J_{23} \le&  \sup_{ 0 \le s\le t} \{ e^{\frac {\lambda_2} 4  s} \Vert h (s )  \Vert_{L^\infty_{x, v}} \}  \int_0^t e^{-\nu(v) (t-s) }  \int_{\R^3}  \int_{\R^3} | l_{k + \beta}  (v, v')  l_{k + \beta} (v', v'')  |  \int_0^s  e^{-\nu(v') (s-\tau )}  e^{-\frac {\lambda_2} 4 \tau} d v''d \tau dv' ds 
\\
\le &\left( \frac {c^2} {k^{\frac {\gamma+3} 2}}   + \frac {C_{k, \beta}}{N^2} \right) e^{-\frac {\lambda_2} 4 t}   \sup_{ 0 \le s \le t} \{ e^{\frac {\lambda_2} 4  s } \Vert h (s )  \Vert_{L^\infty_{x, v}} \}. 
\end{aligned}
\end{equation*}
For the case $|v| \le N$, by \eqref{l k N beta} and \eqref{l k beta decomposition} we split $J_{23}$ into three terms respectively. For the first term we have
\begin{equation*}
\begin{aligned}
&\int_0^t e^{-\nu(v) (t-s) }  \int_{\R^3}  \int_{\R^3} |( l_{k +\beta} (v, v') - l_{k, N, \beta}(v, v') ) l_{k +\beta} (v', v'')  |  \int_0^s  e^{-\nu(v') (s-\tau )}   | h(\tau, \tilde{x} -v' (s-\tau), v'') | d v''d \tau dv' ds 
\\
\le &\frac {C_{k, \beta}} {N}  \sup_{ 0 \le s \le t} \{ e^{\frac {\lambda_2} 4 s} \Vert h (s )  \Vert_{L^\infty_{x, v}} \}      \int_0^t e^{-\nu(v) (t-s) }  \langle v \rangle^\gamma ds \int_0^{s} e^{-\nu(v') (s-\tau) }  \langle v' \rangle^\gamma  e^{-\frac {\lambda_2} 4 \tau }   d\tau   \le e^{-\frac {\lambda_2} 4 t}   \frac {C_{k, \beta}} {N}  \sup_{ 0 \le s\le t} \{ e^{\frac {\lambda_2} 4  s} \Vert h (s )  \Vert_{L^\infty_{x, v}} \}   ,
\end{aligned}
\end{equation*}
the second term can be estimated similarly.  For the third term, we have $l_{k, N, \beta}(v', v'') l_{k, N, \beta}(v', v'')$ is supported in $|v| \le N,  |v'|\le C_{k, N, \beta}',  |v''|\le C_{k, N, \beta}',$ for some constant $C_{k, N, \beta}'>0$. By Lemma \ref{L72} we have
\[
\int_{|v'| \le C_{k, N, \beta}', |v''| \le C_{k, N, \beta}' } | h(\tau, \tilde{x} -v' (s-\tau), v'') | dv' dv''
\le  C_{k, N, \beta} \Vert g(\tau) \Vert_{L^\infty_{x, v} } \le C_{k, N, \beta} e^{-\frac {\lambda_2} 4 \tau} \Vert g_0 \Vert_{L^\infty_{x, v} }   \le C_{k, N, \beta}  e^{-\frac {\lambda_2} 4 \tau} \Vert h_0 \Vert_{L^\infty_{x, v} },
\]
which implies
\begin{equation*}
\begin{aligned}
&\int_0^t e^{-\nu(v) (t-s) }  \int_{\R^3}  \int_{\R^3} | l_{k, N, \beta} (v, v')  l_{k, N, \beta}(v', v'')  |  \int_{0}^{s}  e^{-\nu(v') (s-\tau )}   | h(\tau, \tilde{x} -v' (s-\tau, v'')) | d v''d \tau dv' ds 
\\
\le &C_{k, N, \beta} \int_0^t e^{-\lambda_2 (t-s) } \int_{0}^{s}  e^{-  \lambda_2 (s-\tau )}  \int_{|v'| \le C_{k, N, \beta}', |v''| \le C_{k, N, \beta}' }   | h  (\tau, \tilde{x} -v' (s-\tau), v'') | d v''d \tau dv' ds 
\\
\le &C_{k, N, \beta} \Vert h_0 \Vert_{L^\infty_{x, v}} \int_0^t e^{-\lambda_2 (t-s) } \int_{0}^{s} e^{-  \lambda_2 (s-\tau )}   e^{-\frac {\lambda_2} 4 \tau} d \tau ds  \le C_{k, N, \beta}  e^{-\frac {\lambda_2} 4 t} \Vert h_0 \Vert_{L^\infty_{x, v}}. 
\end{aligned}
\end{equation*}
Gathering all the terms we have
\begin{equation*}
\begin{aligned}
\Vert h(t) \Vert_{L^\infty_{x, v} }  \le & C_{k, \beta, N} e^{-\frac {\lambda_2} 4  t} \Vert h_0 \Vert_{L^\infty_{x, v} }  +  \left   ( \frac {c^2} {k^{\frac {\gamma+3} 2}}   + \frac {C_{k, \beta}} {N}  \right)    e^{-\frac {\lambda_2} 4 t}   \sup_{ 0 \le s \le t} \{ e^{\frac {\lambda_2} 4 s } \Vert h(s )  \Vert_{L^\infty_{x, v} } \}
\\
&  +C_{k, \beta} e^{-\frac {\lambda_2} 4 t}    \sup_{0 \le s \le t, y \in \T^3 } \left \{ \left[e^{\frac {\lambda_2} 4 s} \Vert h(s) \Vert_{L^\infty_{x, v} } \right]  \Vert h(s) \Vert_{L^\infty_{x, v} }^{\frac {p+1} {2p}}     \left ( \int_{\R^3} |f(s, y, v')| dv' \right)^{\frac {p-1} {2p}}   \right\}.
\end{aligned}
\end{equation*}
Taking suitable $k, \beta, N>0$ such that
\[
 \frac {c^2} {k^{\frac {\gamma+3} 2}}  + \frac {C_{k, \beta}} {N}   \le \frac 1 4,
\]
together with  \eqref{small term large amplitude} we conclude that
\[
e^{\frac {\lambda_2} 4 t} \Vert h(t) \Vert_{L^\infty_{x, v} } \le 4C_{k, \beta}  M^4, 
\]
for some constant $C_{k, \beta}>0$. The rate of convergence for $\gamma \in [0, 1]$ is thus proved. 
\end{proof}

\subsection{Convergence rate for the case $-3 < \gamma <0$}

Next we come to prove the convergence for the case $-3< \gamma <0$. Before going to the proof, we first prove some useful lemmas. 

\begin{lem}\label{L74}
For any $a, b, k >  0$ we have 
\[
\sup_{0 < x \le b } x^k e^{-a x } \le C_{b, k} (1+a)^{-k},
\]
for some constant $C_{b, k} > 0$. As a consequence we have for $-3< \gamma \le 0$, for all $t>0 , k \ge 0$
\[
\sup_{v \in \R^d} \{ e^{-\nu(v) t} \nu(v)^k \} \le C (1+t)^{-k}, 
\]
for some constant $C>0$. 
\end{lem}
\begin{proof} Take $f(x) =x^k e^{-a x }$, we have
\[
f'(x) = k x^{k-1} e^{-ax} -a x^k e^{-ax}  = x^{k-1}( k - ax ) e^{-ax },
\]
so $f'(x) =0$ when $x=\frac k a$. We easily deduce that
\[
\sup_{0 \le x \le b } f(x) = f(\frac k a) = \frac {k^k} {a^k} e^{-k}, \quad \hbox{if} \quad  \frac k a \le b, \quad \sup_{0 \le x \le b } f(x) = f(b)  = b^k e^{-ab} , \quad \hbox{if} \quad  \frac k a > b,
\]
so the first statement is thus proved, the second statement just from the fact that $\gamma \le 0$ implies that $0< \nu(v) \le C$ for some constant $C \ge 0$.  
\end{proof}

\begin{lem}\label{L75}
If $-3 <\gamma < 0$, then for any $0 < r  < 1$ we have
\[
\int_0^t e^{-\nu(v) (t-s) } \nu(v) (1+s)^{-r} ds \le C_r (1+t)^{-r}, \quad \forall v \in \R^d, \quad \forall t \ge 0,
\]
for some constant $C_r >0$ independent of $v$.
\end{lem}
\begin{proof}
We split the integral into two parts, $0 \le s \le \frac t 2 $ and $\frac t 2 \le s \le t$. For the first part we have $t-s \ge \frac t 2$, so by  Lemma \ref{L74} we have
\[
\int_0^{\frac t 2} e^{-\nu(v) (t-s) } \nu(v) (1+s)^{-r} ds \le C \int_0^{\frac t 2} (1+t-s)^{-1} (1+s)^{-r} ds \le C (1+t)^{-1} \int_0^{\frac t 2}  (1+s)^{-r} ds  \le C_r (1+t)^{-r}.
\]
For the second part we have $s \ge \frac t 2$, this time we have
\[
\int_{\frac t 2}^t e^{-\nu(v) (t-s) } \nu(v) (1+s)^{-r} ds \le C_r (1+t)^{-r}  \int_{\frac t 2}^t e^{-\nu(v) (t-s) } \nu(v) ds  \le  C_r (1+t)^{-r} (1-e^{-\nu(v)\frac t 2 })\le  C_r (1+t)^{-r},
\]
the proof is thus finished by gathering the two cases. 
\end{proof}
\begin{rmk}
If we directly use $(1+t-s)^{-1}$ instead of $e^{-\nu(v) (t-s) } \nu(v)$, we will have
\[
\int_0^t (1+t-s)^{-1} (1+s)^{-r} ds \le C_r (1+t)^{-r} \log(1+t),
\]
where an extra term $\log(1+t)$ occurs. 
\end{rmk}

\begin{lem}\label{L77} If $\gamma \in (-3, 0)$, for any $l \ge 2, k \ge 3 $ we have
\[
\Vert \langle  v \rangle^l  \Gamma_k (f, f) \Vert_{L^2}  \le C\Vert f\Vert_{  L^2} \Vert \langle v \rangle^l f\Vert_{ L^2 }^{1+\frac {\gamma} 3} \Vert \langle v \rangle^l   f \Vert_{ L^\infty}^{-\frac {\gamma} 3}  \le C \Vert f \Vert_{L^2}  \Vert \langle v \rangle^{2l } f \Vert_{L^\infty}^{\frac {p+1} {2p}} \Vert  f \Vert_{L^1}^{\frac {p-1} {2p}},
\]
where $p$ is defined in \eqref{requirement p}. 
\end{lem}
\begin{proof}
For the $\Gamma_k^+(f, f)$ term we prove by duality, for any smooth function $h$ we have
\begin{equation*}
\begin{aligned}
|( \langle  v \rangle^l  \Gamma^+_k (f, f) , h)| \le & \int_{\R^3} \int_{\R^3}\int_{\mathbb{S}^2} |v - v_* |^\gamma \frac {\langle v \rangle^k} 
{\langle v_*' \rangle^k \langle v' \rangle^k}  b(\cos \theta) |f(v_*')| | f(v')|  |h(v) | \langle v \rangle^{l}   dv dv_* d \sigma
\\
\le &\left(  \int_{\R^3} \int_{\R^3}\int_{\mathbb{S}^2} |v - v_* |^\gamma  b(\cos \theta) |f(v_*')|^2 | f(v')|^2   \langle v \rangle^{2l+\gamma}   dv dv_* d \sigma \right)^{1/2}
\\
& \left (\int_{\R^3} \int_{\R^3}\int_{\mathbb{S}^2} |v - v_* |^\gamma \frac {\langle v \rangle^{2k}} 
{\langle v_*' \rangle^{2k} \langle v' \rangle^{2k} }  b(\cos \theta)  |h(v) |^2  \langle v \rangle^{-\gamma}   dv dv_* d \sigma \right)^{1/2}. 
\end{aligned}
\end{equation*}
By Lemma \ref{L211} we have
\[
\int_{\R^3} \int_{\R^3}\int_{\mathbb{S}^2} |v - v_* |^\gamma \frac {\langle v \rangle^{2k}} 
{\langle v_*' \rangle^{2k} \langle v' \rangle^{2k} }  b(\cos \theta)  |h(v) |^2  \langle v \rangle^{-\gamma}   dv dv_* d \sigma  \le \Vert h \Vert_{L^2}.
\]
Since $2l + \gamma \ge 0$, by pre-post collisional change of variables we have
\begin{equation*}
\begin{aligned}
&\int_{\R^3} \int_{\R^3}\int_{\mathbb{S}^2} |v - v_* |^\gamma  b(\cos \theta) |f(v_*')|^2 | f(v')|^2   \langle v \rangle^{2l+\gamma}   dv dv_* d \sigma
\\
\le&C\int_{\R^3} \int_{\R^3}\int_{\mathbb{S}^2} |v - v_* |^\gamma  b(\cos \theta) |f(v_*)|^2 | f(v)|^2   \langle v_* \rangle^{2l+\gamma}   dv dv_* d \sigma 
\\
 & + C\int_{\R^3} \int_{\R^3}\int_{\mathbb{S}^2} |v - v_* |^\gamma  b(\cos \theta) |f(v_*)|^2 | f(v)|^2   \langle v \rangle^{2l+\gamma}   dv dv_* d \sigma
: =I_1 +I_2.
\end{aligned}
\end{equation*}
Without loss of generality we only compute $I_1$, by Lemma \ref{L25} we have
\begin{align}
\label{f f L2 L infty}
\nonumber
&\int_{\R^3} \int_{\R^3}\int_{\mathbb{S}^2} |v - v_* |^\gamma  b(\cos \theta) |f(v_*)|^2 | f(v)|^2   \langle v_* \rangle^{2l+\gamma}   dv dv_* d \sigma 
\\
\le & C\Vert f \Vert_{L^2}^2 \left( \sup_{v \in \R^3} \int_{\R^3} |v -v_*|^\gamma |f(v_*)|^2 \langle v_* \rangle^{2l+\gamma} dv_*  \right) \le C\Vert f \Vert_{L^2}^2  \Vert f^2 \Vert_{L^1_{2l}}^{1 + \frac \gamma 3} \Vert f^2 \Vert_{L^\infty_{2l}}^{ - \frac \gamma 3} \le C \Vert f \Vert_{L^2_{l}}^2  \Vert f \Vert_{L^2_l}^{2 +  \frac {2\gamma} 3} \Vert f \Vert_{L^\infty_l}^{ - \frac {2\gamma} 3} ,
\end{align}
the $\Gamma^{+}_k(f, f)$ term is proved by duality. For the $\Gamma_k^-(f, f)$ term, we easily compute
\begin{equation*}
\begin{aligned}
(\langle v \rangle^{l} \Gamma_k^- (f, f), h) =& \int_{\R^3}  \int_{\R^3} \int_{\mathbb{S}^2} |v-v_*|^\gamma b(\cos \theta) f(v_*) \langle v_*  \rangle^{-k } \langle v \rangle^l f(v) h(v) d v dv_*  d \sigma
\\
\le &\left(  \int_{\R^3} \int_{\R^3}\int_{\mathbb{S}^2} |v - v_* |^\gamma  b(\cos \theta) |f(v_*)|^2 | f(v)|^2   \langle v \rangle^{2l+\gamma}   dv dv_* d \sigma \right)^{1/2}
\\
& \left (\int_{\R^3} \int_{\R^3}\int_{\mathbb{S}^2} |v - v_* |^\gamma \langle v_* \rangle^{-2k} b(\cos \theta)  |h(v) |^2  \langle v \rangle^{-\gamma}   dv dv_* d \sigma \right)^{1/2}. 
\end{aligned}
\end{equation*}
The first term is the same as \eqref{f f L2 L infty}, for the second term we have
\[
\int_{\R^3} \int_{\R^3}\int_{\mathbb{S}^2} |v - v_* |^\gamma \langle v_* \rangle^{-2k} b(\cos \theta)  |h(v) |^2  \langle v \rangle^{-\gamma}   dv dv_* d \sigma  \le \Vert h \Vert_{L^2},
\]
so the first inequality is thus proved. For the last inequality by  \eqref{requirement p} we have $\frac {p-1} {2p} \le \frac 1 2 + \frac {\gamma} 6$, together with $l \ge 2$ we have
\[
\Vert \langle v \rangle^{l} f\Vert_{  L^2}^{1+\frac {\gamma} 3}  \le C \Vert f  \Vert_{L^1}^{\frac 1 2+\frac {\gamma} 6}  \Vert \langle v \rangle^{2l} f\Vert_{L^\infty}^{\frac 1 2+\frac {\gamma} 6}  \le C \Vert  f \Vert_{L^1}^{\frac {p-1} {2p}} \Vert \langle v \rangle^{2l} f \Vert_{L^\infty}^{1 + \frac \gamma 3 - \frac {p-1} {2p}},
\]
the lemma is thus proved. 
\end{proof}

We first prove that the converge in $L^2_x L^2_v$, in fact the $L^2_x L^2_v$ convergence for the linearized semigroup is proved in Lemma \ref{L312}, so we only need to prove the convergence for the nonlinear equation.

\begin{lem}\label{L78}
Suppose $f$ the solution to \eqref{Boltzmann equation f}, then there exists $k_0 >\max\{3, 3+\gamma \}$ such that if $ k \ge k_0$ large, $\beta  \ge 6$ then we have
\[
\Vert f(t)\Vert_{L^2_x L^2_v} \le 4 C_{k, \beta} M^4(1+t)^{-r_1} ,\quad \forall t \ge 0, 
\]
for some constants $C_{k, \beta}, r_1 > 1$. 
\end{lem}
\begin{proof}
Denote the solution to the linearized equation
\[
\partial_t \xi + v \cdot \nabla_x \xi + \nu(v) \xi = K_{k} \xi, \quad \xi(0, x, v) = \xi_0(x, v),
\]
by $\xi(t) =V(t) \xi_0$. By Lemma \ref{L312} we have
\[
\Vert V(t) \xi \Vert_{L^2_x L^2_v} \lesssim  \langle t \rangle^{ - \frac {l} {|\gamma|}}  \Vert \xi \Vert_{L^2_x L^2_3}, \quad \forall  l \in (0, 3).
\]
By Duhamel's principle we have
\[
f(t) = V(t) f_0 +\int_0^t V(t-s) \{ \Gamma_k (f, f)(s)  \}ds,
\]
which implies
\[
\Vert f(t) \Vert_{L^2_xL^2_v } \le C(1+t)^{-r_1} \Vert \langle  v \rangle^3 f_0 \Vert_{L^2_x  L^2_v} + C \int_0^t (1+t-s)^{-r_1} \Vert \langle  v \rangle^3  \Gamma_k (f, f)(s) \Vert_{L^2_xL^2_v}  ds ,
\]
for some constant $ r_1>1$. By Lemma \ref{L77} we have
\[
\Vert \langle  v \rangle^3  \Gamma_k (f, f)(s) \Vert_{L^2_xL^2_v}  \le C\Vert f(s)\Vert_{L^2_x L^2_v} \Vert \langle v \rangle^{6} f(s) \Vert_{L^\infty_{x, v}}^{\frac {p+1} {2p}} \sup_{y \in \T^3}  \left (\int_{\R^3} |f(s, y, v')| dv' \right)^{\frac {p-1} {2p}},
\]
which implies
\begin{equation*}
\begin{aligned}
&\int_0^t (1+t-s)^{-r_1} \Vert \langle  v \rangle^3  \Gamma_k (f, f)(s) \Vert_{L^2_x L^2_v}  ds 
\\
\le&C \int_0^t (1+t-s)^{-r_1} (1+s)^{-r_1}   \sup_{0 \le s \le t, y \in \T^3 } \left \{ [(1+s)^{r_1} \Vert f(s)   \Vert_{L^2_x L^2_v}  ]\Vert \langle v \rangle^{6} f(s) \Vert_{L^\infty_{x, v}  }^{\frac {p+1} {2p}}  \left (\int_{\R^3} |f(s, y, v')| dv' \right)^{\frac {p-1} {2p}} \right \} ds 
\\
\le&C (1+t)^{-r_1}  \sup_{0 \le s \le t, y \in \T^3 } \left \{ [(1+s)^{r_1} \Vert f(s)   \Vert_{L^2_x L^2_v}  ]\Vert \langle v \rangle^{6} f(s) \Vert_{L^\infty_{x, v}  }^{\frac {p+1} {2p}}  \left (\int_{\R^3} |f(s, y, v')| dv' \right)^{\frac {p-1} {2p}} \right \}.
\end{aligned}
\end{equation*}
Gathering  the two terms together we easily have
\begin{equation*}
\begin{aligned}
\sup_{0 \le s \le t}   [(1+s)^{r_1} \Vert f(s)   \Vert_{L^2_xL^2_v}  ] \le& C_{k} \Vert \langle v \rangle^{3}  f_0\Vert_{L^2_xL^2_v} 
\\
&+ C_{k}  \sup_{0 \le s \le t, y \in \T^3 }  \left \{ [(1+s)^{r_1} \Vert f(s)   \Vert_{L^2_xL^2_v}  ] \Vert \langle v \rangle^{6} f(s) \Vert_{L^\infty_{x, v}}^{\frac {p+1} {2p}}  \left (\int_{\R^3} |f(s, y, v')| dv' \right)^{\frac {p-1} {2p}} \right \}
\\
\le &C_{k, \beta} M^4 + C_{k}   \sup_{0 \le s \le t }   [(1+s)^{r_1} \Vert f(s)   \Vert_{L^2_x L^2_v}  ]  
\\
& \times \sup_{t_1 \le s \le t, y \in \T^3 } \left\{  \Vert \langle v \rangle^{6} f(s) \Vert_{L^\infty_{x, v} }^{\frac {p+1} {2p}}  \left (\int_{\R^3} |f(s, y, v')| dv' \right)^{\frac {p-1} {2p}} \right \}.
\end{aligned}
\end{equation*}
If $\beta \ge 6$, by \eqref{small term large amplitude} we deduce 
\[
\Vert f(t)\Vert_{L^2_xL^2_v} \le 2 C_{k, \beta} M^4(1+t)^{-r_1} ,\quad \forall t \ge 0, 
\]
the proof is thus finished. 
\end{proof}

\begin{proof}({\bf Proof of Theorem  \ref{T13}} ) Recall \eqref{mild solution}, the corresponding mild solution is 
\begin{equation*}
\begin{aligned}
|f(t, x, v)| \le & e^{-\nu(v) t} f_0(x-vt, v)  + \int_0^t e^{-\nu(v) (t-s) } |(K_{k} f) (s, x-v(t-s), v) |ds 
\\ 
&+  \int_0^t e^{-\nu(v) (t-s) } |\Gamma_{k} (f, f) (s, x-v(t-s), v)| ds := J_1 +J_2 +J_3.
\end{aligned}
\end{equation*}
 For any $ r \in (0, 1)$ fixed, for the $J_1$ term by Lemma \ref{L74} we have
\[
|J_1| \le e^{-\nu (v) t}  \nu(v)^r  |\nu(v)^{-r} f_0 (x-vt, v)| \le C (1+t)^{-r} \Vert \nu^{-r} f_0 \Vert_{L^\infty_{x, v} } \le C M (1+t)^{-r}. 
\]
For the $J_3$ term, by Lemma \ref{L65} we have
\begin{equation*}
\begin{aligned}
J_3 \le & C_{k} \int_0^t e^{-\nu(v) (t-s)} \langle v \rangle^\gamma \Vert f(s)\Vert_{L^\infty} \Vert h(s) \Vert_{L^\infty_{x, v} }^{\frac {p+1} {2p}}    \sup_{ y \in \T^3 } \left ( \int_{\R^3} |f(s, y, v')| dv' \right)^{\frac {p-1} {2p}}  ds  
\\
\le & C_{k} \int_0^t e^{-\nu(v) (t-s)} \langle v \rangle^\gamma (1+s)^{-r} ds  \sup_{0 \le s \le t, y \in \T^3 } \left \{ [ (1+s)^{r} \Vert f(s) \Vert_{L^\infty_{x, v}  } ]\Vert h(s) \Vert_{L^\infty_{x, v}  }^{\frac {p+1} {2p}}  \left (   \int_{\R^3}  |f(s, y, v')| dv' \right)^{\frac {p-1} {2p}} \right\}
\\
\le & C_{k} (1+t)^{-r}  \sup_{0 \le s \le t, y \in \T^3 } \left \{ [ (1+s)^{r} \Vert f(s) \Vert_{L^\infty_{x, v}  } ]\Vert h(s) \Vert_{L^\infty_{x, v}  }^{\frac {p+1} {2p}}  \left (   \int_{\R^3}  |f(s, y, v')| dv' \right)^{\frac {p-1} {2p}} \right\}.
\end{aligned}
\end{equation*}
For the $J_2$ term, denote $\tilde{x} = x - v(t-s) $ we have
\[
J_2 \le \int_0^t e^{-\nu(v) (t-s) }\int_{\R^3}| l_{k}  (v, v') f(s, \tilde{x}, v')| dv' ds. 
\] 
by  \eqref{mild solution}  again we have 
\begin{equation*}
\begin{aligned}
J_2 \le& \int_0^t e^{-\nu(v) (t-s) }\int_{\R^3}| l_{k} (v, v')|  e^{-\nu(v') s}  | f_0 (\tilde{x} -v' s , v') | dv' ds  
\\
&+  \int_0^t e^{-\nu(v) (t-s) }\int_{\R^3} | l_{k}  (v, v')| \int_0^s  e^{-\nu(v') (s-\tau )} | \Gamma_{k} (f, f)| (\tau, \tilde{x} -v' (s-\tau), v') d \tau dv' ds 
\\
&+  \int_0^t e^{-\nu(v) (t-s) }  \int_{\R^3}  \int_{\R^3} | l_{k} (v, v')  l_{k} (v', v'')  |  \int_0^s  e^{-\nu(v') (s-\tau )}   | f(\tau, \tilde{x} -v' (s-\tau), v'') | d v''d \tau dv' ds  := J_{21} +J_{22} +J_{23}.
\end{aligned}
\end{equation*}
For the $J_{21}$ term by Lemma \ref{L74} and Lemma \ref{L75} we have
\begin{equation*}
\begin{aligned}
J_{21} \le& \int_0^t e^{-\nu(v) (t-s) }\int_{\R^3}| l_k (v, v')|  e^{-\nu(v') s} \nu(v')^r  |  \nu(v')^{-r} f_0 (\tilde{x} -v' s , v')| dv' ds 
\\
\le&C \Vert \nu^{-r} f_0 \Vert_{L^\infty_{x, v}  } \int_0^t e^{-\nu(v) (t-s) }\int_{\R^3}| l_k (v, v')| (1+s)^{-r} dv' ds 
\\
\le&C_{k} \Vert \nu^{-r} f_0 \Vert_{L^\infty_{x, v}  } \int_0^t e^{-\nu(v) (t-s) } \langle v \rangle^{\gamma} (1+s)^{-r} dv' ds  
\\
\le&  C_{k} \Vert \nu^{-r} f_0 \Vert_{L^\infty_{x, v}  }  (1+t)^{-r} \le   C_{k} M (1+t)^{-r} . 
\end{aligned}
\end{equation*}
For the $J_{22}$ term, by Lemma \ref{L65} we have
\begin{equation*}
\begin{aligned}
J_{22} \le  &\int_0^t e^{-\nu(v) (t-s) }\int_{\R^3} | l_{k} (v, v')| \int_0^s  e^{-\nu(v') (s-\tau )} | \Gamma_{k} (f, f)| (\tau, \tilde{x} -v' (s-\tau), v') d \tau dv' ds
\\
\le & C_{k}    \int_0^t e^{-\nu(v) (t-s) }\int_{\R^3} | l_{k}  (v, v')|  \int_0^s e^{-\nu(v') (s -\tau )} \langle v'  \rangle^\gamma (1+\tau)^{-r} d\tau dv' ds 
\\
&\sup_{0 \le s \le t, y \in \T^3 } \left \{ \left[ ( 1+s)^{r} \Vert f(s) \Vert_{L^\infty_{x, v}  } \right]  \Vert h(s) \Vert_{L^\infty_{x, v}  }^{\frac {p+1} {2p}}     \left ( \int_{\R^3} |f(s, y, v')| dv' \right)^{\frac {p-1} {2p}}   \right\}     
\\
\le &  C_{k} (1+t)^{-r}   \sup_{0 \le s \le t, y \in \T^3 } \left \{ \left[(1+s)^{r} \Vert f(s) \Vert_{L^\infty_{x, v}  } \right]  \Vert h(s) \Vert_{L^\infty_{x, v}  }^{\frac {p+1} {2p}}     \left ( \int_{\R^3} |f(s, y, v')| dv' \right)^{\frac {p-1} {2p}}   \right\}. 
\end{aligned}
\end{equation*}
For the $J_{23}$ term,  we again split it into two parts $|v| \le N$ and $|v| \ge N$. For the case $|v| \ge N$ we have $\langle v \rangle^{-2} \le \frac 1 {N^2}$, by Lemma \ref{L63} and Lemma \ref{L75} we have
\begin{equation*}
\begin{aligned}
& \int_{\R^3}  |l_{k}  (v, v')|  \int_0^s  e^{-\nu(v') (s-\tau )}  \int_{\R^3} |l_{k}  (v', v'')|  | f(\tau, \tilde{x} -v' (s-\tau), v'') | dv'' d\tau 
\\
\le& \sup_{0 \le s \le t} \{ (1 + s)^r \Vert f(s) \Vert_{L^\infty_{x, v} } \}  \int_{\R^3}  |l_{k}  (v, v')|  \int_0^s  e^{-\nu(v') (s-\tau )}  \int_{\R^3} |l_{k} (v', v'')|  (1 + \tau )^{-r} dv'' d\tau 
\\
\le& \sup_{0 \le s \le t}  \{ (1+s )^r \Vert f(s) \Vert_{L^\infty_{x, v}  } \}  \int_{\R^3}  |l_{k}  (v, v')| \left( \frac {c} {k^{\frac {\gamma+3} 4}} + C_{k}  \langle v'  \rangle^{-2}  \right) \int_0^s  e^{-\nu(v') (s-\tau )}  \langle v' \rangle^{\gamma} (1 + \tau )^{-r} dv'' d\tau  dv'
\\
\le& \langle v \rangle^\gamma \left( \frac {c^2} {k^{\frac {\gamma+3} 2}}   + \frac {C_{k}  }{N^2} \right)  (1+s)^{-r} \sup_{0 \le s \le t}  \{ (1+s )^r \Vert f(s) \Vert_{L^\infty_{x, v}  } \}, 
\end{aligned}
\end{equation*}
which implies 
\begin{equation*}
\begin{aligned}
J_{23} \le& \left( \frac {c^2} {k^{\frac {\gamma+3} 2}}   + \frac {C_{k}}{N^2} \right) \sup_{0 \le s \le t}  \{ (1+s)^r \Vert f(s) \Vert_{L^\infty_{x, v}  } \}   \int_0^t e^{-\nu(v) (t-s) }  \langle v \rangle^{\gamma} (1+s)^{-r}  ds 
\\
\le &\left( \frac {c^2} {k^{\frac {\gamma+3} 2}}   + \frac {C_{k}}{N^2} \right) (1+t)^{-r}   \sup_{0 \le s \le t}  \{ (1+s)^r \Vert f (s) \Vert_{L^\infty_{x, v}  } \}. 
\end{aligned}
\end{equation*}
For the case $|v| \le N$, by the decomposition \eqref{l k decomposition}, we split it into three terms respectively. 
For the first term by Lemma \ref{L75} we have
\begin{equation*}
\begin{aligned}
&\int_0^t e^{-\nu(v) (t-s) }  \int_{\R^3}  \int_{\R^3} |( l_{k} (v, v') - l_{k, N}(v, v') ) l_{k } (v', v'')  |  \int_0^s  e^{-\nu(v') (s-\tau )}   | f(\tau, \tilde{x} -v' (s-\tau), v'') | d v''d \tau dv' ds 
\\
\le &\frac {C_{k}} {N} \sup_{0 \le s \le t}  \{ (1+s)^r \Vert f(s) \Vert_{L^\infty_{x, v}  } \}    \int_0^t e^{-  \nu(v)  (t-s) }  \langle v \rangle^\gamma ds \int_0^{s} e^{-    \nu(v') (s-\tau) }  \langle v' \rangle^\gamma  (1+\tau)^{-r}  d\tau  
\\
\le & \frac {C_{k}} {N}   (1+t)^{-r}  \sup_{0 \le s\le t}  \{ (1+s)^r \Vert f(s) \Vert_{L^\infty_{x, v}  } \},    
\end{aligned}
\end{equation*}
the second term can be estimated similarly. For the last term,  since $l_{k, N}(v', v'') l_{k, N}(v', v'')$ is supported in 
$|v| \le N,  |v'|\le C_{k, N}', |v''|\le C_{k, N}'$ for some constant $C_{k, N}'>0$. We again split it into two cases, $\tau \in [s-\lambda, s]$ and $\tau \in [0, s-\lambda ]$,  where $\lambda>0$ is a small constant to be fixed later. For the case $\tau \in [s-\lambda, s]$, we first prove that
\begin{equation}
\label{lambda 1 + t r}
\int_{s-\lambda}^{s} e^{-\nu(v') (s-\tau) }  \langle v' \rangle^\gamma (1+\tau)^{-r} d\tau \lesssim  C_r \lambda(1 + s)^{-r}.
\end{equation}
Similarly as Lemma \ref{L75}, we split the integral into two cases, $0 \le \tau \le \frac s 2 $ and $\frac s 2 \le \tau \le s$. For the first case we have $s - \tau \ge \frac s 2$, so by  Lemma \ref{L74} we have
\[
\int_{s-\lambda}^{s} e^{-\nu(v') (s-\tau) }  \langle v' \rangle^\gamma (1+ \tau)^{-r} d\tau \le C\int_{s-\lambda}^{s}  (1+s -\tau )^{-1} (1+\tau)^{-r} d \tau \le C (1+s)^{-1} \int_{s-\lambda}^{s} d\tau  \le C_r  \lambda ( 1+s)^{-r}.
\]
For the second case we have $\tau  \ge \frac s 2$, since $|v'| \le C_{k, N}'$, this time we have
\[
\int_{s-\lambda}^{s} e^{-\nu(v') (s-\tau) }  \langle v' \rangle^\gamma (1+ \tau)^{-r} d\tau \le  C_r (1 + s )^{-r}  \int_{s-\lambda}^{s} e^{-\nu(v') (s-\tau) }  \langle v' \rangle^\gamma d\tau  \le  C_r (1+s )^{-r} (1-e^{-\nu(v') \lambda })\le  C_r \lambda (1+s )^{-r},
\]
so \eqref{lambda 1 + t r} is proved by gathering the two cases.   Then we have 
\begin{equation*}
\begin{aligned}
&\int_0^t e^{-\nu(v) (t-s) }  \int_{\R^3}  \int_{\R^3} | l_{k, N} (v, v')  l_{k, N}(v', v'')  |  \int_{s-\lambda}^s  e^{-\nu(v') (s-\tau )}   | f (\tau, \tilde{x} -v' (s-\tau), v'') | d v''d \tau dv' ds 
\\
\le &C_{k, N}  \sup_{ 0 \le \tau \le t} \{  (1+\tau)^{r} \Vert f  (\tau )  \Vert_{L^\infty_{x, v} } \}  \int_0^t e^{-\nu(v) (t-s) }  \langle v \rangle^\gamma ds \int_{s-\lambda}^{s} e^{-\nu(v') (s-\tau) }  \langle v' \rangle^\gamma (1+\tau)^{-r} d\tau 
\\
\le &C_{k, N} \lambda  \sup_{ 0 \le \tau \le t} \{ (1+\tau)^{r}  \Vert f(\tau )  \Vert_{L^\infty_{x, v} } \}  \int_0^t e^{-\nu(v) (t-s) }  \langle v \rangle^\gamma (1 + s)^{-r} ds 
\\
\le &  C_{k, N} \lambda   (1 + t)^{-r}   \sup_{ 0 \le \tau \le t} \{  (1 + \tau )^{r} \Vert f(\tau )  \Vert_{L^\infty_{x, v} } \}. 
\end{aligned}
\end{equation*}
For the case $\tau \in [0, s-\lambda]$, by \eqref{l k N bound} and  Lemma \ref{L78} we have
\begin{equation*}
\begin{aligned}
&\int_{|v'| \le C_{k, N}', |v''| \le C_{k, N}' } | f(\tau, \tilde{x} -v' (s-\tau), v'') | dv' dv''
\\
\le & C_{k, N} \frac 1 {(s-\tau)^\frac 3 2}  \left( \int_{\T^3} \int_{|v''| \le C_{k, N}' }  |f (\tau, y, v'') |^2   dv'' dy \right)^{\frac 1 2}   \le C_{k, N} \frac 1 {(s-\tau)^\frac 3 2}  \Vert f(\tau) \Vert_{L^2_xL^2_v}   \le C_{k, N, \beta} \lambda^{-\frac 3 2} M^4 (1+\tau)^{-r} ,
\end{aligned}
\end{equation*}
where we have made a change of variable $ y = \tilde{x} -v' (s-\tau) $. So we have
\begin{equation*}
\begin{aligned}
&\int_0^t e^{-\nu(v) (t-s) }  \int_{\R^3}  \int_{\R^3} | l_{k, N} (v, v')  l_{k, N}(v', v'')  |  \int_{0}^{s-\lambda}  e^{-\nu(v') (s-\tau )}   | f (\tau, \tilde{x} -v' (s-\tau, v'')) | d v''d \tau dv' ds 
\\
\le &C_{k, N} \int_0^t e^{-\lambda_2 (t-s) } \int_{0}^{s}  e^{-  \lambda_2 (s-\tau )}  \int_{|v'| \le C_{k, N}', |v''| \le C_{k, N}' }   | f  (\tau, \tilde{x} -v' (s-\tau), v'') | d v''d \tau dv' ds 
\\
\le &C_{k, N, \beta} \lambda^{-\frac 3 2}   M^4  \int_0^t e^{-\lambda_2 (t-s) } \int_{0}^{s} e^{-  \lambda_2 (s-\tau )}   (1+\tau)^{-r}   d \tau ds  \le C_{k, N, \beta} \lambda^{-\frac 3 2}  M^4   (  1 + t )^{-r}  .
\end{aligned}
\end{equation*}
Gathering all the terms we have
\begin{equation*}
\begin{aligned}
 \Vert f(t) \Vert_{L^\infty_{x, v} }  \le & C_{k, \beta, N} \lambda^{-\frac 3 2}  M^4 (  1 + t )^{-r}  +   \left( \frac {c^2} {k^{\frac {\gamma+3} 2}}   + \frac {C_{k}} {N}   +C_{k, N} \lambda \right)     (1+t)^{-r}  \sup_{0 \le s \le t}  \{ (1 + s)^r \Vert f (s) 
 \Vert_{L^\infty_{x, v}  } \}     
 \\
 &+  C_{k} (1+t)^{-r}   \sup_{0 \le s \le t, y \in \T^3 } \left \{ \left[(1+s)^{r} \Vert f(s) \Vert_{L^\infty_{x, v}  } \right]  \Vert h(s) \Vert_{L^\infty_{x, v}  }^{\frac {p+1} {2p}}     \left ( \int_{\R^3} |f(s, y, v')| dv' \right)^{\frac {p-1} {2p}}   \right\}
\\
\le & C_{k, \beta, N} \lambda^{-\frac 3 2}  M^4 (  1 + t )^{-r}  +   \left( \frac {c^2} {k^{\frac {\gamma+3} 2}}   + \frac {C_{k}} {N} + C_{k, N} \lambda \right)     (1+t)^{-r}  \sup_{0 \le s \le t}  \{ (1+s)^r \Vert f (s) 
 \Vert_{L^\infty_{x, v}  } \}     
\\
 &+  C_{k} (1+t)^{-r}    \sup_{0 \le s \le t }   [(1+s)^{r} \Vert f(s)   \Vert_{L^\infty_{x, v}}  ]  
 \sup_{t_1 \le s \le t, y \in \T^3 } \left\{  \Vert h(s) \Vert_{L^\infty_{x, v} }^{\frac {p+1} {2p}}  \left (\int_{\R^3} |f(s, y, v')| dv' \right)^{\frac {p-1} {2p}} \right \}.
\end{aligned}
\end{equation*}
Taking suitable $k, N, \lambda$ such that 
\[
 \frac {c^2} {k^{\frac {\gamma+3} 2}}   + \frac {C_{k}} {N} +C_{k, N} \lambda  \le \frac 1 4,
\]
together with \eqref{small term large amplitude} we conclude
\[
 \Vert f(t) \Vert_{L^\infty_{x, v}} \le 4C_{k, \beta} M^4 (1+t)^{-r}, \quad \forall t \ge 0.
\]
The rate of convergence is thus proved for the  case $-3 < \gamma < 0$. 
\end{proof}

\bigskip{\bf Acknowledgments.} The author would thanks to Lingbing He and Yong Wang for fruitful talks on the paper. The author is supported by grants from Beijing Institute of Mathematical Sciences and Applications and Yau Mathematical Science Center, Tsinghua University. \\

\bigskip{\bf Declarations of interest: } None.

\end{document}